 \def\versionno{ fgjs   --   version 5.0   --   by jf   --  23.06.23}

\documentclass[english,12pt]{article}

\usepackage{babel}
\usepackage{float}
\usepackage{mathtools}
\usepackage{amsmath}
\usepackage{amsthm}
\usepackage{amssymb}
\usepackage{enumerate}
\usepackage{multicol}
\usepackage{stmaryrd}

\usepackage{bbm}
\usepackage{color,colordvi}
\usepackage[mathscr]{eucal}

\numberwithin{equation}{section}

\makeatletter \@ifundefined{date}{}{\date{}} \makeatother

\theoremstyle{plain}
\newtheorem{thm}{Theorem}[section]
\newtheorem{cor}[thm]{Corollary}
\newtheorem{lem}[thm]{Lemma}
\newtheorem{prop}[thm]{Proposition}

\newtheorem*{thmx}{{\Color Corollary} \ref{thm:Radford}}
\newtheorem*{thmya}{Theorem \ref{thm:pivmodule-pivMorita}}
\newtheorem*{thmyb}{Theorem \ref{bicategories_pivotal_modules}}
\newtheorem*{thmyc}{Theorem \ref{thm:A=B-ZA=ZB}}
\newtheorem*{cors}{Corollary \ref{spherical_pme_invariant}}

\theoremstyle{definition}

\newtheorem{defi}[thm]{Definition}
\newtheorem{rem}[thm]{Remark}

\theoremstyle{definition}

\newtheorem*{rep@lemma}{\rep@title}
\newcommand{\newreplemma}[2]{%
\newenvironment{rep#1}[1]{%
 \def\rep@title{#2 \ref{##1}}%
 \begin{rep@lemma}}%
 {\end{rep@lemma}}}
\newreplemma{lem}{Lemma}

\makeatletter
\newcommand*{\relrelbarsep}{.386ex}
\newcommand*{\relrelbar}{%
  \mathrel{%
      \mathpalette\@relrelbar\relrelbarsep}}
\newcommand*{\@relrelbar}[2]{%
  \raise#2\hbox to 0pt{$\m@th#1\relbar$\hss}%
    \lower#2\hbox{$\m@th#1\relbar$}}
\providecommand*{\rightrightarrowsfill@}{%
      \arrowfill@\relrelbar\relrelbar\rightrightarrows}
\providecommand*{\leftleftarrowsfill@}{%
       \arrowfill@\leftleftarrows\relrelbar\relrelbar}
\providecommand*{\xrightrightarrows}[2][]{%
          \ext@arrow 0359\rightrightarrowsfill@{#1}{#2}}
\providecommand*{\xleftleftarrows}[2][]{%
    \ext@arrow 3095\leftleftarrowsfill@{#1}{#2}}
\makeatother

\usepackage{tikz}
\usepackage{tikz-cd}
\everymath{\displaystyle}
\usepackage{lmodern}
\usepackage{fancyhdr}
\usetikzlibrary{decorations,arrows}
\usetikzlibrary{decorations.markings}
\usetikzlibrary{decorations.pathmorphing}
\usetikzlibrary{matrix,arrows}

             \newcommand\Cite[2] {\cite[#1]{#2}}

\catcode`\@=11
\newif\if@fewtab\@fewtabtrue
{\count255=\time\divide\count255 by 60
\xdef\hourmin{\number\count255}
\multiply\count255 by-60\advance\count255 by\time
\xdef\hourmin{\hourmin:\ifnum\count255<10 0\fi\the\count255}}
\def\ps@draft{\let\@mkboth\@gobbletwo
    \def\@oddfoot{\hbox to 7 cm{\tiny \versionno
       \hfil}\hskip -7cm\hfil\rm\thepage \hfil {\tiny\draftdate}}
    \def\@oddhead{}
    \def\@evenhead{}\let\@evenfoot\@oddfoot}
\def\draftdate{\number\month/\number\day/\number\year\ \ \ \hourmin }

\catcode`\@=12

\def\Color          {\color{black}}

\def\Act           {{\triangleright}}   
\def\act           {\,{\Act}\,}         
\def\Actr          {{\triangleleft}}    
\def\actr          {\,{\Actr}\,}        
\def\actF          {\,{\Yright}\,}
\def\ActrF         {{\Yleft}} 
\def\actrF         {\,\ActrF\,} 
\def\be            {\begin{equation}}
\def\bearl         {\begin{array}{l}}
\def\bearll        {\begin{array}{ll}}

\def\boti          {\,{\boxtimes}\,}

\def\cala          {{\mathcal A}}
\def\Cala          {{\!\mathcal A}}
\def\calb          {{\mathcal B}}
\def\calc          {{\mathcal C}}
\def\Calc          {{\!\mathcal C}}
\def\cald          {{\mathcal D}}
\def\Cald          {{\!\mathcal D}}
\def\calm          {{\mathcal M}}
\def\Calm          {{\!\mathcal M}}
\def\caln          {{\mathcal N}}
\def\Caln          {{\!\mathcal N}}
\def\call          {{\mathcal L}}
\def\calx          {{\mathcal X}}
\def\calz          {{\mathcal Z}}
\def\cir           {\,{\circ}\,}
\def\Colon         {:\quad}

\def\dualcat       {{\cala^*_\calm}}                 
\def\Dualcat       {{\!\cala^*_\calm}}               
\def\DD            {\mathbb{D}}
\def\DDA           {\mathbb{D}_{\Cala}}
\def\dd            {^{\vee\vee}}
\def\ldd           {{}^{\vee\vee\!}}

\def\ee            {\end{equation}}
\def\eear          {\end{array}}

\def\Enumerate     {\def\leftmargini{1.34em}~\\[-1.42em]\begin{enumerate}}
\def\Enumeratei    {
	~\\[-1.42em]\begin{enumerate}[{\rm (i)}]\addtolength\itemsep{-5pt}}

\def\eq            {\,{=}\,}  
\def\findim        {fini\-te-di\-men\-si\-o\-nal}
\def\Funle         {{\mathrm{Lex}}}                  
\def\Funre         {{\mathrm{Rex}}}                  
\def\Fun           {{\mathrm{Fun}}}            
\def\FunM          {{\Fun_\cala(\calm,\cala)}}

\def\H             {{}^{\#\!}}                       
\def\HH            {{}^{\#\!\!}}                     
\def\Hom           {\mathrm{Hom}}
\def\HomA          {\ensuremath{\Hom_\cala}}

\def\HomF          {\ensuremath{\Hom_\FunM}}
\def\HomM          {\ensuremath{\Hom_\calm}}
\def\HomDual       {\ensuremath{\Hom_\dualcat}}

\def\id            {{\mathrm{id}}}

\def\iHom          {\underline{\Hom}}
\def\icoHom        {\underline{\mathrm{coHom}}}
\def\iHomM         {\underline{\Hom}_\calm}
\def\iN            {\,{\in}\,}

\def\Itemize       {
	           ~\\[-1.65em] \begin{itemize}\addtolength\itemsep{-6pt}}
\def\ko            {{\ensuremath{\Bbbk}}}    
\def\la            {^{\rm la}}             
\def\lla           {^{\rm lla}}           
\def\lllla         {^{\rm lllla}}

\def\Mor           {\mathbf{\mathbbm{M}}}
\def\Mod           {{\rm Mod}}

\def\Nat           {\mathrm{Nat}}
\def\Nak           {\mathbb{N} }
\def\Nakr          {\mathbb{N}_\calm^r}
\def\Nakl          {\mathbb{N}_\calm^l}
\newcommand\nxl[1] {\\[#1mm]}
\newcommand\Nxl[1] {\\[-1.3em]\\[0.#1em]}

\def\one           {{\bf1}}

\def\opp           {^{\rm opp}}              
\def\Ot            {{\otimes}}
\def\ot            {\,{\otimes}\,}

\newcommand\rarr[1]{\xrightarrow{~#1~}}
\newcommand\Rarr[1]{\,{\xrightarrow{\,#1\,}}\,}
\def\ra            {^{\rm ra}}              
\def\re            {^{\rm re}}              
\def\rra           {^{\rm rra}}            
\def\rrrra           {^{\rm rrrra}}            

\newcommand{\Se}{\mathbb{S}}              
\newcommand{\lSe}{\overline{\mathbb{S}}}              
\def\Times         {\,{\times}\,}
\def\To            {\,{\to}\,}
\def\Xcong         {\xrightarrow{~\cong~}}
\def\xcong         {\,{\Xcong}\,}
\def\xnatiso       {{\;\xRightarrow{\,\cong\,}\,}}
\def\xNatiso       {{\;\xRightarrow{~\cong~}\,}}    

\def\un            {{\mathbf{1}_\Cala}}

\def\Vee           {{}^{\vee\!}}
\newcommand\void[1]{}
\def\xcong         {\,{\xrightarrow{~\cong\,}}\,}
\def\xsimeq        {\,{\xrightarrow{~\simeq~}}\,}

\def\Gsum    {\bigoplus_{g\in G} }

\newcommand\mixt[2]   {#1 \,{\odot}\, #2}        
\newcommand\mixto[2]  {#1 {\odot} #2}            
\newcommand\mixtd[2]  {#1 \,{\boxdot}\, #2}      
\newcommand\mixtdo[2] {#1 {\boxdot} #2}          
\def\Umixt            {\underline{\odot}}
\newcommand\umixt[2]  {#1 \,{\Umixt}\, #2}       
\newcommand\umixto[2] {#1 {\Umixt} #2}           
\def\Umixtd           {\underline{\boxdot}}
\newcommand\umixtdo[2]{#1 {\Umixtd} #2}          

\newcommand\bmixt[2]  {#1 \,{\ominus}\, #2}      
\newcommand\bmixto[2] {#1 {\ominus} #2}          
\newcommand\ubmixt[2] {#1 \,\underline{\ominus}\, #2} 
\newcommand\ubmixto[2] {#1 \underline{\ominus} #2}    
\newcommand\bmixtd[2] {#1 \,{\boxminus}\, #2}    
\newcommand\bmixtdo[2]{#1 {\boxminus} #2}        
\newcommand\ubmixtd[2] {#1 \,\underline{\boxminus}\, #2} 
\newcommand\ubmixtdo[2]{#1 \underline{\boxminus} #2}     

\begin{document}


\thispagestyle{empty}
\begin{flushright}
   {\sf ZMP-HH/22-13}\\
   {\sf Hamburger$\;$Beitr\"age$\;$zur$\;$Mathematik$\;$Nr.$\;$927}\\[2mm]
   July 2022
\end{flushright}

\vskip 2.7em

\begin{center}
	{\bf \Large Spherical Morita contexts and relative Serre functors}

\vskip 2.6em

{\large 
J\"urgen Fuchs\,$^{\,a},~~$
C\'esar Galindo\,$^{\,b},~~$
David Jaklitsch\,$^{\,c},~~$
Christoph Schweigert\,$^{\,c}$
}

\vskip 15mm

 \it$^a$
 Teoretisk fysik, \ Karlstads Universitet\\
 Universitetsgatan 21, \ S\,--\,651\,88\, Karlstad
 \\[9pt]
 \it$^b$
 Departamento de Matem\'aticas, Universidad de los Andes\\
 Carrera 1 \# 18A - 12, Edificio H, Bogot\'a, Colombia
 \\[9pt]
 \it$^c$
 Fachbereich Mathematik, \ Universit\"at Hamburg\\
 Bereich Algebra und Zahlentheorie\\
 Bundesstra\ss e 55, \ D\,--\,20\,146\, Hamburg

\end{center}

\vskip 3.7em

\noindent{\sc Abstract}\\[3pt]
The Morita context provided by an exact module category over
a finite tensor category gives a 
two-object bicategory with duals. Right and left duals of objects 
in the module category are given by internal Homs and coHoms, 
respectively. We express the double duals in terms of relative 
Serre functors, which leads to a Radford isomorphism for module 
categories. There is a bicategorical version of the Radford $S^4$ theorem: on the bicategory of a Morita context, the relative Serre functors
assemble into a pseudo-functor, and the Radford isomorphisms furnish 
a trivialization of the square of this pseudo-functor, i.e.\ of the 
fourth power of the duals.

 \noindent
We also show that the Morita bicategories coming from pivotal 
exact module categories are pivotal as bicategories, leading to 
the notion of pivotal Morita equivalence. This equivalence of 
tensor categories amounts to the equivalence of their 
bicategories of pivotal module categories. Furthermore, we 
introduce the notion of a spherical module category; it ensures 
that all categories in the Morita context of a spherical module 
category are spherical. Our results are motivated by and have 
applications to topological field theory.
                         
\newpage
\tableofcontents{}
\newpage


\section{Introduction}

Two familiar mathematical structures with countless applications are
(finite-dimensional) algebras and symmetric Frobenius algebras.
It is important to notice that a Frobenius algebra
is an algebra with additional structure (which can be expressed in different ways,
e.g.\ as a Frobenius form or as a homotopy fixed point). Topological field theory
in two dimensions relates these two types of structures to two different flavors
of topological field theories: algebras to framed theories,
and symmetric Frobenius algebras to oriented theories. Algebras 
and symmetric Frobenius algebras are thus associated with very different geometry. 
Algebraically, they should be carefully told apart as well.

A similar pattern is present in higher dimensions: finite tensor categories are
related to framed modular functors, while pivotal finite tensor categories are
related to oriented modular functors. Here a pivotal structure on a finite
tensor category is the additional datum of a monoidal trivialization of the double
dual. Finite tensor categories
and pivotal finite tensor categories should be carefully told apart, too.

It is natural to think about finite tensor categories as a tricategory, with 
bimodule categories as 1-morphisms, bimodule functors as 2-morphisms, and 
bimodule natural transformations as 3-morphisms. Invertible bimodule categories provide a particularly important subcategory. We could call this subcategory the Morita groupoid of finite tensor categories. An invertible bimodule category is in particular exact. Now
an indecomposable exact left \emph{module} category $\calm$ over a finite
tensor category $\cala$ is automatically also a right module category over the dual
finite tensor category $\cala_\calm^*$ and is thus a bimodule category. In fact 
-- as we will show in Theorem \ref{Morita_context_exact_module} --
it is even an invertible bimodule category and thus provides an invertible 
1-morphism. This theory is well-understood (see \cite{EGno}, and \cite{Mue} for 
the case of $*$-categories). Furthermore, it has applications to invertible 
topological defects between framed modular functors of Turaev-Viro type \cite{DSPS}.

Pivotal finite tensor categories are not only important for topological field theory,
but e.g.\ also \cite{Sh,fuSc25} as a prolific source of Frobenius algebras in tensor
categories. It is therefore highly desirable to have a similar Morita groupoid 
of \emph{pivotal} finite tensor categories at hand. Several concepts 
pertinent to such a structure have appeared 
recently, such as the notion of a relative Serre functor \cite{fuScSc} which 
allows one to introduce the notion of a pivotal module category. However, so far
they have not been assembled to a consistent \emph{pivotal Morita theory}.\,%
 \footnote{~This has not prevented physicists from computing successfully with
 structures that correspond to a pivotal Morita context as if it were
 just a pivotal monoidal category, see e.g.\ \cite{bhar}.}
The primary contribution of this paper is to develop such a theory. 

\medskip

A summary of our approach and of the organization of the paper is as follows.
The category of finite-dimensional modules over a finite-dimensional algebra
is a finite abelian category. The basic stimulus for Morita theory is to see 
this category as more fundamental than the algebra itself. Moreover, the
Eilenberg-Watts theorem implies that Morita equivalence of algebras can be 
captured in terms of bimodules. This story repeats itself in higher categorical
dimensions, replacing algebras by finite tensor categories $\cala$ and bimodules by
bimodule categories over them. This leads to the notion of categorical Morita 
equivalence. It is formulated in terms of exact module categories, i.e.\ 
$\cala$-module categories $\calm$ for which the internal Hom functor 
$\iHom(m,-)\colon \calm \,{\to}\, \cala$ is exact for every object $m\iN \calm$. 

After recalling some pertinent background in Section \ref{sec:background} so as 
to fix notation, we review and extend this categorical Morita theory in Section 
\ref{sec:Morita}. We show in Theorem \ref{Morita_bicategory} that a Morita context 
gives rise to a bicategory with two objects. Known coherence results for bicategories
then lead to strictification results for Morita contexts. Next we show that any 
exact module category induces a Morita context. This Morita context is actually 
\emph{strong}, i.e.\ the bimodule functors are equivalences. In Theorem 
\ref{strong_Morita} we provide a converse: any strong Morita context comes 
from an exact module category.

Section \ref{dualities_Morita} starts our study of pivotal Morita theory. We
first show that the bicategory induced by an exact module category is a 
bicategory with duals. To this end it is crucial to realize that
the duals of objects in the module category are module functors given by internal 
Homs and coHoms, respectively, see Definition \ref{def:duals}. When thinking
about the bimodule functors in the Morita context as ``mixed tensor products'',
one can work with the duals in the bicategory associated with a strong Morita context
in much the same way as in the 
familiar case of a single rigid monoidal category. Among the resulting insights 
are the statements that left and right duals are inverse and define a pseudo-functor 
from the bicategory to the bicategory with reversed 1- and 2-morphisms 
(Proposition \ref{dual_properties}), and that all internal Homs and 
coHoms can be expressed in terms of duals and mixed tensor products (Proposition 
\ref{Morita_adjunctions} and Remark \ref{iHoms_icoHoms_Morita}).
We proceed to compute the double dual endofunctors on the module categories. In 
Proposition \ref{double_duals} we show that they are given by the relative Serre
functors. A consequence is a Radford isomorphism for module categories:
\begin{thmx} 
Let $\cala$ be a finite tensor category and $\calm$ an exact $\cala$-module. There
is a natural isomorphism
  $$
  \DD_\Cala^{}\act
  {-}\actr\DD_{\!\dualcat}^{-1} \xRightarrow{~\cong~\:} \Se_\calm^\cala\cir\Se_\calm^\cala
  $$
of $($twisted$)$ bimodule functors, where $\Se_\calm^\cala$ is the relative Serre
functor of $\calm$ and $\DD_\Cala^{}$ and $\DD_{\!\dualcat}$ are the distinguished
invertible objects of $\cala$ and $\dualcat$, respectively.
\end{thmx}

The relative Serre functors of the categories in a Morita context form a pseudo-functor
on its associated bicategory (see Definition \ref{def_Serre_pseudo_functor}). In 
analogy to {\Color Corollary} \ref{thm:Radford}, there are Radford isomorphims for all the 
categories in the Morita context bicategory. In Theorem \ref{Thm_Radford_pseudo} we show
that these assemble into a trivialization of the square of the relative Serre 
pseudo-functor. Altogether these results establish natural rigid duality structures 
for Morita bicategories.

Whenever one deals with dualities, the question arises whether the double dual
admits a trivialization in an appropriate sense. Such a trivialization, known as
a \emph{pivotal} structure, has important consequences. For instance, it is needed
to define oriented topological field theories. Consequently,
in order to study dualities for oriented topological field theories,
it is central to understand Morita equivalence in a pivotal setting. The interaction
between Morita theory and pivotal structures is the subject of Section
\ref{sec:piv}. The notion of a pivotal structure for monoidal and module categories 
is well understood, see Definition \ref{def:pivmodule}. The general definition of a 
pivotal bicategory then directly leads to Definition 
\ref{def:pivMorita} of a \emph{pivotal Morita context}. 
These different aspects of pivotality fit together well: 

\begin{thmya} 
The Morita context associated with a pivotal module category $\calm$ over a 
pivotal tensor category $\cala$ is a pivotal Morita context.
\end{thmya} 

\noindent
This result is by no means
trivial: the pivotality of a module category $\calm$ just imposes a restriction
on $\calm$ as a one-sided module; in contrast, the pivotality of the Morita context 
implies that $\calm$ is pivotal as a \emph{bi}module category.
 \\
Our study of pivotality culminates in the following statements:

\begin{thmyb} 
Two pivotal tensor categories are pivotal Morita equivalent if and only if 
their associated $2$-categories of pivotal module categories, module functors
and module natural transformations are equivalent as pivotal bicategories.
\end{thmyb} 

\begin{thmyc} 
If two pivotal categories are pivotal Morita equivalent,
then their Drinfeld centers are equivalent as pivotal braided tensor categories.
\end{thmyc} 

\noindent
Section \ref{sec:spherical} is devoted to the property of a pivotal structure being 
spherical. Lately the perspective on sphericality has undergone a change: while 
trace-sphericality emphasized the equality of left and right traces for 
endomorphisms in a pivotal category, a more recent notion of sphericality \cite{DSPS}
requires the pivotal structure to square to Radford's trivialization of the 
quadruple dual. The two notions agree in the semisimple case. {\Color We prove 
that (beyond semisimplicity) the latter notion of sphericality is invariant 
under pivotal Morita equivalence:

\begin{cors}
Let $\cala$ and $\calb$ be two pivotal tensor categories that are pivotal 
Morita equivalent. $\cala$ is $($unimodular$)$ spherical if and only 
if $\calb$ is $($unimodular$)$ spherical.
\end{cors}

\noindent 
In the same spirit as for tensor categories \cite{DSPS}, we propose in Definition 
\ref{spherical_mod} a notion of sphericality for a pivotal module category 
${}_\cala\calm$ by means of the Radford isomorphism from Theorem \ref{thm:Radford}. 
However, for module categories there is a subtlety: this definition is relative to 
the choice of trivializations of the distinguished invertible objects of $\cala$ 
and $\dualcat$.
Fixing such trivializations, all monoidal and module categories in the Morita 
context associated to a spherical module category are spherical, which justifies 
our definition. Under the additional assumption of semisimplicity, there are 
distinguished choices of such trivializations (see Remark \ref{semi_dist_choices}).
}

In the final Section \ref{sec:equivariant} we extend some of the main
results of the previous sections to the equivariant case.


\section{Background}\label{sec:background}

In this section we fix notation and conventions and revisit some pertinent concepts and structures.
Throughout the text, all categories we consider are supposed to be linear 
abelian over an algebraically closed field \ko. 
A linear abelian category where every object is of finite length and all 
morphism spaces are \findim\ is said to be \emph{locally finite}. A \emph{finite} 
\ko-linear abelian category is a linear abelian category that is equivalent to 
the category of \findim\ modules over a \findim\ \ko-algebra.


\subsection{Tensor categories and module categories}

We first recall a few standard definitions in the theory of tensor categories 
(see \cite{EGno}). A \emph{multi-tensor category} is a locally finite rigid 
monoidal category $\cala$ whose tensor product functor $\otimes$ is bilinear. 
$\cala$ is said to be a \emph{tensor category} iff in addition its monoidal 
unit $\un$ is a simple object. We take without loss of generality a monoidal 
category to be strict to simplify the exposition. The \emph{monoidal opposite} 
$\overline{\cala\,}$ of a monoidal category $\cala$ is the monoidal category 
with the same underlying category as $\cala$, but with reversed tensor product, 
i.e.\ $a \,{\otimes_ {\overline{\cala\,}}}\,b \eq b \ot a$ and with the 
respectively adjusted associators. When convenient we denote the object in 
$\overline{\cala\,}$ that corresponds to an object $a \iN \cala$ by $\overline{a}$.

Concerning dualities on a monoidal category $\cala$ our conventions are as 
follows. A right dual $a^\vee$ of an object $a\iN\cala$ comes equipped with 
evaluation and coevaluation morphisms
  \be
  {\rm ev}_a\Colon a^\vee \Ot\, a\to \un \qquad{\rm and}\qquad
  {\rm coev}_a\Colon \un\to a\ot a^\vee .
  \ee
Similarly, a left dual $\Vee a$ and of $a\iN \cala$ comes with evaluation and 
coevaluation morphism
  \be
  \widetilde{{\rm ev}_a}\Colon a \,\Ot \Vee a\to \un \qquad{\rm and}\qquad
  \widetilde{{\rm coev}_a}\Colon \un\to \Vee a\ot a \,.
  \ee
In every finite tensor category $\cala$ there is a \emph{distinguished invertible 
object} $\DDA$, which comes with a monoidal natural isomorphism
  \be\label{Radford}
  r_\Cala^{}\Colon \DD_\Cala^{}\otimes-\otimes\DD_\cala^{-1}
  \xRightarrow{~\cong~} (-)^{\vee\vee\vee\vee}
  \ee
known as the \emph{Radford isomorphism}. A finite tensor category $\cala$ is said
to be \emph{unimodular} iff $\DDA$ is isomorphic to the monoidal unit $\un$.

 \medskip

A \emph{$($left$)$ module category} over a tensor category $\cala$, or 
\emph{$($left$)$ $\cala$-module}, for short, is a category $\calm$ together with an 
exact \emph{module action} functor 
  \be
  \Act \Colon \cala\Times\calm\to\calm 
  \ee
and a mixed associator obeying a pentagon axiom. In order to indicate the
tensor category over which $\calm$ is a module, we also denote it by
${}_\cala\calm$. Invoking an analogue for module categories of Mac Lane's 
strictification theorem (see \Cite{Rem.\,7.2.4}{EGno}), we assume strictness also
for module categories.  In the case that $\cala$ is finite, we require that 
$\calm$ is finite as a linear category as well. 

A \emph{right} $\calb$-module $\caln$ is defined as a left 
$\overline{\calb}$-module; its module action functor is denoted by 
  \be
  \actr \Colon \caln\Times\calb\to\caln ,
  \ee
and we also denote it by $\caln_\calb$. Similarly, for finite tensor categories 
$\cala$ and $\calb$, an $(\cala,\calb)$-\emph{bimodule category} is defined as a 
(left) module category over the Deligne product $\cala\boti\overline{\calb}$.

An $\cala$-module category $\calm$ is called \emph{exact} iff $p\act m$ is 
projective in $\calm$ for any projective object $p\iN\cala$ and any object $m\iN\calm$.
A module category which is not equivalent to a direct sum of two non-trivial module 
categories is said to be \emph{indecomposable}.

{\Color A \emph{module functor} between $\cala$-module categories $\calm$ and $\caln$ is a functor $H\colon \calm\Rarr{}\caln$ together with a \textit{module constraint}, i.e.\ a collection of natural isomorphisms $H(a\act m)\xcong a\act H(m)$ for $a\iN\cala$ and $m\iN\calm$ obeying appropriate pentagon axioms. A \emph{module natural transformation} between module functors is a natural transformation between the underlying functors that commutes with the respective module structures.
We denote by $\Fun_\cala(\calm,\caln)$ the category which has module functors between two $\cala$-modules $\calm$ and $\caln$ as objects and module natural transformations as morphisms. We denote by $\Nat_{\rm mod}(H_1,H_2)$ the space of module natural transformations between given module functors $H_1$ and $H_2$. Similarly, $\Funre_\cala(\calm,\caln)$ and $\Funle_\cala(\calm,\caln)$ denote the categories of right exact and left exact module 
functors, respectively.  In the case that $\calm$ is exact, every $H\iN \Fun_\cala(\calm,\caln)$ is an exact module functor.
}
\begin{lem}{\rm \Cite{Cor.\,2.13}{DSPS2}}
Let $H\colon \calm\Rarr{}\caln$ be an $\cala$-module functor. If its underlying 
functor admits a right $($left$)$ adjoint functor, then $H$ admits
a right $($left$)$ adjoint $\cala$-module functor such that the unit and counit of the adjunction are module natural transformations.
\end{lem}

The \emph{dual tensor category} $\dualcat$ of a tensor 
category $\cala$ with respect to an $\cala$-module $\calm$ is the category of right 
exact module endofunctors, $\dualcat \eq \Funre_\cala(\calm,\calm)$, with tensor product 
given by composition of functors. If $\calm$ is exact $\dualcat$ is rigid and in case $\calm$ indecomposable, then the identity functor $\id_\calm$ is simple, making $\dualcat$ a tensor category. The evaluation of a functor on an object turns $\calm$ into a $\dualcat$-module category.

The category of module functors has the structure of a bimodule category. More
specifically, for given bimodule categories ${}_\cala\calm_\calb$, 
${}_\cala\caln_\calc$ and ${}_\cald\call_\calb$, 
$\Fun_\cala(\calm,\caln)$ becomes a $(\calb,\calc)$-bimodule category via the actions
  \be
  \label{left_module_functors}
  b\actF H:= H\circ (-\actr b) \qquad{\rm and}\qquad H\actrF c:=(-\actr c)\circ H 
  \ee
for $H \iN \Fun_\cala(\calm,\caln)$, $b\iN\calb$ and $c\iN\calc$,
while $\Fun_\calb(\calm,\call)$ inherits a $(\cald,\cala)$-bimodule category structure
given by
  \be
  \label{right_module_functors}
  d\actF H:= (d\act -)\circ H \qquad{\rm and}\qquad H\actrF a:=H\circ (a\act -) \,.
  \ee    
Analogously, the categories of right exact and left exact module functors are endowed with the structure of a bimodule category as well.

The opposite category $\calm\opp$ of the linear category $\calm$ that underlies a bimodule ${}_\cala\calm_\calb$ can be endowed in many different
ways with the structure of a $(\calb,\cala)$-bimodule category, by twisting the 
actions with odd powers of duals.

\begin{defi}
Let $\calm$ be an $(\cala,\calb)$-bimodule category over tensor categories $\cala$
and $\calb$.
We define $\HH\calm$ as the $(\calb,\cala)$-bimodule category with underlying 
category $\calm\opp$ and actions given by 
  \be
  b\act\overline{m}\actr a := \overline{\Vee a\act m \actr {}^\vee b}
  \ee
for $a\iN\cala$, $b\iN\calb$ and $\overline{m}\iN\calm\opp$. Similarly, $\calm^\#$ is
defined to be the $(\calb,\cala)$-bimodule with actions twisted by right duals, i.e.
  \be
  b\act\overline{m}\actr a:=\overline{a^\vee\act m \actr b^\vee} 
  \ee
for $a\iN\cala$, $b\iN\calb$ and $\overline{m}\iN\calm\opp$.
\end{defi}


\subsection{Relative Deligne product}

Let $\calm$ be a right module and $\caln$ a left module over a finite tensor
category $\calb$.

\begin{enumerate}[(i)]
    \item 
A $\calb$-\emph{balancing} on a bilinear functor $F \colon \calm\Times\caln\To\call$ 
into a linear category $\call$ is a natural family of isomorphisms 
  \be
  F(m\actr b,n) \Xcong F(m,b\act n)
  \ee
for $b\iN\calb$, $m\iN\calm$ and $n\iN\caln$,
obeying an obvious pentagon coherence condition (for details, see for instance 
\Cite{Def.\,2.12}{schaum2}). A bilinear functor endowed with a balancing is called
{\Color a} \emph{balanced functor}.

    \item 
A \emph{balanced natural transformation} between balanced functors is a 
natural transformation between the underlying functors that commutes with the respective balancings.

    \item 
Balanced functors $F \colon \calm\Times\caln\To\call$ together with balanced 
natural transformations form a category, denoted by ${\rm Bal}(\calm\times\caln,\call)$.
Its full subcategory of right exact balanced functors is denoted by
${\rm Bal}\re(\calm\times\caln,\call)$.
		
    \item 
The \emph{relative Deligne product} of $\calm_\calb$ and ${}_\calb\caln$ is a linear category $\calm\boxtimes_\calb\caln$ equipped with a right exact 
$\calb$-balanced functor $\boxtimes_\calb\colon \calm\Times\caln \Rarr{}
\calm\boxtimes_\calb\caln$ such that for every linear category $\call$ the functor
  \be
  \label{relative_Deligne_up}
  \begin{array}c
  \Funre(\calm\boxtimes_\calb\caln,\call) \xrightarrow{\;\simeq\;}
  {\rm Bal}\re(\calm{\times}\caln,\call)\,,
  \Nxl4
  \hspace*{2.5em} F\xmapsto{~~~} F \,{\circ}\;\boxtimes_\calb
  \eear
  \ee
is an equivalence of categories. The relative Deligne product exists, different realizations are known, e.g. \cite{DSPS2,fuScSc} and Corollary 2.4 below.
		
    \item 
The relative Deligne product of bimodule categories is naturally endowed with the 
structure of a bimodule category. Accordingly, the equivalence 
\eqref{relative_Deligne_up} descends to an equivalence between the categories 
of bimodule right exact functors and balanced right exact bimodule functors 
(see \Cite{Prop.\,3.7}{schaum2}).
\end{enumerate}

We call an object of the form $m\boti n\iN\calm\boxtimes_\calb\caln$
a \emph{$\boxtimes$-factorized object} of $\calm\boxtimes_\calb\caln$.

\begin{lem}\label{Deligne_objects}
Every object in the relative Deligne product $\calm\boxtimes_\calb\caln$ is
isomorphic to a finite colimit of $\boxtimes$-factorized objects.
\end{lem}

\begin{proof}
As shown in \Cite{Thm.\,3.3}{DSPS2}, the relative Deligne product can be realized 
as a category of bimodules internal to $\calb$. Moreover, as pointed out there
in the proof, any such bimodule can be written as a coequalizer of 
$\boxtimes$-factorized objects.
\end{proof}

\begin{prop}{\rm \Cite{Prop.\,2.4.10}{DSPS}}\label{Deligne_eq}
Let ${}_\cala\calm_\calb$, ${}_\cala\caln_\calc$ and ${}_\cald\call_\calb$ be 
bimodule categories over finite tensor categories. The balanced functor 
  \be
  \begin{array}c
  \Funre_\cala(\calm,\cala)\Times\caln \rarr{~} \Funre_\cala(\calm,\caln) \,,
  \Nxl4
  \hspace*{2.0em} (H,n) \xmapsto{~~~} H(-)\act n
  \eear
  \ee
induces an adjoint equivalence
  \be
  \label{Deligne_left_eq} 
  \Funre_\cala(\calm,\cala)\boxtimes_\cala\caln \xrightarrow{\;\simeq\;}
  \Funre_\cala(\calm,\caln)
  \ee
of $(\calb,\calc)$-bimodule categories. Similarly, there is an adjoint equivalence
  \be
  \label{Deligne_right_eq}
  \call\boxtimes_\calb\Funre_\calb(\calm,\calb) \xrightarrow{\;\simeq\;}
  \Funre_\calb(\calm,\call)
  \ee
of $(\cald,\cala)$-bimodule categories.
\end{prop}

\begin{cor}{\rm \Cite{Cor.\,2.4.11}{DSPS}}\label{relative_Deligne_module_functors}
Let ${}_\cala\calm_\calb$ and ${}_\calb\caln_\calc$ be bimodule categories over 
finite tensor categories. There are equivalences 
  \be
  \calm\boxtimes_\calb\caln\simeq\Funre_\calb(\HH\calm,\caln) \qquad {\rm and}\qquad
  \calm\boxtimes_\calb\caln\simeq\Funre_\calb(\caln^\#\!,\calm)
  \ee
of $(\cala,\calc)$-bimodule categories
between relative Deligne products and categories of right exact module functors.
\end{cor}


\subsection{Internal Hom and coHom}

Given a module category ${}_\cala\calm$ over a finite tensor category, for every
object $m\iN\calm$ the action functor $-{\act}m\colon \cala\To\calm$ is exact 
and therefore comes with a right adjoint $\iHomM^\cala(m,-)\colon \calm\To\cala$, 
i.e.\ there are natural isomorphisms
  \be
  \Hom_\calm(a\act m,n) \rarr\cong \Hom_\cala(a,\iHomM^\cala(m,n))
  \ee
for $a\iN\cala$ and $m,n\iN\calm$. This extends to a left exact functor
  \be
  \label{inn_hom_functor}
  \iHom_\calm^\cala(-,-)\Colon \calm\opp\Times\calm \rarr~ \cala \,,
  \ee
which is called the \emph{internal Hom} functor of ${}_\cala\calm$. We denote 
the internal Hom also by $\iHomM$ or just by $\iHom$ when it is clear from the 
context which module category is meant. Additionally, there are canonical natural
isomorphisms
  \be
  \label{ihom_lmod} 
  \iHomM^\cala(m,a\act n) \rarr\cong a\ot \iHomM^\cala(m,n)
  \ee
which turn $-{\act} m\dashv \iHom_\calm(m,-)$ into an adjunction of $\cala$-module 
functors, as well as
  \be
  \label{ihom_rmod}
  \iHomM^\cala(a\act m,n) \rarr \cong \iHomM^\cala(m,n)\ot a^\vee .
  \ee
Together these naturally endow the internal Hom with a bimodule functor structure 
  \be
  \label{ihom_bimod}
  \iHomM^\cala(-,-)\Colon \HH\calm \Times\calm \rarr~ \cala\,.\ee
 
The internal Hom is well-behaved with respect to adjunctions of module functors:

\begin{lem}{\rm \Cite{Lemma\,3}{fuSc25}}\label{adjoints_ihom} 
Let $\cala$ be a finite tensor category, $\calm$ and $\caln$ be $\cala$-module 
categories, and $L\colon \calm\To\caln$ and $R\colon \caln\To\calm$ be $\cala$-module 
functors. Then $L$ is left adjoint to $R$ if and only if there are natural isomorphisms
  \be
  \iHom_\caln^\cala(L(m),n) \rarr \cong \iHomM^\cala(m,R(n))
  \ee
for $m\iN\calm$ and $n\iN\caln$.
\end{lem}

The monoidal category $\overline{\dualcat}$ acts on a module category $\calm$ by evaluation of a functor on an object. If  ${}_\cala\calm$ is exact, then the dual tensor category is rigid,
and a left dual for $F\iN\overline{\dualcat}$ is given by its left adjoint $F\la$. In
this situation the isomorphisms from Lemma \ref{adjoints_ihom} turn the internal Hom 
\eqref{ihom_bimod} into an $\overline{\dualcat}$-balanced bimodule functor.

 \medskip

Dually, for $m\iN\calm$ the action functor $-{\act} m\colon \cala\To\calm$ has a 
left adjoint 
  \be
  \Hom_\calm(n,a\act m) \rarr\cong \Hom_\cala(\icoHom_\calm^\cala(m,n),a) \,,
  \ee
called the \emph{internal coHom}, where $a\iN\cala$ and $m,n\iN\calm$. This
extends to a right exact functor
$\icoHom_\calm^\cala(-,-)\colon \calm\opp\Times\calm \To\cala$.
The internal Hom and coHom are related by
  \be
  \label{hom_cohom}\icoHom_\calm^\cala(m,n) \cong {}^\vee\iHom_\calm^\cala(n,m)
  \qquad\text{for}~~ m,n\iN\calm \,.
  \ee 
Hence analogously to \eqref{ihom_lmod}) and \eqref{ihom_rmod}), there are 
coherent natural isomorphisms
  \be
  \icoHom_\calm^\cala(m,a\act n) \rarr\cong a\ot \icoHom_\calm^\cala(m,n)
  \ee 
and 
  \be
  \icoHom_\calm^\cala(a\act m,n) \rarr \cong \icoHom_\calm^\cala(m,n)\ot \Vee a
  \ee 
turning the internal coHom into an $\cala$-bimodule functor
  \be
  \label{icohom_bimod}
  \icoHom_\calm^\cala(-,-)\Colon \calm{}^\#\Times\calm\rarr~\cala\,.
  \ee
In view of the relation \eqref{hom_cohom},  Lemma \ref{adjoints_ihom} takes the 
following form:

\begin{lem}\label{adjoints_icohom} 
Let $\calm$ and $\caln$ be module categories over a finite tensor category $\cala$, 
and $L\colon \calm\To\caln$ and $R\colon\caln\To\calm$ be $\cala$-module functors.
Then $L$ is left adjoint to $R$ if and only if there are natural isomorphisms
  \be
  \label{icohom_adjunction} 
  \icoHom_\calm^\cala(R(n),m) \rarr \cong \icoHom_\caln^\cala(n,L(m))
  \ee 
for $n\iN\caln$ and $m\iN\calm$.
\end{lem}

Again, if $\calm$ is exact, the isomorphisms (\ref{icohom_adjunction}) provide 
an $\overline{\dualcat}$-balancing to the internal coHom bimodule functor 
\eqref{icohom_bimod}.


\subsection{Relative Serre Functors}

A Serre functor $\Se$ of a linear additive category $\calx$ with finite dimensional 
morphism spaces furnishes natural isomorphisms 
between the vector spaces $\Hom_\calx(x,y)$ and $\Hom_\calx (y,\Se(x))^*$. In case
$\calx$ is finite abelian, a Serre functor only exists if $\calx$ is semisimple.
But in the case that $\calx=\calm$ is a finite module category over a finite tensor 
category $\cala$, then there is internalized version
of the Serre functor which exists beyond the semisimple case:

\begin{defi}{\Cite{Def.\,4.22}{fuScSc}}
Let $\calm$ be a left $\cala$-module category. A \emph{$($right$)$ relative 
Serre functor} on $\calm$ is an endofunctor $\Se_\calm^\cala \colon \calm\To\calm$
together with a family
  \be
  \iHomM^\cala\left(m,n\right)^\vee \rarr\cong
  \iHomM^\cala\left(n,\Se_\calm^\cala(m)\right)
  \ee
of natural isomorphisms for $m,n\iN\calm$. Similarly, a \emph{$($left$)$ relative
Serre functor} $\lSe_\calm^\cala$ comes with a family
  \be
  {}^\vee\iHomM^\cala\left(m,n\right) \rarr\cong
  \iHomM^\cala\left(\lSe_\calm^\cala(n),m\right)
  \ee
of natural isomorphisms.
\end{defi}

According to \Cite{Prop.\,4.24}{fuScSc} a module category admits relative 
Serre functors if and only if it is exact. In that case the left and right 
relative Serre functors are quasi-inverses of each other and can be uniquely 
identified (see also \Cite{Lemmas\,3.3-3.5}{Sh}) by the formulas
  \be\label{eq:relSerre4M}
  \Se_\calm^\cala(m)\cong \iHomM^\cala(m,-)\ra(\un) \qquad\text{and}\qquad
  \lSe_\calm^\cala(m)\cong \icoHom_\calm^\cala(m,-)\la(\un)\,.
  \ee

\begin{prop}
For $\calm$ an exact $(\cala,\calb)$-bimodule category, there are coherent
natural isomorphisms
  \be
  \label{Serre_twisted}
  \Se_\calm^\cala(a\act m)\cong a\dd\act \Se_\calm^\cala(m) \qquad\text{and}\qquad
  \Se_\calm^\cala(m\actr b)\cong \Se_\calm^\cala(m)\actr b\dd
  \ee
which turn the relative Serre functor $\Se_\calm^\cala$ into a twisted bimodule 
equivalence. In a similar manner, its quasi-inverse comes with coherent natural 
isomorphisms
  \be
  \lSe_\calm^\cala(a\act m)\cong \ldd a\act \lSe_\calm^\cala(m) \qquad\text{and}\qquad
  \lSe_\calm^\cala(m\actr b)\cong \lSe_\calm^\cala(m)\actr \ldd b
  \ee
making of $\lSe_\calm^\cala$ a twisted bimodule equivalence.
\end{prop}

\begin{proof}
For left module categories this was already shown in \Cite{Lemma\,4.23}{fuScSc}. Now
for $b\iN\calb$ we have an adjunction $(-{\actr}b) \,{\dashv}\, (-{\actr} b^{\vee})$ 
of $\cala$-module functors. It follows that there is a chain
  \be
  \begin{aligned}
  \iHomM^\cala(n,\Se_\calm^\cala(m\actr b))
  &\cong \iHomM^\cala(m\actr b,n)^\vee\cong\iHomM^\cala(m,n\actr b^\vee)^\vee
  \\[.4em]
  &\cong \iHomM^\cala(n\actr b^\vee,\Se_\calm^\cala(m))
  \cong \iHomM^\cala(n,\Se_\calm^\cala(m)\actr b\dd)
  \end{aligned}
  \ee
of natural isomorphisms, where $m,n\iN\calm$ and $b\iN\calb$. Hence the desired 
family of isomorphisms is granted by the Yoneda Lemma for internal Hom's, and by 
construction these are coherent with respect to the tensor product.
\end{proof}

\begin{rem}
A right module category $\caln_\calb$ can be seen as a left module category 
${}_{\overline{\calb}}\caln$ over the monoidal opposite of $\calb$. Since right 
duals in $\calb$ correspond to left duals in its monoidal opposite and vice versa, 
the module structure on the relative Serre functors is twisted according to
  \be\label{Serre_twisted_op}
  \Se_\caln^{\overline{\calb}}(n\actr b)
  \cong \Se_\caln^{\overline{\calb}}(n)\actr \ldd b \qquad\text{and}\qquad
  \lSe_\caln^{\overline{\calb}}(n\actr b)
  \cong \lSe_\caln^{\overline{\calb}}(n)\actr  b\dd
  \ee
for $n\iN\caln$ and $b\iN\calb$. In the case of a bimodule category ${}_\cala\caln_\calb$, the actions of $\cala$ are twisted by
  \be\label{Serre_twisted_op2}
  \Se_\caln^{\overline{\calb}}(a\act n)
  \cong \ldd a\act\Se_\caln^{\overline{\calb}}(n) \qquad\text{and}\qquad
  \lSe_\caln^{\overline{\calb}}(a\act n)
  \cong a\dd\act\lSe_\caln^{\overline{\calb}}(n)\,.
  \ee
\end{rem}

\begin{prop}\label{Serre_and_module_functors}
Given exact $\cala$-module categories $\calm$ and $\caln$, for every module functor 
$F \colon \calm\To\caln$ there is a natural isomorphism
  \be
  \label{Serre_module_functor}
  \Lambda_F\Colon \Se_\caln^\cala\cir F \xRightarrow{~\cong~}
  F\rra\cir \Se_\calm^\cala
  \ee
of twisted module functors, where $F\rra$ is the double right adjoint of $F$. 
$\Lambda_F$ is compatible with composition of module functors, i.e.\ the diagram
  \be
  \label{Serre_compatibility_composition}
  \begin{tikzcd}[row sep=2.3em,column sep=2.9em]
  \Se_\call^\cala\cir H\cir F \ar[r,"\Lambda_{H\circ F}\,",Rightarrow]
  \ar[d,swap,"\Lambda_H\circ\id\,",Rightarrow]
  & (H\cir F)\rra\cir \Se_\calm^\cala \ar[d,"\,\cong",Rightarrow]
  \\
  H\rra\cir\Se_\caln^\cala\,\cir F \ar[r,swap,"\id\circ\Lambda_F~",Rightarrow]
  & H\rra\cir F\rra\cir \Se_\calm^\cala
  \end{tikzcd}
  \ee
commutes for $H\iN\Fun_\cala(\caln,\call)$ and $F\iN\Fun_\cala(\calm,\caln)$. 
Analogously, in the case of bimodule categories and a bimodule functor $F$, 
the natural isomorphism $\eqref{Serre_module_functor}$ is an isomorphism of 
twisted bimodule functors.
\end{prop}

\begin{proof}
The first part of the statement referring to left modules is shown in 
\Cite{Thm.\,3.10}{Sh}. For bimodules ${}_\cala\calm_\calb$ and ${}_\cala\caln_\calb$
and $F\colon \calm\To\caln$ a bimodule functor, we now check that the isomorphism 
$\Lambda_F$ in \eqref{Serre_module_functor} is compatible with the $\calb$-module 
functor structures: Consider the diagram
  \be
  \label{eq:Lambda-compat}
  \begin{tikzcd}[row sep=3.1em,column sep=3.4em]
  \Se_\caln^\cala\cir F(m\actr b) \ar[r,"\psi"] \ar[d,"\Lambda_F\circ\,\id",swap]
  \ar[dr,"\Lambda_{F(-)\actr b}"]
  & \Se_\caln^\cala\left(F(m)\actr b\right) \ar[r,"\Lambda_{-\Actr b}"]
  \ar[dr,"\Lambda_{F(-\Actr b)}"]
  & \Se_\caln^\cala\cir F(m)\actr b\dd \ar[d,"\Lambda_F\actr\id"]
  \\
  F\rra\cir \Se_\calm^\cala(m\Actr b) \ar[r,swap,"\id\,\circ\Lambda_{-\Actr b}"]
  & F\rra\big(\Se_\calm^\cala(m)\actr b\dd\big) \ar[r,swap,"\psi\rra"]
  & F\rra\cir \Se_\calm^\cala(m)\actr b\dd
  \end{tikzcd}
  \ee
where $\psi_{m,b}\colon F(m\actr b)\Rarr\cong F(m)\actr b$ denotes the 
$\calb$-module functor structure on $F$ and $\Lambda_{-\Actr b}$ is the 
twisted $\calb$-module structure of the relative Serre functor. Both of the 
triangles in \eqref{eq:Lambda-compat} are realizations of diagram 
\eqref{Serre_compatibility_composition} and are thus commutative. The middle square
in \eqref{eq:Lambda-compat} commutes owing to naturality of (\ref{Serre_module_functor}) with respect to $F$.
\end{proof}

\begin{rem}\label{Serre_quasiinverse}
An analogous statement as in Proposition \ref{Serre_and_module_functors} holds 
for left relative Serre functors: Given $F\iN\Fun_\cala(\calm,\caln)$, there is
a natural isomorphism
  \be\label{lSerre_module_functor}
  \overline{\Lambda}_F\Colon \lSe_\caln^\cala\circ F\xRightarrow{~\cong~\,}
  F\lla\!\circ \lSe_\calm^\cala
  \ee
of twisted module functors, compatible with composition of module functors in 
a similar manner as in (\ref{Serre_compatibility_composition}). These isomorphisms 
are related with (\ref{Serre_module_functor}) by the commutative diagram
  \be
  \begin{tikzcd}[row sep=2.3em,column sep=4.9em]
  \lSe_\caln^\cala\cir F \ar[r,"\overline{\Lambda}_F",Rightarrow]
  \ar[d,"\cong",swap,Rightarrow]
  & F\lla\cir\lSe_\calm^\cala \ar[d,"\cong",Rightarrow]
  \\
  \lSe_\caln^\cala\cir F\cir\Se_\calm^\cala\cir\lSe_\calm^\cala
  & \ar[l,"~\id\circ\Lambda_{F\lla}\circ\id",Rightarrow]
  \lSe_\caln^\cala\cir \Se_\caln^\cala\cir F\lla\cir \lSe_\calm^\cala
  \end{tikzcd}
  \ee
considering that the left and right relative Serre functors are mutual 
quasi-inverses. Likewise also the statements for bimodule functors hold.
\end{rem}

\begin{rem}
Given a finite linear category $\calx$, its right exact \emph{Nakayama functor}
is the image of the identity functor under the Eilenberg-Watts correspondence 
\Cite{Def.\,3.14}{fuScSc}, i.e.\ explicitly
  \be\label{Nakayama}
  \mathbb{N}_\calx^r := \int^{x\in\calx}\!\!\Hom_\calx (-,x)^* \ot x \,.
  \ee
It comes equipped with a family of natural isomorphisms
\be\label{Nakayama_twisted_functor}
\mathbb{N}_\calx^r\circ F\xRightarrow{\;\cong~}F\rra\circ\mathbb{N}_\calx^r
\ee
for every $F\in\Funre(\calx,\calx)$ having a right adjoint which is right exact as well.
The left exact analogue of the Nakayama functor is a left adjoint to 
$\mathbb{N}_\calx^r$ which is explictly given by
  \be
  \mathbb{N}_\calx^l := \int_{x\in\calx}\!\!\Hom_\calx (x,-) \ot x \,.
  \ee
{\Color  
In a similar manner as $\mathbb{N}_\calx^r$, the functor $\mathbb{N}_\calx^l$ is endowed with a family of natural isomorphisms
\be\label{lNakayama_twisted_functor}
\mathbb{N}_\calx^l\circ F\xRightarrow{\;\cong~}F\lla\circ\mathbb{N}_\calx^l
\ee
for every $F\iN\Funle(\calx,\calx)$ having a left adjoint that is left exact as well.

The Nakayama functors of the $\ko$-linear category underlying a finite multitensor category $\cala$ \Cite{Lemma\ 4.10}{fuScSc} can be described, with the help of \eqref{Nakayama_twisted_functor} and \eqref{lNakayama_twisted_functor}, as
\be\label{tensor_Nakayama}
\mathbb{N}_\cala^r\cong \DD_\cala^{-1}\otimes(-)\dd\quad\text{ and }\quad 
\mathbb{N}_\cala^l\cong \DD_\cala\otimes\ldd(-)\,,
\ee
where $\DDA\eq\mathbb{N}_\cala^l(\un)$ is the \emph{distinguished invertible object} and $\DDA^{-1}\eq\mathbb{N}_\cala^r(\un)$ \Cite{Lemma\ 4.11}{fuScSc}. Composing the isomorphisms \eqref{tensor_Nakayama} with the Radford isomorphism \eqref{Radford} we obtain
\be\label{tensor_Nakayama2}
\mathbb{N}_\cala^r\cong \ldd(-)\otimes\DD_\cala^{-1}\quad\text{ and }\quad \mathbb{N}_\cala^l\cong  (-)\dd\otimes\DD_\cala\;.
\ee
A finite tensor category $\cala$ is said to be \emph{unimodular} iff its distinguished invertible object $\DDA$ is isomorphic to the monoidal unit $\un$.
}

By \eqref{Nakayama_twisted_functor} the Nakayama functor  of an exact module category $\calm$ over a finite 
tensor category $\cala$ is endowed with a twisted module functor structure. The relative Serre functors of $\calm$ are related to the Nakayama functors by
  \be\label{Nakayama_Serre}
  \DD_\cala\act\mathbb{N}_\calm^r \cong \Se_\calm^\cala \qquad\text{and}\qquad
  \DD_\cala^{-1}\act\mathbb{N}_\calm^l \cong \lSe_\calm^\cala
  \ee
as twisted module functors \Cite{Thm.\,4.26}{fuScSc}. A direct computation shows that
for a bimodule category these are indeed isomorphisms of twisted bimodule functors.
\end{rem}


\section{Categorical Morita context}\label{sec:Morita}

In the theory of rings or $\ko$-algebras the notion of Morita equivalence 
finds a generalization in the structure of a Morita context or pre-equivalence 
data \Cite{Ch.2\,\S3}{Bass}. These data involving two rings provide an adjunction
(though not necessarily an equivalence) between their categories of modules. We 
now study the analogue of this notion for categorical Morita equivalence of 
finite tensor categories.

\begin{defi}\label{Morita_context}~\\
{\rm (i)} 
A (categorical) \emph{Morita context} consists of the following data:
\Enumerate
\addtolength\itemsep{-0.3em}
    \item 
Two finite (multi-)tensor categories $\cala$ and $\calb$.
    \item 
Two bimodule categories ${}_\cala\calm_\calb$ and ${}_\calb\caln_\cala$.
    \item 
Two bimodule functors 
  \be
  \label{mixed_prod}
  \ubmixt{}{}\!\Colon \calm\,{\boxtimes_\calb}\, \caln \rarr~ \cala \qquad \text{and}
  \qquad \ubmixtd{}{}\!\Colon \caln \,{\boxtimes_\cala}\, \calm \rarr~ \calb \,.
  \ee
    \item 
Two bimodule natural isomorphisms $\alpha$ and $\beta$ of the form

  \be
  \label{alpha}
  \begin{tikzcd}
  (\calm {\boxtimes_\calb} \caln) \,{\boxtimes_\cala}\, \calm
  \ar[rr,"\simeq"] \ar[d,swap,"\ubmixt{}{}\boxtimes\,\id"]
  && \calm \,{\boxtimes_\calb}\, (\caln{\boxtimes_\cala}\calm)
  \ar[d,"\id\,\boxtimes \ubmixtd{}{}"]
  \\[0.8em]
  \cala\,{\boxtimes_\cala}\,\calm \ar[dr,"~\act",""{name=U},swap]
  && \calm \,{\boxtimes_\calb}\, \calb \ar[dl,"\actr~~"]
  \ar[Rightarrow,from=1-3,to=U,swap,start anchor={[xshift=-1.3ex,yshift=-0.8ex]},
                               end anchor={[xshift=1.7ex,yshift=1.8ex]},"\alpha"]
  \\
  & \calm &
  \end{tikzcd}
  \ee
and
  \be
  \label{beta}
  \begin{tikzcd}
  (\caln {\boxtimes_\cala} \calm) \,{\boxtimes_\calb} \caln
  \ar[rr,"\simeq"]\ar[d,swap,"\ubmixtd{}{}\boxtimes\,\id"]
  && \caln \,{\boxtimes_\cala}\, (\calm{\boxtimes_\calb}\caln)
  \ar[d,"\id\,\boxtimes\ubmixt{}{}"]
  \\[0.8em]
  \calb \,{\boxtimes_\calb}\, \caln \ar[dr,"~\act",""{name=U},swap]
  && \caln \,{\boxtimes_\cala}\, \cala \ar[dl,"\actr~~"]
  \ar[Rightarrow,from=1-3,to=U,swap,start anchor={[xshift=-1.3ex,yshift=-0.8ex]},
                               end anchor={[xshift=1.7ex,yshift=1.8ex]},"\beta"]
  \\
  & \caln &
  \end{tikzcd}
  \ee
\end{enumerate}

\noindent
These data are required to fulfill the condition that the pentagon diagrams 
  \be
  \label{Morita_coh1}
  \begin{tikzcd}[column sep=1.4em]
  & (\ubmixto{m_1}{n_1}) \,{\otimes_\cala} (\ubmixto{m_2}{n_2})
  \ar[dl,"\phi_{\ubmixto{m_1}{n_1},m_2\boxtimes n_2}",swap]
  \ar[dr,"\phi'_{m_1\boxtimes n_1,\ubmixto{m_2}{n_2}}"] &
  \\
  \ubmixto{((\ubmixt{m_1}{n_1})\act m_2)}{n_2}
  && \ubmixt{m_1}{(n_1\actr (\ubmixto{m_2}{n_2}))}
  \ar[d,"\ubmixt{\id}{\,\beta_{n_1,m_2,n_2}}"]
  \\[0.7em]
  \ubmixt{(m_1\actr (\ubmixtdo{n_1}{m_2}))}{n_2}
  \ar[u,"\ubmixt{\alpha_{m_1,n_1,m_2}}{\id}"] \ar[rr,"\cong"]
  && \ubmixt{m_1}{((\ubmixtdo{n_1}{m_2})\act n_2)}
  \end{tikzcd}
  \ee
and
  \be
  \label{Morita_coh2}
  \begin{tikzcd}[column sep=1.4em]
  & (\ubmixtdo{n_1}{m_1}) \,{\otimes_\calb}\, (\ubmixtdo{n_2}{m_2})
  \ar[dl,"\psi_{\ubmixtdo{n_1}{m_1},n_2\boxtimes m_2}",swap]
  \ar[dr,"\psi'_{n_1\boxtimes m_1,\ubmixtdo{n_2}{m_2}}"] &
  \\
  \ubmixtd{((\ubmixtdo{n_1}{m_1})\act n_2)}{m_2}
  && \ubmixtd{n_1}{(m_1\actr (\ubmixtdo{n_2}{m_2}))}
  \ar[d,"\ubmixtd{\id}{\alpha_{m_1,n_2,m_2}}"]
  \\[0.7em]
  \ubmixtd{(n_1\actr (\ubmixto{m_1}{n_2}))}{m_2}
  \ar[u,"\ubmixtd{\,\beta_{n_1,m_1,n_2}}{\id}"] \ar[rr,"\cong"]
  && \ubmixtd{n_1}{((\ubmixto{m_1}{n_2})\act m_2)}
  \end{tikzcd}
  \ee
commute for all $m_1,m_2\iN\calm$ and $n_1,n_2\iN\caln$,
where $\phi$ and $\phi'$ are the bimodule structures of the functor $\ubmixt{}{}$ and $\psi$ and $\psi'$ are the bimodule structures of $\ubmixtd{}{}$.
 \\[.3em]
{\rm (ii)} 
We say that a Morita context is \emph{strict} iff $\cala$ and $\calb$ are strict tensor categories, $\calm$ and $\caln$ are strict bimodules, $\ubmixt{}{}$ and $\ubmixtd{}{}$ have strict bimodule functor structures and $\alpha$ and $\beta$ are identities.
 \\[.3em]
{\rm (iii)}
A Morita context is said to be \emph{strong} iff $\ubmixt{}{}$ and $\ubmixtd{}{}$ are equivalences. {\Color In that case, the bimodule ${}_\cala\calm_\calb$ is called \textit{invertible} and ${}_\calb\caln_\cala$ \textit{its inverse}.}
\end{defi}
\vspace{1em}
\begin{rem}
\Enumeratei
    \item 
    Via the equivalence (\ref{relative_Deligne_up}) coming from the universal property of the relative Deligne product, the bimodule functors 
\eqref{mixed_prod} in a Morita context correspond to bimodule balanced functors
  \be
  \bmixt{}{}\!\Colon\calm\Times \caln\rarr~ \cala \qquad \text{and} \qquad
  \bmixtd{}{}\!\Colon\caln\Times \calm\rarr~ \calb\;.
  \ee
As a consequence, a Morita context can be equivalently defined by replacing 
\eqref{alpha} and \eqref{beta} with bimodule balanced natural isomorphisms
  \be
  \begin{tikzcd}[column sep=3.3em,row sep=2.6em]
  \calm\Times \caln\Times \calm
  \ar["\id\times \bmixtd{}{}",r] \ar[d,"\bmixt{}{}\times \id",swap]
  & \calm\Times \calb \ar[d,"\actr~"] \ar[dl,"\alpha",Rightarrow]
  \\
  \cala \Times \calm \ar[r,swap,"\act"] & \calm
  \end{tikzcd}
  \qquad \text{and} \qquad
  \begin{tikzcd}[column sep=3.3em,row sep=2.6em]
  \caln\Times \calm\Times \caln
  \ar["\id\times \bmixt{}{}",r] \ar[d,"~\bmixtd{}{}\times \id",swap]
  & \caln\Times\cala \ar[d,"\actr~"] \ar[dl,"\beta",Rightarrow]
  \\
  \calb\Times \caln \ar[r,swap,"\act"] & \caln
  \end{tikzcd} \qquad
  \ee
obeying relations analogous to \eqref{Morita_coh1} and \eqref{Morita_coh2}.
    \item 
The four categories in a Morita context interact with each other via the tensor 
products of $\cala$ and $\calb$ and their actions on the bimodule categories
$\calm$ and $\caln$. Accordingly the functors $\bmixt{}{}$ and $\bmixtd{}{}$ 
play the role of mixed products between the categories $\calm$ and $\caln$.
\end{enumerate}
\end{rem}


\subsection{The bicategory of a Morita context}

In the spirit of \Cite{Rem.\,3.18}{Mue} the data of a Morita context form
a bicategory. The construction is as follows. Given a Morita context
$(\cala,\calb,\calm,\caln,\bmixto{}{},\bmixtdo{}{},\alpha,\beta)$, define a
bicategory $\Mor$ consisting of two objects {\Color$\{\mathbf{+},\mathbf{-}\}$ }
and Hom-categories
{\Color  
  \be
  \bearll
  \Mor(\mathbf{+},\mathbf{+}):=\cala \,, &
  \Mor(\mathbf{-},\mathbf{+}):=\calm \,,
  \Nxl6
  \Mor(\mathbf{-},\mathbf{-}):=\calb \,, \qquad &
  \Mor(\mathbf{+},\mathbf{-}):=\caln \,.
  \eear
  \ee
}  
The horizontal composition 
  \be
  \circ_{i,j,k}\Colon \Mor(\textbf{j},\textbf{k}) \Times \Mor(\textbf{i},
  \textbf{j})\rarr~ \Mor(\textbf{i},\textbf{k}) \qquad\text{for}\quad
  \textbf{i},\textbf{j},\textbf{k}\iN{\Color\{\mathbf{+},\mathbf{-}\}}
  \ee
in the bicategory $\Mor$ is given by the eight functors
  \be
  \bearll
  \otimes_\Cala\Colon\cala\Times\cala\to\cala \,, &
  \otimes_\calb\Colon\calb\Times\calb\to\calb \,,
  \Nxl5
  \act_\Calm\Colon\cala\Times\calm\to\calm \,, \qquad\quad &
  \act_\Caln\Colon\calb\Times\caln\to\caln \,,
  \Nxl5
  \actr_\Caln\Colon\caln\Times\cala\to\caln \,, &
  \actr_\Calm\Colon\calm\Times\calb\to\calm \,,
  \Nxl5
  \bmixto{}{}\Colon \calm\Times\caln\to\cala \,, &
  \bmixtdo{}{}\Colon \caln\Times\calm\to\calb \,,
  \eear
  \ee
i.e.\ by the tensor products, module actions and mixed products in the Morita context.
 \\[.2em]
The associativity constraints in $\Mor$ are natural isomorphisms
  \be
  \begin{tikzcd}[row sep=3.3em,column sep=4.3em]
  \Mor(\textbf{j},\textbf{k})\Times\Mor(\textbf{i},\textbf{j})
  \Times\Mor(\textbf{h},\textbf{i})
  \ar[r,"\id\times\circ_{h,i,j}"] \ar[d,swap,"\circ_{i,j,k}\times\id"]
  & \Mor(\textbf{j},\textbf{k})\Times\Mor(\textbf{h},\textbf{j})
  \ar[d,"\circ_{h,j,k}"]
  \ar[dl,Rightarrow,start anchor={[xshift=-3.7ex,yshift=-0.9ex]},
                    end anchor={[xshift=3.9ex,yshift=1.2ex]},"\gamma_{h,i,j,k}"]
  \\
  \Mor(\textbf{i},\textbf{k})\Times\Mor(\textbf{h},\textbf{i})
  \ar[r,swap,"\circ_{h,i,k}"]
  & \Mor(\textbf{h},\textbf{k})
  \end{tikzcd}
  \ee
for each quadruple of objects $\textbf{h},\textbf{i},\textbf{j},\textbf{k}
\iN {\Color\{\mathbf{+},\mathbf{-}\}}$. We require these sixteen constraints to be the following:
\begin{itemize} 
\addtolength\itemsep{-0.3em}
    \item 
The two associativity constraints from the tensor products of $\cala$ and $\calb$.
    \item 
The six module associativity constraints from the actions on the bimodules 
$\calm$ and $\caln$.
    \item 
Four constraints coming from the bimodule functor structures of $\bmixt{}{}$ 
and $\bmixtd{}{}$.
    \item 
The two balancings of $\bmixt{}{}$ and $\bmixtd{}{}$.
    \item 
The coherence data of the Morita context, i.e.\ the two isomorphisms $\alpha$ and $\beta$.
\end{itemize}

These data provide indeed a bicategory:

\begin{thm}\label{Morita_bicategory}
The data of a Morita context form a bicategory $\Mor$ with two objects.
\end{thm}

\begin{proof}
Clearly, the data given above are those needed for a bicategory. Thus it remains 
to check that the associators defined for $\Mor$ obey the relevant axiom in a 
bicategory, i.e.\ that the diagram
  \be
  \begin{tikzcd}[column sep=1.9em]
  & (f_{k,l}\cir f_{j,k})\cir (f_{i,j}\cir f_{h,i})
  \ar[dl,"\gamma_{h,i,j,l}",swap] \ar[dr,"\gamma_{h,k,j,l}"] &
  \\
  ((f_{k,l}\cir f_{j,k})\cir f_{i,j})\cir f_{h,i}
  \ar[d,swap,"\gamma_{i,j,k,l}\circ\,\id"]
  && f_{k,l}\cir (f_{j,k}\cir (f_{i,j}\cir f_{h,i}))
  \ar[d,"\id\,\circ\,\gamma_{h,i,j,k}"]
  \\[.9em]
  (f_{k,l}\cir (f_{j,k}\cir f_{i,j}))\cir f_{h,i} \ar[rr,"\gamma_{h,i,k,l}",swap]
  && f_{k,l}\cir ((f_{j,k}\cir f_{i,j})\cir f_{h,i})
  \end{tikzcd}
  \ee
commutes for all $\textbf{h},\textbf{i},\textbf{j},\textbf{k},\textbf{l}
\iN{\Color\{\mathbf{+},\mathbf{-}\}}$ and every quadruple of 1-morphisms
$f_{k,l}\iN\Mor(\textbf{k},\textbf{l})$,
$f_{j,k}\iN\Mor(\textbf{j},\textbf{k})$, $f_{i,j}\iN\Mor(\textbf{i},\textbf{j})$
and $f_{h,i}\iN\Mor(\textbf{h},\textbf{i})$. It can be verified
that these conditions correspond to the following thirty-two diagrams:
\Enumerate
\addtolength\itemsep{-0.5em}
    \item 
Two pentagons obeyed by the associativity constraints of $\cala$ and $\calb$.
    \item 
Four pentagons fulfilled by the left and right module constraints of $\calm$ and $\caln$.
    \item 
Four pentagons fulfilled by the middle associativity constraints of the bimodules $\calm$ and $\caln$.
    \item 
Six compatibility conditions on the bimodule functor structures of $\bmixt{}{}$ and $\bmixtd{}{}$.
    \item 
The two diagrams that describe the conditions on the balancings of $\bmixt{}{}$ and 
$\bmixtd{}{}$.
    \item 
Four conditions corresponding to the compatibility between the bimodule structures 
and the balancings of $\bmixt{}{}$ and $\bmixtd{}{}$.
    \item 
Four compatibility conditions characterizing $\alpha$ and $\beta$ as bimodule 
natural transformations.
    \item 
Four conditions defining $\alpha$ and $\beta$ as balanced natural transformations.
    \item 
The diagrams \eqref{Morita_coh1} and \eqref{Morita_coh2}, which exhibit the 
compatibility between the coherence data $\alpha$ and $\beta$ of the Morita context.
\end{enumerate}
Thus all the conditions are satisfied by construction.
\end{proof}
\vspace{1em}
\begin{rem}~
\begin{enumerate}[(i)]
  \item 
The notion of a Morita context as a bicategory studied in \cite{Mue} is an 
instance of Theorem \ref{Morita_bicategory} where a (Frobenius) algebra is 
chosen to realize the bimodules in the Morita context.
  \item
Theorem \ref{Morita_bicategory} justifies Definition \ref{Morita_context}. Coherence 
results for bicategories, such as Corollary 2.7 of \cite{Gu}, imply that all 
paths between any two possible bracketings of a product of multiple objects in 
the Morita context through associativity constraints are the same isomorphism.
\end{enumerate}
\end{rem}

As Theorem \ref{Morita_bicategory} indicates, the bicategorical setting 
emerges as the natural framework for studying Morita contexts. Let us recall 
a few pertinent notions from this  setting. A \emph{pseudo-func\-tor} 
$U\colon \mathscr{F}\Rarr~\mathscr{G}$ between bicategories consists of 
assignments at the level of objects and at the level of Hom-categories 
together with invertible $2$-morphisms that witness their compatibility 
with horizontal composition and which obey a pentagon axiom. A 
\emph{pseudo-natural equivalence} $\eta \colon U \,{\xRightarrow{\,\simeq~}}\,V$ 
between pseudo-functors amounts to an invertible $1$-morphism 
$\eta_x\colon U(x) \,{\xrightarrow{\,\simeq~}}\, V(x)$ for every object 
$x\iN\mathscr{F}$ and an invertible $2$-morphism $\eta_f\colon \eta_y\cir U(f)
\,{\xRightarrow{\,\simeq~}}\, V(f)\cir\eta_x$ 
for every $1$-mor\-phism $f\colon x\To y$,
respecting horizontal composition. 
Complete definitions can for instance be found in \Cite{App.\,A.2}{schaum2}.

\begin{defi}
An \emph{equivalence of Morita contexts} is a pseudo-equivalence
between their bicategories.
\end{defi}

\begin{cor}[Coherence for Morita contexts]
Every Morita context is equivalent to a strict Morita context.
\end{cor}

\begin{proof}
Using that every bicategory is equivalent to a strict $2$-category 
\Cite{Cor.\,2.7}{Gu}, this follows directly from Theorem \ref{Morita_bicategory}.
\end{proof}


\subsection{The Morita context derived from an exact module category}

Let $\cala$ be a finite tensor category and $\calm$ an exact $\cala$-module. There
is a strong Morita context associated to ${}_\cala\calm$. To see this, first 
recall that the category $\dualcat$ of module endofunctors has the structure
of a tensor category and that the evaluation of a functor on an object turns $
\calm$ into an exact $\dualcat$-module category. More specifically, hereby 
$\calm$ becomes an $(\cala, \overline{\dualcat})$-bimodule category.

As described in \eqref{left_module_functors}, the category $\FunM$ of module functors
is naturally endowed with the structure of an $(\overline{\dualcat},\cala)$-bimodule 
category, via the actions
  \be
  F\actF H := H\cir F \qquad \text{and} \qquad H\actrF a := H(-)\ot a = (-\ot a)\cir H
  \ee
for $H\iN \FunM$, $F\iN \overline{\dualcat}$ and $a\iN\cala$.

Having obtained two tensor categories and two bimodules categories from the 
exact $\cala$-module $\calm$, we can proceed to define the mixed products in the
Morita context:

\begin{defi}
The \emph{$\overline{\dualcat}$-valued mixed product} of $\calm$ is the functor
  \be
  \begin{aligned}
  \qquad \mixtd{}{}\!\Colon \FunM\times \calm & \rarr~ \overline{\dualcat} \,,
  \Nxl2
  (H,m) & \xmapsto{~~~} H(-)\act m = (-\act m)\cir H \,,
  \end{aligned}
  \ee
which is a special case of \eqref{Deligne_left_eq}. 
The \emph{$\cala$-valued mixed product} of $\calm$ is the functor given by evaluation
  \be
  \begin{aligned}
  \mixt{}{}\Colon \calm\times \FunM & \rarr~ \cala \,,
  \Nxl2
  (m,H) & \xmapsto{~~~} H(m) \,. \qquad
  \end{aligned}
  \ee
  \end{defi}

\begin{lem}\label{mixed_bal_bimod}
The mixed product $\mixtd{}{}$ is an $\cala$-balanced $\overline{\dualcat}$-bimodule functor, and the mixed product $\mixt{}{}$ is an $\overline{\dualcat}$-balanced $\cala$-bimodule functor. 
\end{lem}

\begin{proof}
First notice that the functor $\mixtdo{}{}$ comes with an $\cala$-balancing given
by the module associativity constraints of ${}_\cala\calm$:
  \be
  \mixtd{(H\actrF a)}{m}=(H(-)\ot a)\act m
  \rarr\cong H(-)\act(a\act m)=\mixtd{H}{(a\act m)} \,.
  \ee
Furthermore, the identities 
  \be
  F\act(\mixtd{H}{m})=(-\act m)\circ H\circ F = \mixtd{(F\actF H)}{m}
  \ee
and the isomorphisms 
  \be
  (\mixtdo{H}{m})\actr F= F(H(-)\act m) \rarr\cong H(-)\act F(m)=\mixtd{H}{F(m)}
  \ee
that come from the module functor structure of $F$ endow $\mixtdo{}{}$ with the
structure of an $\overline{\dualcat}$-bimodule functor. These are compatible with
the balancing because the module functor structure of $F$ is compatible with the associativity constraints of ${}_\cala\calm$.
 \\[2pt]
Similarly, for the functor $\mixto{}{}$ the identity morphisms
$\mixt{(m\actr F)}{H} \eq H\cir F (m) \eq \mixt{m}{(F\actF H)}$ provide an 
$\overline{\dualcat}$-balancing. Moreover, there is a natural $\cala$-bimodule 
functor structure on $\mixto{}{}$, namely
  \be
  a\ot(\mixt{m}{H})=a\ot H(m) \cong H(a\act m) = \mixt{(a\act m)}{H}
  \ee
and
  \be
  (\mixt{m}{H})\ot a= H(m)\ot a=\mixt{m}{(H\actrF a)}
  \ee
given by the module functor structure on $H$ and the identity morphisms, respectively.
\end{proof}

\begin{thm}\label{Morita_context_exact_module}
Let $\calm$ be an exact module category over a finite tensor category $\cala$. Then the data 
  \be\label{eq:Morita_context_exact_module}
  \left(\cala,\,\overline{\dualcat},\,\calm,\,\FunM,\,\mixt{}{},\,\mixtd{}{}\right)
  \ee 
form a Morita context.

\begin{proof}
Notice that the diagram
  \be
  \begin{tikzcd}[column sep=3.2em,row sep=2.4em]
  \calm\Times \FunM\Times \calm \ar["\id\times \mixtd{}{}",r]
  \ar[d,"\mixt{}{}\times \id",swap]
  & \calm\Times \overline{\dualcat} \ar[d,"\actr"]
  \\
  \cala \Times \calm \ar[r,swap,"\act"] & \calm
  \end{tikzcd}
  \ee
commutes strictly owing to $(\mixt{m}{H})\act n \,{=}\, H(m)\act n \,{=}\, m\actr (\mixtd{H}{n})$, 
and thus the identities serve as the bimodule natural isomorphism (\ref{alpha}). On the other hand, the diagram
  \be
  \begin{tikzcd}[column sep=3.2em,row sep=2.4em]
  \FunM\Times \calm\Times \FunM \ar["\id\times \mixt{}{}",r]
  \ar[d,"\mixtd{}{}\times \id",swap]
  & \FunM\Times\cala\ar[d,"\actrF"]
  \\
  \overline{\dualcat} \Times \FunM\ar[r,swap,"\actF"] & \FunM
  \end{tikzcd}
  \ee
commutes up to the isomorphism 
  \be
  (\mixtd{H_1}{m})\actF H_2 = H_2(H_1(-)\act m) \cong H_1(-)\ot H_2(m)
  = H_1\actrF (\mixt{m}{H_2})
  \ee
that comes from the module functor structure of $H_2$. One can directly verify that 
the associated conditions (\ref{Morita_coh1}) and (\ref{Morita_coh2}) are satisfied.
\end{proof}
\end{thm}

\begin{rem}\label{strong}
As we will see in Proposition \ref{Strong_Morita_context_module}, the Morita context
derived from an exact module category $\calm$ is in fact a strong Morita context.
\end{rem}


\subsection{Characterization of strong Morita contexts}

Exact module categories provide examples of strong Morita contexts, as stated 
in Remark \ref{strong}. It turns out that, conversely, every strong Morita 
context in the sense of Definition \ref{Morita_context}(iii) is equivalent to 
the Morita context of an exact module category.

\begin{thm}\label{strong_Morita}
Let $\cala$ and $\calb$ be finite tensor categories and ${}_\cala\calm_\calb$ 
and ${}_\calb\caln_\cala$ be finite bimodule categories. Consider a strong 
Morita context $(\cala,\calb,\calm,\caln,\bmixt{}{},\bmixtd{}{},\alpha,\beta)$. 
Then the following statements hold:
\begin{enumerate}[\rm (i)]
    \item
$\calm$ and $\caln$ are exact indecomposable bimodule categories. 
    \item
The assignments
  \be
  \begin{aligned}
  \caln & \xrightarrow{\;\simeq\;} \FunM \qquad\text{and}\qquad
  & \caln & \xrightarrow{\;\simeq\;}\Fun_\calb(\calm,\calb)
  \\[2pt]
  n & \longmapsto \bmixt{-}{n} & n & \longmapsto \bmixtd{n}{-}
  \end{aligned}
  \ee
are equivalences of $(\calb,\cala)$-bimodule categories.
    \item
The assignments
  \be
  \begin{aligned}\label{MRex}
  \calm & \xrightarrow{\;\simeq\;} \Fun_\calb(\caln,\calb) \qquad\text{and}\qquad
  & \calm & \xrightarrow{\;\simeq\;}\Fun_\cala(\caln,\cala)
  \\[2pt]
  m & \longmapsto \bmixtd{-}{m} & m &\longmapsto \bmixt{m}{-}
  \end{aligned}
  \ee
are equivalences of $(\cala,\calb)$-bimodule categories.
    \item
The assignments
  \be
  \begin{aligned}
  R_\calm\Colon\calb & \xrightarrow{\;\simeq\;} \Fun_\cala(\calm,\calm)
  \qquad\text{and}\qquad
  & L_\caln\Colon\calb & \xrightarrow{\;\simeq\;}\Fun_\cala(\caln,\caln)
  \\[2pt]
  b & \longmapsto -\actr b & b & \longmapsto b\act -
  \end{aligned}
  \ee
are equivalences of tensor categories and of $\calb$-bimodule categories.
  \item
The assignments
  \be
  \begin{aligned}
  R_\caln\Colon \cala & \xrightarrow{\;\simeq\;} \Fun_\calb(\caln,\caln)
  \qquad\text{and}\qquad
  & L_\calm\Colon\cala & \xrightarrow{\;\simeq\;} \Fun_\calb(\calm,\calm)
  \\[2pt]
  a & \longmapsto -\actr a & a & \longmapsto a\act -
  \end{aligned}
  \ee
are equivalences of tensor categories and of $\cala$-bimodule categories.
    \item
The commutativity of the diagrams
  \be
  \label{strong_compatibility}
  \hspace*{-1.8em}
  \begin{tikzcd}[row sep=2.8em]
  \calm\Times\caln \ar[r,"\bmixt{}{}"] \ar[d,swap,"{(m,n)\mapsto(m,\bmixto{-}{n})}"]
  & \cala\ar[d,"\id_\cala"]
  \\
  \calm\Times\FunM \ar[r,swap,"\mixt{}{}"] & \cala
  \end{tikzcd}
  \quad\text{and}\quad
  \begin{tikzcd}[row sep=2.8em]
  \caln\times\calm \ar[r,"\bmixtdo{}{}"] \ar[d,swap,"{(n,m)\mapsto(\bmixto{-}{n},m)}"]
  & \calb\ar[d,"R_\calm"]
  \\
  \FunM\Times\calm \ar[r,swap,"\mixtd{}{}"] & \overline{\dualcat}
  \end{tikzcd}
  \quad
  \ee
is witnessed by the identity and by $\alpha_{-,n,m} \colon {-}\actr (\bmixtdo{n}{m})
\Rarr\cong (\bmixto{-}{n})\act m$, respectively.
\end{enumerate}
\end{thm}

\begin{proof}
First notice that (i) follows from (iv) and (v). For instance, assuming that
$R_\calm$ is an equivalence, rigidity of $\calb$ implies that every $\cala$-module
endofunctor of $\calm$ is exact and thus ${}_\cala\calm$ is an exact module
\Cite{Prop.\,7.9.7(2)}{EGno}. Also, since the monoidal unit of $\calb$ is simple, then $\calm$ is an indecomposable $\cala$-module.
To prove the remaining statements we consider categories of right exact module
functors. Once the statements are checked, exactness of the bimodules $\calm$ 
and $\caln$ will follow, and thus that every module functor is exact. 
 \\[.2em]
We now prove (ii). That $\bmixto{}{}\colon \calm\Times\caln \To \cala$ is a
balanced bimodule functor implies that $\bmixto{-}{n}$ has an $\cala$-module 
structure and that the functor
  \be
  \label{NRex}
  \begin{aligned}
  \caln & \longrightarrow \Funre_\cala(\calm,\cala)
  \\
  n & \longmapsto\bmixt{-}{n}
  \end{aligned}
  \ee
is endowed with a $(\calb,\cala)$-bimodule functor structure. On the other hand,
we have an equivalence 
$\underline{\bmixtd{}{}}\colon \caln\,{\boxtimes_\cala}\,\calm \,{\simeq}\,\calb$,
and in view of Lemma \ref{Deligne_objects} there is an object 
$\text{colim}_{i\in I}\, n_i\boti m_i \iN \caln\,{\boxtimes_\cala}\,\calm$
such that $\one_\calb \,{\cong}\, \underline{\bmixtd{}{}}(\text{colim}_{i\in I}\,
n_i{\boxtimes}m_i) \eq \text{colim}_{i\in I}\,\bmixtdo{n_i}{m_i}$. Then the functor
  \be
  \begin{aligned}
  \Funre_\cala(\calm,\cala) & \longrightarrow\caln
  \\[2pt]
  H & \longmapsto\text{colim}_{i\in I}\, n_i\Actr H(m_i)
  \end{aligned}
  \ee
defines a quasi-inverse to the bimodule functor (\ref{NRex}). Indeed, given
$n\iN\caln$ there is a natural isomorphism
  \be
  \text{colim}_{i\in I}\,n_i\Actr (\bmixto{m_i}{n}) \xrightarrow[\cong]{~\beta~}
  \text{colim}_{i\in I}\,(\bmixtdo{n_i}{m_i})\act n \cong \textbf{1}_\calb\act n
  = n \,.
  \ee
Similarly, for a module functor $H\iN \Funre_\cala(\calm,\cala)$ and an object
$n\in\caln$ we have a natural isomorphism
  \be
  \begin{aligned}
  \text{colim}_{i\in I}\,\bmixt{n}{(n_i\Actr H(m_i))}
  & \cong\text{colim}_{i\in I}\, (\bmixto{n}{n_i})\ot H(m_i)
  \cong \text{colim}_{i\in I}\, H((\bmixto{n}{n_i})\act m_i)
  \\[2pt]
  & \xrightarrow[\cong]{~\alpha~} \text{colim}_{i\in I}\,H(n\actr(\bmixtdo{n_i}{m_i}))
  \cong H(n\Actr \one_\calb) = H(n) \,,
  \end{aligned}
  \ee
where the first isomorphism is the right $\cala$-module structure on $\bmixto{}{}$,
the second isomorphism is the module structure of $H$, and the last isomorphism 
after $\alpha$ expresses the fact that the colimit is preserved since $H$ is right
exact, the action functor is exact and $\underline{\bmixtd{}{}}$ is an equivalence. 
The proof for the functor
  \be
  \begin{aligned}
  \caln & \longrightarrow\Funre_\calb(\calm,\calb)
  \\
  n & \longmapsto\bmixtd{n}{-}
  \end{aligned}
  \ee
is obtained by merely exchanging the roles of $\bmixt{}{}$ and $\bmixtd{}{}$. 
Similarly, the assertion (iii) follows by symmetry.
 \\[.2em]
Next we prove (iv). We have a $\calb$-bimodule equivalence
  \be 
  \phi\Colon\calb \xrightarrow{\;\underline{\bmixtd{}{}}^{-1}\;}
  \caln \,{\boxtimes_\cala}\,\calm \xrightarrow{\,(\ref{NRex})\;}
  \Funre_\cala(\calm,\cala) \,{\boxtimes_\cala}\,\calm
  \xrightarrow{\,(\ref{Deligne_left_eq})\;} \Funre_\cala(\calm,\calm) \,.
  \ee
We write $\phi(\one_\calb) \,{=:}\, I$ and follow the argument given in
\Cite{Prop.\,4.2}{eno}. The bimodule structure on $\phi$ provides natural isomorphisms
  \be
  \label{phi_I}
  \phi(b) \cong b\actF I = I\cir (-\actr b) \qquad{\rm and}\qquad
  \phi(b) \cong I\actrF b = (-\actr b)\cir I
  \ee
for $b\iN\calb$. Since $\phi$ is an equivalence, there exists an object $B\iN\calb$
such that $\id_\calm \,{\cong}\, I\cir (-\actr B)\cong(-\actr B)\cir I$, which
means that $I$ is invertible and therefore the endofunctor 
$(-\cir I) \colon \Funre_\cala(\calm,\calm) 
       $\linebreak[0]$
       \to \Funre_\cala(\calm,\calm)$ 
is an equivalence. From \eqref{phi_I} we have in particular 
$\phi \,{\cong}\, (-\cir I)\cir R_\calm$, and thus $R_\calm$ must be an equivalence
as well. The same type of argument can be applied to the bimodule equivalence
  \be 
  \calb \xrightarrow{\;\underline{\bmixtd{}{}}^{-1}\;}
  \caln\,{\boxtimes_\cala}\,\calm \xrightarrow{\,(\ref{MRex})\;}
  \caln\,{\boxtimes_\cala}\,\Funre_\cala(\caln,\cala)
  \xrightarrow{\,(\ref{Deligne_right_eq})\;} \Funre_\cala(\caln,\caln) \,,
  \ee
leading to the second part of the statement. And similarly assertion (v)
follows by symmetry. Finally, (vi) holds by direct calculation.
\end{proof}

\begin{rem}
Theorem \ref{strong_Morita} implies in particular that every strong Morita context
is equivalent to the Morita context of an exact module category. More explicitly,
together with the isomorphisms \eqref{strong_compatibility} the equivalences
$\caln\xsimeq \FunM$ and $R_\calm\colon \calb\xsimeq \overline{\dualcat}$ 
furnish an equivalence
  \be
  \left( \cala,\calb,\calm,\caln,\bmixt{}{},\bmixtd{}{} \right) \xsimeq
  \big( \cala,\,\overline{\dualcat},\,\calm,\,\FunM,\,\mixto{}{},\,\mixtdo{}{}\big)
  \ee
between Morita contexts, i.e.\ a pseudo-equivalence between the associated bicategories.
\end{rem}


\section{Dualities in a Morita context}
\label{dualities_Morita}

There is a notion of dualities for a $1$-morphism in a bicategory $\mathscr{F}$,
see for instance \Cite{App.\,A.3}{schaum2}: A \emph{right dual $($or right adjoint$\,)$}
to a $1$-morphism $a\iN\mathscr{F}(x,y)$ consists of a $1$-morphism $a^\vee\iN\mathscr{F}(y,x)$
and $2$-morphisms $\textbf{1}_y\,{\Rightarrow}\, a\cir a^\vee$ and 
$a^\vee\cir a\,{\Rightarrow}\, \textbf{1}_x$ that fulfill the appropriate snake 
relations; left duals are defined similarly. A bicategory for which every $1$-morphism 
has both duals is said to be a \emph{bicategory with dualities}. It is worth noting that the existence of duals is a property, rather than extra structure, of a bicategory. In the present section 
we explore this notion of dualities for the bicategory $\Mor$ described in Theorem \ref{Morita_bicategory}, that is associated to the Morita
context given by an exact module category, cf. Theorem \ref{Morita_context_exact_module}.

\subsection{Existence of dualities in the bicategory $\Mor$}\label{existence_dualities}

For $\calm$ an exact module category over a finite tensor category $\cala$ we
consider its (strong) Morita context $\left(\cala,\,\overline{\dualcat},\,\calm,\,\FunM,\,
\mixto{}{},\,\mixtdo{}{}\right)$. Since the tensor categories $\cala$ and 
$\overline{\dualcat}$ are rigid, their objects come with dualities. 
Also, since $\calm$ is exact, an object $H\iN\FunM$ has left and right adjoints.
We now show that in the whole Morita context every object is equipped with dualities, in 
such a way that the associated bicategory $\Mor$ is a bicategory with dualities.

\begin{defi} \label{def:duals}
Let $\calm$ be an exact module category over a finite tensor category $\cala$ and
$\left(\cala,\,\overline{\dualcat},\,\calm,\,\FunM,\,
\mixto{}{},\,\mixtdo{}{}\right)$ the associated Morita context.
 \\[.2em]
(i) The \emph{right dual} of an object $m\iN \calm$ is the triple $(m^\vee,\underline{\rm ev}_{m},\underline{\rm coev}_{m})$ consisting of the following data:
First, the $\cala$-module functor 
  \be
  m^\vee := \iHom_\calm^\cala(m,-) ~\in \FunM \,.
  \ee
Second, the evaluation morphisms with which the internal Hom comes naturally endowed,
which form a module natural transformation
  \be
  \underline{\rm ev}_{m}\Colon \mixtd{m^\vee}{m} = \iHom_\calm(m,-)\act m
  \xRightarrow{\quad} \id_\calm \,,
  \label{eq:evr_m}
  \ee
defined as the counit of the adjunction $(-\Act m) \,{\dashv}\, \iHom_\calm(m,-)$.
And third, the coevaluation morphism 
  \be
  \underline{\rm coev}_{m}\Colon \un \rarr~ \iHom_\calm(m,m)=\mixt{m}{m^\vee} ,
  \label{eq:coevr_m}
  \ee 
which is defined as the component of the unit of the adjunction 
$(-\Act m) \,{\dashv}\, \iHom_\calm(m,-)$ corresponding to $\un$.
 \\[.2em]
(ii) The \emph{left dual} of an object $m\iN \calm$ is the triple $(\Vee m,\overline{\rm ev}_m,\overline{\rm coev}_m)$ consisting of the following data:
First, the $\cala$-module functor 
  \be
  \Vee m := \icoHom_\calm^\cala(m,-) ~\in \FunM \,.
  \ee
Second, the evaluation morphism
  \be
  \overline{\rm ev}_m\Colon \mixt{m}{\Vee m} = \icoHom_\calm(m,m) \rarr~ \un \,,
  \label{eq:evl_m}
  \ee
defined as the component of the counit of the adjunction 
$\icoHom_\calm(m,-) \,{\dashv}\, (-\Act m)$ that corresponds to $\un$.
And third, the coevaluation morphism
  \be
  \overline{\rm coev}_m\Colon \id_\calm \xRightarrow{\quad} \icoHom_\calm(m,-)\act m
  = \mixtd{\Vee m}{m} \,,
  \label{eq:coevl_m}
  \ee
defined as the unit of the adjunction $\icoHom_\calm(m,-) \,{\dashv}\,(-\Act m)$.
 \\[.2em]
(iii) The \emph{right dual} of an object $H \iN \FunM$ is the triple $(H^\vee,{\rm ev}_{m},{\rm coev}_{m})$ consisting of: The object
  \be
  H^\vee := H\ra(\un) ~\in\calm \,,
  \ee
together with an evaluation morphism 
  \be
  {\rm ev}_H\Colon \mixt{H^\vee}{H}=H\circ H\ra(\un) \rarr~ \un
  \label{eq:evr_H}
  \ee
given by the component of the counit of the adjunction $H \,{\dashv}\, H\ra$ 
corresponding to $\un$, and with a coevaluation morphism 
  \be
  {\rm coev}_H\Colon \id_\calm \xRightarrow{\quad} H\ra\circ H\cong H(-)\act H\ra(\un)
  = \mixtd{H}{H^\vee}
  \label{eq:coevr_H}
  \ee
given by the unit of the adjunction $H \,{\dashv}\, H\ra$.
 \\[.2em]
(iv) The \emph{left dual} of an object $H\iN\FunM$ is the triple $(\Vee H,\widetilde{\rm ev}_{m},\widetilde{\rm coev}_{m})$ consisting of: The object
  \be
  \Vee H := H\la(\un) ~\in\calm
  \ee
together with an evaluation morphism
  \be
  \widetilde{\rm ev}_H\Colon \mixtd{H}{\Vee H} = H(-)\act H\la(\un) \cong H\la\circ H
  \xRightarrow{\quad} \id_\calm
  \label{eq:evl_H}
  \ee
provided by the adjunction $H\la \,{\dashv}\, H$ and with a coevaluation morphism 
  \be
  \widetilde{\rm coev}_H\Colon \un\rarr~ H\cir H\la(\un) = \mixtd{\Vee H}{H}
  \label{eq:coevl_H}
  \ee
being the component of the unit of the adjunction $H\la \,{\dashv}\, H$ corresponding 
to $\un$.
\end{defi}

Since they are the unit and counit of an adjunction, the morphisms \eqref{eq:evr_m} 
and \eqref{eq:coevr_m}, respectively \eqref{eq:evl_m} and \eqref{eq:coevl_m}, obey 
the snake relations for a right and left duality. For the same reason,
the morphisms \eqref{eq:evr_H} and \eqref{eq:coevr_H}, respectively \eqref{eq:evl_H}
and \eqref{eq:coevl_H}, obey the relevant snake relations as well.  
In summary, the left and right duals introduced in Definition \ref{def:duals}
indeed provide proper dualities, so that we have
determined dualities for all $1$-morphisms in the bicategory $\Mor$. 

We have thus established:

\begin{thm}\label{Morita_bicategory_dualities}
Let $\cala$ be a finite tensor category and $\calm$ be an exact $\cala$-module category. 
The associated bicategory $\Mor$ is a bicategory with dualities.
\end{thm}


\subsection{Properties of duals in a Morita context}

Some further properties familiar from the duals in tensor categories are again
fulfilled by duals in a Morita context. For instance, an object serves as left (right)
dual of its right (left) dual, and the dual of a product is the product of the duals 
in the reversed order up to isomorphism.

\begin{prop} \label{dual_properties}
There are natural isomorphisms
  \be
  \begin{array}{rlrl}
  {\rm (i)} & \Vee(m^\vee) \cong m\cong (\Vee m)^\vee \,, \qquad\quad
  & {\rm (v)} & (F\actF H)^\vee\cong H^\vee\actr F^\vee \,,
  \Nxl6
  {\rm (ii)} & \Vee(H^\vee) \cong H\cong (\Vee H)^\vee \,, 
  & {\rm (vi)} & (H\actrF a)^\vee\cong a^\vee\act H^\vee \,,
  \Nxl6
  {\rm (iii)} & (a\act m)^\vee \cong m^\vee\actrF a^\vee \,,
  & {\rm (vii)} & (\mixt{m}{H})^\vee\cong \mixt{H^\vee}{m^\vee} \,,
  \Nxl6
  {\rm (iv)} & (m\actr F)^\vee \cong F^\vee\actF m^\vee \,,
  & {\rm (viii)} & (\mixtd{H}{m})^\vee\cong \mixtd{m^\vee}{H^\vee} 
  \end{array}
  \ee
for $a\iN\cala$, $m\iN\calm$, $H\iN\FunM$ and $F\iN\overline{\dualcat}$. Analogous
relations are valid for left duals. 
\end{prop}

\begin{proof}
All these isomorphisms follow immediately from the definitions and Theorem 
\ref{Morita_bicategory_dualities}. We provide nonetheless the proof for some of the
statements. For instance, the definition of duals directly implies (i):
  \be
  \Vee(m^\vee) = \iHom_\calm^\cala(m,-)\la(\un) = \un\act m \cong m \,.
  \ee
Similarly, (ii) follows with the help of Lemma \ref{adjoints_ihom}:
  \be
  (\Vee H)^\vee = \iHom_\calm^\cala(H\la(\un),-) \cong \iHom_\cala^\cala(\un,H(-))
  = H(-)\ot\un^\vee \cong H \,.
  \ee
The isomorphism in (iii) corresponds to the module functor structure on the 
internal Hom given by (\ref{ihom_rmod}). Finally, (viii) comes from the composition
  \be
  \begin{aligned}
  (\mixtd{H}{m})^\vee &= [(-\act m)\circ H]\ra = H\ra\circ\iHom_\calm^\cala(m,-)
  \Nxl1
  & \cong H\ra\left(\iHom_\calm^\cala(m,-)\ot\un\right)
  \Nxl1
  &\cong \iHom_\calm^\cala(m,-)\act H\ra(\un)=\mixtd{m^\vee}{H^\vee} ,
  \end{aligned} 
  \ee
where we first identify the adjoint of a composite with the composition of the 
adjoints in reversed order, and where the last isomorphism corresponds to the 
module functor structure of $H\ra$.
\end{proof}

Dualities in a tensor category provide an equivalence between its opposite category 
and its monoidal opposite. Analogously, the duals we just defined for a module 
category exhibit how the bimodule category $\FunM$ plays the role of an opposite 
to $\calm$.

\begin{prop}
Let $\calm$ be an exact $\cala$-module category. The dualities on the bicategory $\Mor$ 
induce equivalences
  \be
  (-)^\vee\Colon\calm\xsimeq\FunM^\# \qquad\text{and}\qquad \Vee(-)\Colon\calm\xsimeq\H\FunM
  \ee
of $(\cala,\overline{\dualcat})$-bimodule categories, and equivalences
  \be
  (-)^\vee\Colon\FunM\xsimeq\calm^\# \qquad\text{and}\qquad \Vee(-)\Colon\FunM\xsimeq\HH\calm
  \ee
of $(\overline{\dualcat},\cala)$-bimodule categories.
\end{prop}

\begin{proof}
Proposition \ref{dual_properties} shows that right and left duals are mutual 
quasi-inverses, and the isomorphisms from Proposition \ref{dual_properties}(iii) and 
\ref{dual_properties}(iv) endow the duality functor with a bimodule structure. 
The statements for the remaining equivalences follow by analogous reasoning.
\end{proof}



There are further relations involving the dualities which constitute a duality 
calculus in the Morita context, a powerful tool for performing calculations
such as the computation of relative Serre functors.

\begin{rem}\label{Homs_adj}~{}
\begin{enumerate}[$($i$)$]
    \item 
Recall that for the regular module ${}_\cala\cala$ we have
  \be
  \iHom_\cala^\cala(a,b) = b\ot a^\vee \qquad\text{and}\qquad
  \icoHom_\cala^\cala(a,b) = b\ot \Vee a \,.
  \ee
for $a,b\in\cala$. Similarly it follows from the definition of the mixed products and of the dualities that
  \be
  \iHom_\calm^\cala(m,n) =\mixt{n}{m^\vee} \qquad\text{and}\qquad
  \icoHom_\calm^\cala(m,n)=\mixt{n}{\Vee m} \,.
  \ee
for $m,n\in\calm$.  
    \item 
Acting with dualities on the bimodules in the Morita context leads to the adjunctions
  \be
  (a^\vee\act-)\dashv (a\act-) \dashv (\Vee a\act -) \qquad\text{and}\qquad
  (-\actr F\la)\dashv (-\actr F) \dashv (-\actr F\ra)
  \ee
as well as
  \be
  (F\ra\actF-)\dashv (F\actF -) \dashv (F\la\actF-) \qquad\text{and}\qquad
  (-\actrF \Vee a)\dashv (-\actrF a)\dashv (-\actrF a^\vee)
  \ee
for $a\iN\cala$ and $F\iN\dualcat$.
\end{enumerate}
\end{rem}

It turns out that the relations given in Remark \ref{Homs_adj} can be extended to 
the entire Morita context:

\begin{prop}\label{Morita_adjunctions}
There are natural isomorphisms
  \be
  \begin{array}{rl}
  {\rm (i)}   & \HomM(a\act m, n)\cong \HomA(a,\mixt{n}{m^\vee}) \,,
  \Nxl6
  {\rm (ii)}  & \HomM(n, a\act m)\cong \HomA(\mixt{n}{\Vee m}, a) \,,
  \Nxl6
  {\rm (iii)} & \HomM(m\actr F,n)\cong \HomDual(F,\mixtd{\Vee m}{n}) \,,
  \Nxl6
  {\rm (iv)}  & \HomM(n, m\actr F)\cong \HomDual(\mixtd{m^\vee}{n},F) \,,
  \Nxl6
  {\rm (v)}   & \HomF(F\actF H_1,H_2)\cong \HomDual(F,\mixtd{H_2}{H_1^\vee}) \,,
  \Nxl6
  {\rm (vi)}  & \HomF(H_2,F\actF H_1)\cong \HomDual(\mixtd{H_2}{\Vee H_1},F) \,,
  \Nxl6
  {\rm (vii)} & \HomF(H_1\actrF a,H_2)\cong \HomA(a,\mixt{\Vee H_1}{H_2}) \,,
  \Nxl6
  {\rm (viii)}& \HomF(H_2,H_1\actrF a)\cong \HomA(\mixt{H_1^\vee}{H_2},a) 
  \end{array}
  \ee
for $a\iN\cala$, $m,n\iN\calm$, $H_1,H_2\iN\FunM$ and $F\iN\overline{\dualcat}$.
\end{prop}

\begin{proof}
(i) and (ii) follow directly from the definition of the duals and the defining properties of
the internal Hom and coHom.
The bijection 
  \be
  \label{ad_iii}
  \HomM(F(m),n)\longrightarrow \text{Nat}_{\rm mod}(F,\icoHom(m,-)\act n)
  \ee
in (iii) is an arrow assigning to $f\colon F(m)\To n$ a module natural transformation 
whose component at $l\iN\calm$ is given by the composition
  \be
  F(l) \rarr{F(\overline{\rm coev})}F(\icoHom(m,l)\act m) \cong \icoHom(m,l)\act F(m)
  \rarr{\id\act f}\icoHom(m,l)\act n \,.
  \ee
The inverse of (\ref{ad_iii}) is given by the assignment
  \be
  \eta ~\xmapsto{\quad}~ F(m) \rarr{\eta_m}\icoHom(m,m)\act n\rarr{\overline{\rm ev}_m\act \id_n}n
  \ee
for $\eta \iN \text{Nat}_{\rm mod}(F,\icoHom(m,-)\act n)\eq\HomDual(F,\mixtd{\Vee m}{n})$.
The bijection in (iv) is defined in a similar fashion. To prove (v) consider 
the bijection

  \be\label{ad_v}
  \text{Nat}_{\rm mod}(H_1 \cir F,H_2) \rarr~ \text{Nat}_{\rm mod}(F, H_1\ra\cir H_2)
  \ee
which assigns to $\eta\colon H_1\cir F\,{\xRightarrow~}\, H_2$ the natural transformation
  \be
  F\xRightarrow{\quad} H_1\ra\cir H_1\cir F\xRightarrow{\;\id\circ\,\eta~} H_1\ra\cir H_2 \,.
  \ee
Given $\gamma \iN  \text{Nat}_{\rm mod}(F,H_1\ra \cir H_2)\eq\HomDual(F,\mixtd{H_2}{H_1^\vee})$ the assignment
  \be
  \gamma ~\xmapsto{\quad}~ H_1\cir F \xRightarrow{\;\id\circ\,\gamma~}
  H_1\cir H_1\ra\cir H_2 \xRightarrow{\quad} H_2
  \ee
serves as inverse of \eqref{ad_v}. The isomorphism (vi) is defined analogously. 
To obtain (vii), consider the function
  \be
  \begin{aligned}
  \text{Nat}_{\rm mod}\left((
  {-}\ot a)\cir H_1,H_2\right) & \longrightarrow \Hom_\cala\big(a,\,H_2^{}(H_1\la(\un))\big)
  \\
  \eta\, & \longmapsto a \rarr{} ({-}\ot a)\cir H_1^{}\cir H_1\la(\un)
  \xrightarrow{\,\eta_{H_1\la(\un)}\;} H_2^{}(H_1\la(\un)) \,,
  \end{aligned}
  \ee
	which has as inverse the arrow that assigns to $f\colon a\To H_2^{}(H_1\la(\un))$ the composition
  \be
  ({-}\ot a) \cir H_1\xRightarrow{\;(-\otimes f)\,\circ\,\id~}
  \big({-}\ot H_2^{}(H_1\la(\un))\big) \cir H_1\cong H_2^{}\cir H_1\la\cir H_1
  \xRightarrow~ H_2^{} \,,
  \ee
where the isomorphism is the module structure of $H_2\cir H_1\la$ and
the last arrow is the counit of the adjunction $H_1\la \,{\dashv}\, H_1$.
\end{proof}

\begin{rem}\label{iHoms_icoHoms_Morita}
Notice that the adjunctions in Proposition \ref{Morita_adjunctions} describe 
the internal Homs and coHoms of the bimodule categories in the Morita context 
in terms of the products and dualities:
  \be
  \begin{array}{rl}
  {\rm (i)}   &  \iHom_\calm^\cala(m,n) = \mixt{n}{m^\vee} ,
  \Nxl6
  {\rm (ii)}  &  \icoHom_\calm^\cala(m,n) = \mixt{n}{{}^{\vee\!} m} \,,
  \Nxl6
  {\rm (iii)} &  \iHom_\calm^\dualcat(m,n)
  = \mixtd{{}^{\vee\!} m}{n}=\icoHom_\calm^\cala(m,-)\act n \,, 
  \Nxl6
  {\rm (iv)}  &  \icoHom_\calm^\dualcat(m,n)
  = \mixtd{m^\vee}{n}=\iHom_\calm^\cala(m,-)\act n \,, 
  \Nxl6
  {\rm (v)}   &  \iHom_\FunM^{\overline{\dualcat}}(H_1,H_2)
  = \mixtd{H_2}{H_1^\vee}=H_1\ra\cir H_2 \,,
  \Nxl6
  {\rm (vi)}  &  \icoHom_\FunM^{\overline{\dualcat}}(H_1,H_2)
  = \mixtd{H_2}{{}^\vee H_1}=H_1\la\cir H_2 \,,
  \Nxl6
  {\rm (vii)} &  \iHom_\FunM^{\cala}(H_1,H_2)=\mixt{{}^\vee H_1}{H_2} = H_2\cir H_1\la(\un) \,,
  \Nxl6
  {\rm (viii)}&  \icoHom_\FunM^{\cala}(H_1,H_2)=\mixt{H_1^\vee}{H_2} = H_2\cir H_1\ra(\un) \,.
  \eear
  \label{BasId}
  \ee
In particular the formulas (iii) and (iv) relate 
the internal Homs and coHoms of the module categories ${}_\cala\calm$ and
${}_\dualcat\calm$.
\end{rem}

\begin{lem}\label{Eq_functor_cats}
Let $\calm$ be an exact module category over a finite tensor category $\cala$. Then 
the assignment
  \be
  \label{Functor1}
  \begin{aligned}
  \FunM & \longrightarrow \Fun_{\,\overline{\dualcat}\,}
  (\calm,\,\overline{\dualcat} ) \,,
  \\
  H & \longmapsto \mixtd{H}{-}
  \end{aligned}
  \ee
is an equivalence of $(\overline{\dualcat},\cala)$-bimodule categories.
\end{lem}

\begin{proof}
According to Lemma \ref{mixed_bal_bimod} the functor $\mixtd{}{}$ is a 
balanced bimodule functor. It follows that $\mixtdo{H}{-}$ is an
$\overline{\dualcat}$-module functor and that the functor \eqref{Functor1} has
an $(\overline{\dualcat},\cala)$-bimodule structure. The following functor 
is a quasi-inverse to \eqref{Functor1}:
  \be
  \begin{aligned}
  \Fun_{\,\overline{\dualcat}\,}(\calm,\,\overline{\dualcat})
  & \longrightarrow \FunM \,,
  \\ 
  K & \longmapsto \Vee[K\la(\id_\calm)]=\icoHom_\calm^\cala(K\la(\id_\calm),-)
  \end{aligned}
  \ee
Indeed, given $H\iN\FunM$ we have
  \be
  \Vee[(\mixtdo{H}{-})\la(\id_\calm)]
  \cong \Vee(H^\vee\!\actr \id_\calm) \cong H \,,
  \ee
where the first isomorphism comes from Proposition \ref{Morita_adjunctions}(iii)
and the second isomorphism from Proposition \ref{dual_properties}(ii). Conversely,
for $K\iN\Fun_{\,\overline{\dualcat}\,}(\calm,\,\overline{\dualcat})$ and $n\iN\calm$
we have the chain
  \be
  \begin{aligned}
  \mixtd{\Vee[K\la(\id_\calm)]}{n}=\icoHom_\calm^\cala(K\la(\id_\calm),-)\act n&\cong\iHom_\calm^{\dualcat}(K\la(\id_\calm),n)\\
  &\cong\iHom_\dualcat^{\dualcat}(\id_\calm,K(n))\cong K(n)
  \end{aligned}
  \ee
of natural isomorphisms, where the first isomorphism is from Remark
\ref{iHoms_icoHoms_Morita}(iii), while the second is Lemma \ref{adjoints_ihom}.
\end{proof}

\begin{prop}\label{Strong_Morita_context_module}
Let $\calm$ be an exact module category over a finite tensor category $\cala$.
Then the Morita context 
$(\cala,\,\overline{\dualcat},\,\calm,\,\FunM,\,\mixto{}{},\,\mixtdo{}{})$
from Theorem $\ref{Morita_context_exact_module}$ is strong.
\end{prop}

\begin{proof}
By invoking Proposition \ref{Deligne_eq} it follows immediately that the mixed 
product $\mixtdo{}{}$ descends to an equivalence
$\umixtdo{}{}:\FunM\boti_\cala \calm \,{\rarr{}}\, \overline{\dualcat}$.
Thus it remains to verify that $\mixto{}{}$ descends to an equivalence 
$\umixt{}{}\colon \calm\,{\boxtimes_{\calb}}\, \FunM\,{\rarr{}}\, \cala$,
with $\calb \,{:=}\, \overline{\dualcat}$, as well. To see this, consider the 
canonical equivalence \Cite{Thm.\,7.12.11}{EGno}
  \be
  \text{can}\Colon \cala\rarr{~\simeq~} (\dualcat)_\calm^* \,,\qquad
  a\mapsto a\act-
  \ee
and the bimodule equivalence from Lemma \ref{Eq_functor_cats}. The diagram
  \be
  \begin{tikzcd}[column sep=2.9em,row sep=3.1em]
  \calm\boxtimes_\calb \FunM\ar[d,"(\ref{Functor1})",swap]\ar[r,"\umixt{}{}"]&\cala \ar[d,"\text{can}"]\\
  \calm\boxtimes_\calb \Fun_\calb(\calm,\calb)\ar[r,swap,"(\ref{Deligne_right_eq})"]&\Fun_\calb(\calm,\calm)
  \end{tikzcd}
  \ee
commutes strictly: we have $m\actr (\mixtd{H}{n}) \eq H(m)\act n \eq (\mixt{m}{H})\act n$
for $n\iN\calm$. Since all other functors in the diagram are equivalences, it 
follows that $\umixto{}{}$ is an equivalence, too.
\end{proof}

\begin{rem}
Above we have focused our attention on Morita contexts associated to 
exact module categories. The reason for this is the following: According 
to Proposition \ref{Strong_Morita_context_module} the Morita context 
derived from an exact module category, as described in Theorem 
\ref{Morita_context_exact_module}, is strong. In view of Theorem 
\ref{strong_Morita}, every strong Morita context is of this type. As a 
consequence, Theorem \ref{Morita_bicategory_dualities} and related 
statements, valid for the Morita context of an exact module category, 
also hold for any arbitrary strong Morita context.
\end{rem}


\subsection{Double duals and relative Serre functors}

Given a tensor category $\cala$, the relative Serre functor of the 
regular module ${}_\cala\cala$ corresponds to the double right dual,
$\Se_\cala^\cala(a) \,{\cong}\, a\dd$. Similarly, The double duals of 
objects in the Morita context of an exact module ${}_\cala\calm$ 
admit the following description involving the relative Serre functors:

\begin{prop}\label{double_duals}
For $\calm$ an exact module category over a finite tensor category $\cala$,
let $m\iN\calm$ and $H\iN\FunM$. We have isomorphisms
  \be
  \begin{aligned}
  {\rm (i)} \quad & m\dd\cong \Se_\calm^\cala(m) \,, \qquad\qquad 
  & {\rm (iii)} \quad & H\dd\cong H\rra \cong \Se_\FunM^{\overline{\dualcat}}(H),
  \Nxl2
  {\rm (ii)} \quad & \ldd m\cong \lSe_\calm^\cala(m) \,, \qquad
  & {\rm (iv)} \quad & \,{}\dd\! H\cong H\lla \cong \Se_\FunM^{\overline{\cala}}(H).
  \end{aligned}
  \ee
\end{prop}

\begin{proof}
Combining the realization \eqref{eq:relSerre4M} of the relative Serre functor with 
the description of right duals in Definition \ref{def:duals}(i) and (iii) we directly get
  \be
  \Se_\calm^\cala(m) \cong \iHom_\calm^\cala(m,-)\ra(\un) = m\dd . 
  \ee
Similarly for a module functor we have 
  \be
  H\dd = \iHom_\calm^\cala(H\ra(\un),-) \cong (-\act H\ra(\un))\ra \cong H\rra .
  \ee
The second isomorphism in (iii) follows as
  \be
  \Se^{\overline{\dualcat}}_\FunM(H)\cong \iHom_\FunM^{\overline{\dualcat}}(H,-)\ra(\id_\calm)
  \cong (\mixtd{-}{H^\vee})\ra(\id_\calm)\cong \id_\calm\actrF H\dd
  \ee
with the help of Remark \ref{iHoms_icoHoms_Morita}(v) and Proposition
\ref{Morita_adjunctions}(vi).
The statements for double left duals follow in a similar manner.
\end{proof}

\begin{cor}\label{Serre_functors_bimodules}
The relative Serre functors of $\calm$ are related by
  \be
  \Se_\calm^\dualcat(m)\cong \lSe_\calm^\cala(m) \qquad\text{and}\qquad
  \Se_\calm^\cala(m)\cong \lSe_\calm^\dualcat(m) \,.
   \label{eq:SMA*=SMA}
  \ee
\end{cor}

\begin{proof}
The statements follow by considering again the standard realization of the relative 
Serre functors together with the duality calculus in the Morita context. For instance,
combining Remark \ref{iHoms_icoHoms_Morita}(iii), Proposition \ref{Morita_adjunctions}(iv)
and Proposition \ref{double_duals}(ii) yields the first isomorphism:
  \be
  \Se_\calm^\dualcat(m) \cong \iHomM^\dualcat(m,-)\ra(\id_\calm)
  \cong (\mixtd{\Vee m}{-})\ra(\id_\calm)\cong\Vee\Vee m\actr \id_\calm\cong\lSe_\calm^\cala(m) \,.
  \ee
The second isomorphism is obtained in a similar manner.
\end{proof}

Double duals in a tensor category are compatible with tensor products: the
double dual of a product is isomorphic to the product of the double duals of 
the factors. This property extends to any bicategory $\mathscr{F}$ with 
dualities. Moreover, the double duals of $1$-morphisms form a pseudo-equivalence
$(-)\dd\colon \mathscr{F}\xsimeq\mathscr{F}$.

In the case of the bicategory $\Mor$ associated with the Morita context of 
a module category, Proposition \ref{double_duals} implies that the double-dual
functors are isomorphic to relative Serre functors. The compatibility between 
double duals and products ensures that there are coherent natural isomorphisms
  \be\label{rSerre_compositor}
  \begin{array}{rlrl}
  {\rm (i)} &  \Se^\cala_\calm(a\act m)\cong a\dd\act \Se^\cala_\calm(m)  \,, \qquad\quad
  & {\rm (iv)} & \Se^\cala_\calm(m\actr F)\,\cong\,\Se^\cala_\calm(m)\actr F\rra \,,
  \Nxl7
  {\rm (ii)} &(F\actF H)\rra\cong F\rra\actF H\rra \,, 
  & {\rm (v)} &(H\actrF a)\rra\cong H\rra\actrF a\dd  \,,
  \Nxl7
  {\rm (iii)} &(\mixt{m}{H})^{\vee\vee}\cong\mixt{\Se^\cala_\calm(m)}{H\rra}  \,,
  & {\rm (vi)} &\mixtd{H\rra}{\Se^\cala_\calm(m)}\cong(\mixtd{H}{m})\rra 
  \end{array}
  \ee
for $a\iN\cala$, $m\iN\calm$, $H\iN\FunM$ and $F\iN\overline{\dualcat}$. These 
isomorphisms can be obtained by iterating the isomorphisms from Proposition 
\ref{dual_properties}. In particular, (i) and (iv) recover the twisted bimodule 
functor structure of $\Se_\calm^\cala$. Put differently, the isomorphisms 
\eqref{rSerre_compositor} relate the value of the relative Serre functor of 
a product with the product of the relative Serre functors evaluated in the 
corresponding factors. For instance, (iii) exhibits the coherence data
  \be
  \Se_\cala^\cala(\mixt{m}{H}) \cong (\mixt{m}{H})^{\vee\vee}
  \cong \mixt{\Se^\cala_\calm(m)}{H\rra}
  \cong \mixt{\Se^\cala_\calm(m)}{\Se_\FunM^{\overline{\dualcat}}(H)}
  \ee
for the composition of $m\iN\calm$ and $H\iN\FunM$ in $\Mor$.
In this spirit the relative Serre functors of the categories in $\Mor$ 
assemble into a pseudo-equivalence:

\begin{defi}[Relative Serre pseudo-functor]\label{def_Serre_pseudo_functor}~\\
Let $\calm$ be an exact module category over a finite tensor category $\cala$. The
\emph{relative Serre pseudo-functor} on the bicategory $\Mor$ consists of the assignment
  \be\label{Serre_pseudo_functor}
  \begin{aligned}[c]
  \Se\Colon \Mor & \,\xsimeq\, \Mor,
  \\
  \hspace*{2.5em} x & \,\longmapsto\, x \,,
  \\
  \hspace*{2.5em} \Mor(x,y)\ni a & \,\longmapsto\, \Se_{\Mor(x,y)}(a)\in \Mor(x,y)
  \end{aligned}
  \ee
together with the natural isomorphisms \eqref{rSerre_compositor}, which witness 
the compatibility with the horizontal composition in $\Mor$. 
\end{defi}


\subsection{The Radford pseudo-equivalence}

{\Color 
For a finite tensor category $\cala$, the double left dual functor is naturally 
isomorphic to the double right dual functor up to the action of the distinguished 
invertible object $\DD_\cala$. We explore how this extends to a bimodule category 
and to the entirety of the Morita context of an exact module category. 

\begin{thm}[Radford isomorphism of a bimodule category]\label{thm:bimod_Radford}~
\\
Let ${}_\calc\call_\cald$ be an exact bimodule category. There exists a natural isomorphism
  \be\label{eq:bimod_Radford}
 \mathcal{R}_{\!\call}\Colon \DD_{\!\calc}^{-1} \act \Se^\calc_\call(-) \xnatiso \Se^{\overline{\cald}}_\call(-)\actr\DD_ {\!\cald}^{-1}
  \ee
of twisted bimodule functors.
\end{thm}
\begin{proof}
We can regard $\call$ both as a left $\calc$-module and as a left $\overline{\cald}$-module.
Therefore from \eqref{Nakayama_Serre} we obtain the desired isomorphism
    \be
    \DD_{\!\calc}^{-1} \act \Se^\calc_\call(-) \cong \mathbb{N}_\call^r(-)
  \cong \DD_ {\!\cald}^{-1}\;\overline{\act\!}\,\;\Se^{\overline{\cald}}_\call(-)\equiv\Se^{\overline{\cald}}_\call(-)\actr\DD_ {\!\cald}^{-1}
    \ee
of twisted bimodule functors.    
\end{proof}
As a first consequence of Theorem \ref{thm:bimod_Radford} together with the computation of the relative Serre functors of $\calm$ we arrive at a description of the distinguished invertible 
object of the dual tensor category.

\begin{prop}\label{distinguished_dualcat}
Let $\calm$ be an exact 
module category over a finite tensor category $\cala$.
There is an isomorphism
  \be
  \label{eq:distinguished_dualcat}
  \DD_{\!\dualcat} \cong \DD_{\!\cala}^{} \act \left(\,{\lSe_\calm^\cala}\,\right)^2
  \cong \mathbb{N}_\calm^l\circ\lSe_\calm^\cala
  \ee
of $\cala$-module endofunctors of $\calm$, where $\DD_{\!\dualcat}$ is the distinguished invertible object 
of the dual tensor category $\dualcat$.
\end{prop}

\begin{proof}
Applying Theorem \ref{thm:bimod_Radford} to the $(\cala,\overline{\dualcat})$-bimodule category $\calm$ we obtain an isomorphism 
  \be
  \DD_{\!\cala}^{-1} \act \Se^\cala_\calm 
  \cong \Se^{\dualcat}_\calm(-)\actr\DD_ {\!\Dualcat}^{-1}\cong\DD_ {\!\Dualcat}^{-1}\circ\Se^\dualcat_\calm
  \ee
of bimodule functors. The result now follows by taking into account that 
$\Se_\calm^\dualcat \,{\cong}\, \lSe_\calm^\cala$. The second isomorphism in 
\eqref{eq:distinguished_dualcat} comes from the isomorphism 
$\mathbb{N}_\calm^l \,{\cong}\, \DD_{\!\cala}^{} \act\lSe_\calm^\cala$ 
in \eqref{Nakayama_Serre}.
\end{proof}

We find that Radford's theorem can be extended to exact module categories, 
with the relative Serre functor playing the role of the double right dual functor:
}
\begin{cor}[Radford isomorphism of a module category]\label{thm:Radford}~
\\
Let $\cala$ be a finite tensor category and $\calm$ an exact $\cala$-module. There 
is a natural isomorphism
  \be
  \label{Radford_mod}
  r_\Calm^{}\Colon \DD_\Cala^{}\act
  {-}\actr\DD_{\!\dualcat}^{-1} \xRightarrow{~\cong~\,} \Se_\calm^\cala\cir\Se_\calm^\cala
  \ee
of twisted bimodule functors.
\end{cor}

\begin{proof}
The statement follows from Proposition \ref{distinguished_dualcat} by reformulating
the description of $\DD_{\!\dualcat}$. From the isomorphism 
\eqref{eq:distinguished_dualcat} we obtain
  \be
  \DD_{\!\dualcat} {\circ}\,\Se_\calm^\cala \cir \Se_\calm^\cala
  \cong \Nakl \cir \lSe_\calm^\cala \cir \Se_\calm^\cala \cir \Se_\calm^\cala
  \cong \Nakl \cir \Se_\calm^\cala \cong \DD_\Cala^{} \act- \,,
  \ee
where we use the fact that $\Se_\calm^\cala$ and $\lSe_\calm^\cala$ are quasi-inverses
and the isomorphism coming from \eqref{Nakayama_Serre}.
\end{proof}

Note that Equation \eqref{Radford_mod} takes a very symmetric form because we see 
the $\cala$-module category $\calm$ as an $(\cala,\overline{\dualcat})$-bimodule: 
the distinguished invertible objects of both $\cala$ and $\overline{\dualcat}$ enter 
in \eqref{Radford_mod} on the same footing.

There are similar Radford isomorphisms for the categories $\FunM$ and 
$\overline{\dualcat}$ in the Morita context $\Mor$, where again the corresponding 
relative Serre functors play the role of the double right dual. These Radford 
isomorphisms assemble into a trivialization of the square of the relative Serre 
pseudo-functor \eqref{Serre_pseudo_functor}, i.e.\ a trivialization of the fourth 
power of the pseudo-functor of dualities of $\Mor$.

\begin{thm}[Radford pseudo-equivalence of a Morita context]~\label{Thm_Radford_pseudo}
\\
Let $\calm$ be an exact module category over a finite tensor category $\cala$ and 
$\Mor$ the bicategory associated to its Morita context. There is a pseudo-natural 
equivalence
  \be\label{Radford_pseudo}
  \mathcal{R}\Colon \id_\Mor \xRightarrow{~\cong~\,} \Se^2 ,
  \ee
where $\Se$ is the relative Serre pseudo-functor \eqref{Serre_pseudo_functor}.
\end{thm}

\begin{proof}
To construct the pseudo-natural equivalence, consider the following data:
\begin{enumerate}[(i)]
 \item 
For the objects $0$ and $1$ of $\Mor$ the distinguished invertible $1$-morphisms
  \be
  \mathcal{R}_0 := \DD_\Cala \qquad \text{and}\qquad \mathcal{R}_1 := \DD_{\dualcat} \,.
  \ee
 \item 
For $1$-morphisms in $\Mor$, the following invertible $2$-morphisms:
\end{enumerate}
\begin{itemize}
 \item[--]
For $a\iN\cala$ and $F\iN\overline{\dualcat}$, the natural isomorphisms
  \be
  \mathcal{R}_a \Colon \DD_\Cala \,{\otimes}\, a
  \xcong a^{\vee\vee\vee\vee} {\otimes}\, \DD_\cala
  \ee
and
  \be
  \mathcal{R}_m \Colon \DD_\Cala^{}\act m
  \xcong \Se^\cala_\calm \,{\circ}\, \Se^\cala_\calm(m) \actr\DD_\dualcat
  \ee
coming from \eqref{Radford} and from \eqref{Radford_mod}, respectively.
 \item[--]
For $F\iN\overline{\dualcat}$, the natural isomorphism
  \be
  \mathcal{R}_F\Colon \DD_\dualcat {\circ\opp F}
  \xRightarrow{~\cong~} F\rrrra \,{\circ\opp } \DD_\dualcat
  \ee
given by the composite
  \be
  F \,{\circ}\, \DD_\dualcat\cong F \,{\circ}\, \DD_\Cala^{}\act(\lSe^\cala_\calm)^2
  \cong \DD_\Cala^{}\act(\lSe^\cala_\calm)^2 \,{\circ}\, F\rrrra
  \cong \DD_\dualcat {\circ}\, F\rrrra ,
  \ee
where the first and last isomorphisms come from \eqref{eq:distinguished_dualcat} 
and the middle isomorphism uses the module structure of $F$ and the twisted structure 
\eqref{lSerre_module_functor} of the relative Serre functor twice.
 \item[--]
Analogously, for any $H\iN\FunM$ a natural isomorphism
  \be
  \mathcal{R}_H\Colon \DD_\dualcat\actF H \xRightarrow{~\cong~}H\rrrra\actrF\DD_\cala
  \ee
given by the composite
  \be
  H\,{\circ}\, \DD_\dualcat \cong H \,{\circ}\, \DD_\cala\act(\lSe^\cala_\calm)^2
  \cong \DD_\Cala^{} \,{\otimes}\, {}^{\vee\vee\vee\vee\!}(-)\circ H\rrrra
  \cong ({-} \,{\otimes}\, \DD_\Cala)^{} \,{\circ}\, H\rrrra ,
  \ee
where again we first use \eqref{eq:distinguished_dualcat} and then the 
twisted structure \eqref{lSerre_module_functor} of the relative Serre functor
(also taking into account that $\lSe_\cala^\cala \,{\cong}\, \ldd(-)$), and where
the last isomorphism is the Radford isomorphism of $\cala$.
\end{itemize}

\noindent
The claim now reduces to making the routine check of the commutativity of the diagram
    \be
    \begin{tikzcd}[column sep=huge]
    \DD\circ s\circ t\ar[r,"\mathcal{R}_{s\circ t}"]\ar[d,"\mathcal{R}_s\circ\id",swap]&(s\circ t)^{\vee\vee\vee\vee}\circ\DD\ar[d,"\eqref{rSerre_compositor}"]\\
    s^{\vee\vee\vee\vee}\circ\DD\circ  t\ar[r,"\id\circ\mathcal{R}_t",swap]&s^{\vee\vee\vee\vee}\circ t^{\vee\vee\vee\vee}\circ\DD
    \end{tikzcd}
    \ee
where the symbol $\circ$ denotes the horizontal composition in the bicategory $\Mor$ 
and $s$ and $t$ are composable $1$-morphisms. 
\end{proof}

The bicategorical formulation in Theorem \ref{Thm_Radford_pseudo} unifies Radford's
theorems for tensor and module categories. It also offers a natural home to the
invertible objects in Radford-type theorems, since they become part of the
data of a pseudo-natural equivalence.

\begin{rem}
Theorem \ref{Thm_Radford_pseudo} can be extended to the bicategory
$\textbf{Mod}^{\text{ex}}(\cala)$ of exact module categories over a finite tensor 
category $\cala$. $\textbf{Mod}^{\text{ex}}(\cala)$ is a bicategory with dualities
for $1$-mor\-phisms, given by $H^* \,{:=}\, H\la$ and ${}^{*\!}H := H\ra$ for every module functor
$H\iN\Fun_\cala(\calm,\caln)$. The isomorphisms coming from 
\eqref{eq:distinguished_dualcat} and \eqref{lSerre_module_functor} induce a natural isomorphism
  \be
  \mathcal{R}_H\Colon \DD_{{\cala^*_\caln}}\actF H \xRightarrow{~\cong~}H\lllla\actrF\DD_\dualcat\,,
  \ee
for each such module functor, and these together assemble into a pseudo-natural equivalence 
$\mathcal{R} \colon \id_{\textbf{Mod}^{\text{ex}}(\cala)}
    $\linebreak[0]$ 
{\xRightarrow{~\cong~\,}}\, (-)^{****}$.
\end{rem}


\section{On pivotality and Morita theory}\label{sec:piv}

\subsection{Pivotal module categories}

An additional structure that a tensor category $\cala$ can carry is a 
\emph{pivotal structure}, that is, a monoidal natural isomorphism 
$p\colon \id_\cala\Rarr{\sim}(-)\dd$ between the identity functor and 
the double-du\-al functor. The monoidal opposite $\overline{\cala\,}$ of a pivotal 
tensor category is endowed with a canonical pivotal structure, given by 
  \be \label{opp_pivotal}
  \overline{p}_{\overline{a}} := p^{-1}_{\ldd a}
  \Colon \overline{a} \xcong \overline{{}\dd a} = \overline{a}\dd .
  \ee
A pivotal structure on a module category over a pivotal tensor category can be
defined as follows:

\begin{defi}\label{def:pivmodule}
Let $\cala$ and $\calb$ be pivotal finite tensor categories.
\Enumeratei
    \item 
(\Cite{Def.\,5.2}{schaum2} and \Cite{Def.\,3.11}{Sh})
A \emph{pivotal structure} on an exact left $\cala$-module category $\calm$ is 
a natural isomorphism $\widetilde{p} \colon \id_\calm{\xRightarrow{~\cong\,}}\Se_\calm^\cala$ 
such that the diagram
  \be\label{condition_pivotal_module}
  \begin{tikzcd}[row sep=2.3em, column sep=1.6em]
  a\act m \ar[rr,"\widetilde{p}_{a\Act m}"] \ar[dr,"p_a\act \widetilde{p}_m\!\!\!",swap]
  &~& \Se_\calm^\cala(a\:\Act m) \ar[dl,"\!\eqref{Serre_twisted}"]
  \\
  ~& a\dd\act \Se_\calm^\cala(m) &~
  \end{tikzcd}
  \ee
commutes for all $a\iN\cala$ and $m\iN\calm$. A module category together with 
a module structure is said to be a \emph{pivotal module category}. 
    \item 
An exact right $\calb$-module category $\caln$ is said to be \emph{pivotal} if 
the left module category ${}_{\overline{\calb}}\caln$ has a pivotal structure. 
    \item 
A \emph{pivotal bimodule category} is an exact bimodule ${}_\cala\calm_\calb$ 
together with the structure $\widetilde{p}\colon \id_\calm{\xRightarrow{~\cong\,}}\Se_\calm^\cala$
of a pivotal $\cala$-module and the structure 
$\widetilde{q}\colon \id_\calm{\xRightarrow{~\cong\,}}\Se_\calm^{\overline{\calb}}$
of a pivotal $\calb$-module, such that the diagrams
  \be
  \label{pivotal_bimod1}
  \begin{tikzcd}[row sep=2.3em, column sep=1.6em]
  m\actr b \ar[rr,"\widetilde{p}_{m\Actr b}"] \ar[dr,"\widetilde{p}_m\Actr q_{b}\!\!\!",swap]
  &~& \Se_\calm^\cala(m\Actr\: b) \ar[dl,"\!\eqref{Serre_twisted}"]
  \\
  ~& \Se_\calm^\cala(m)\actr b\dd &~
  \end{tikzcd}
  \ee
and 
  \be
  \label{pivotal_bimod2}
  \begin{tikzcd}[row sep=2.3em, column sep=1.6em]
  a\act m \ar[rr,"\widetilde{q}_{a\Act m}"]
  \ar[dr,"p^{-1}_{\ldd a}\Act \widetilde{q}_m\!\!\!",swap]
  &~& \Se_\calm^{\overline{\calb}}(a\:\Act m) \ar[dl,"\!\eqref{Serre_twisted_op2}"]
  \\
  ~& \ldd a\act \Se_\calm^{\overline{\calb}}(m) &~
  \end{tikzcd}
  \ee
commute for all $a\iN\cala$, $b\iN\calb$ and $m\iN\calm$.
\end{enumerate}
\end{defi}

The dual tensor category of a pivotal module inherits 
the structure of a pivotal finite multi-ten\-sor category:

\begin{prop} \label{pivotal_dual_category}
Let $\cala$ be a pivotal tensor category and $\calm$ a pivotal $\cala$-module category. Then
\Enumeratei
    \item
{\rm\Cite{Thm.\,3.13}{Sh}}
The dual tensor category $\dualcat$ has a pivotal structure given by the composite
  \be
  \label{pivotal_dualcat}
  q_F\Colon F \xRightarrow{~\id\,\circ\,\widetilde{p}~} F \cir \Se_\calm^\cala
  \xRightarrow{\;\eqref{Serre_module_functor}~\,} \Se_\calm^\cala\cir F\lla
  \xRightarrow{~\widetilde{p}^{-1}\circ\,\id~\,} F\lla
  \ee
for a module endofunctor $F\iN\dualcat$, with $\widetilde{p}$ the pivotal structure
of $\calm$.
{\Color
\item Given a pivotal bimodule category ${}_\cala\calm_\calb$, the assignment
\be
\calb\longrightarrow \overline{\dualcat},\qquad b\longmapsto -\actr b
\ee
is a pivotal tensor functor.
}
\end{enumerate}
\end{prop}

\begin{proof}
The statement in \Cite{Thm.\,3.13}{Sh} concerns $\overline{\dualcat}$, the monoidal 
opposite of the dual tensor category, with its pivotal structure described by
  \be
  \label{pivotal_dualcat1}
  \overline{q}_F\Colon F \xRightarrow{\;\widetilde{p}\,\circ\,\id~\,}\Se_\calm^\cala\cir F
  \xRightarrow{\;\eqref{Serre_module_functor}~\,} F\rra\cir \Se_\calm^\cala
  \xRightarrow{\;\id\,\circ\,\widetilde{p}^{-1}~} F\rra .
  \ee
Considering the opposite pivotal structure \eqref{opp_pivotal} on $\dualcat$ 
we obtain \eqref{pivotal_dualcat}. {\Color
Assertion (ii) states that the diagram
\be
\begin{tikzcd}[column sep=3.1em]
    m\actr b \ar[r,"\widetilde{p}_{m\Actr b}"]\ar[d,swap,"\id\actr q_b"]& \Se_\calm^\cala(m\actr b)\ar[d,"\eqref{Serre_module_functor}"]\\
    m\actr b\dd \ar[r,swap,"\widetilde{p}_m\actr \id"]& \Se_\calm^\cala(m)\actr b\dd
\end{tikzcd}
\ee
commutes for every $m\iN\calm$ and $b\iN\calb$. This diagram is nothing but the condition \eqref{pivotal_bimod1}.
}
\end{proof}


\subsection{Pivotality of the category of module functors}
{\Color
Let $\cala$ be a finite tensor category and $\calm$ and $\caln$ exact $\cala$-module 
categories. Composition of module functors turns the category $\Fun_\cala(\calm,\caln)$ of module functors into an $({\cala^*_\caln},\dualcat)$-bimodule category or, equivalently, an $(\,\overline{\dualcat}, \overline{\cala^*_\caln}\,)$-bimodule category
with action given by 
  \be\label{Fun_actions}
  F_1\actF H \actrF F_2:= F_2\circ H\cir F_1 \qquad\text{for}\quad F_1\iN~\overline{\dualcat}\,,\quad F_2\iN\overline{\cala^*_\caln}
  \quad\text{and}\quad H\iN \Fun_\cala(\calm,\caln) \,.
  \ee
More generally, if $\calm$ and $\caln$ are bimodules, then $\Fun_\Cala(\calm,\caln)$
becomes a bimodule category with actions given by \eqref{left_module_functors}.

\begin{lem}\label{Serre_category_module_functors}
Let $\calm$ and $\caln$ be exact module categories over a finite tensor category $\cala$.
\Enumeratei
    \item
The relative Serre functors of the module category
$\,{}_{\overline{\dualcat}\,} \Fun_\cala(\calm,\caln)\,$ are given by
  \be
  \Se^{\overline{\dualcat}}_{\Fun_\Cala(\calm,\caln)}(H) \cong H\rra
  \qquad\text{and}\qquad
  \lSe^{\overline{\dualcat}}_{\Fun_\Cala(\calm,\caln)}(H) \cong H\lla
  \label{eq:SA*Fun=Hrra/Hlla}
  \ee
for $H\iN\Fun_\cala(\calm,\caln)$.
    \item
The Nakayama functors of $\Fun_\cala(\calm,\caln)$ are given by
  \be
  \label{Nakayama_Funr}
  \mathbb{N}_{\Fun_\Cala(\calm,\caln)}^r(H) \cong \mathbb{N}_\caln^r\cir H\cir \Se_\calm^\cala
  \cong \Se_\caln^\cala \cir H \cir \mathbb{N}_\Calm^r
  \ee
and
  \be
  \label{Nakayama_Fun}
  \mathbb{N}_{\Fun_\Cala(\calm,\caln)}^l(H) \cong \mathbb{N}_\caln^l\cir H\cir \lSe_\calm^\cala
  \cong \lSe_\caln^\cala \cir H \cir \mathbb{N}_\Calm^l
  \ee
for $H\iN\Fun_\Cala(\calm,\caln)$.
    \item
For exact bimodules ${}_\cala\calm_\calb$ and ${}_\cala\caln_\calc$, the relative
Serre functors of the bimodule category $\,{}_{\calb}\Fun_\Cala(\calm,\caln)_{\,\calc}$ are given by
  \be
  \bearll
  \Se^{\overline{\calc}}_{\Fun_\Cala(\calm,\caln)}(H)
  \cong \Se_\caln^{\overline{\calc}} \cir H \cir \Se^\cala_\calm \,, \qquad &
  \lSe^{\overline{\calc}}_{\Fun_\Cala(\calm,\caln)}(H)
  \cong \lSe_\caln^{\overline{\calc}} \cir H \cir \lSe^\cala_\calm \,,
  \Nxl6
  \lSe^{\calb}_{\Fun_\Cala(\calm,\caln)}(H)
  \cong \lSe^\cala_\caln\cir H\cir \lSe^{\overline{\calb}}_\calm \,, &
  \Se^{\calb}_{\Fun_\Cala(\calm,\caln)}(H)
  \cong \Se^\cala_\caln\cir H\cir \Se^{\overline{\calb}}_\calm
  \eear
  \label{eq:Serre_bimodFun}
  \ee
for $H\iN\Fun_\Cala(\calm,\caln)$.
\end{enumerate}
\end{lem}

\begin{proof}
\Enumeratei
    \item 
There is a bijection analogous to \eqref{ad_v} for module functors in 
$\Fun_\Cala(\calm,\caln)$. The statement thus follows in complete analogy to the 
computation for $\Fun_\Cala(\calm,\cala)$ in Proposition \ref{double_duals}{\Color , once one takes into account that $\Fun_\Cala(\calm,\caln)$ is exact over $\Dualcat$ \Cite{Prop.\ 7.12.14}{EGno}.}
    \item
We show \eqref{Nakayama_Fun}; the isomorphisms \eqref{Nakayama_Funr} follow in the same manner.
According to Proposition \ref{distinguished_dualcat} the distinguished object of the 
dual tensor category is $\DD_\Dualcat \,{\cong}\, \mathbb{N}_\calm^l\cir\lSe_\calm^\cala$.
Thus
  \be
  \mathbb{N}_{\Fun_\Cala(\calm,\caln)}^l(H)
  \cong \DD_\Dualcat \!\actF\lSe^{\overline{\dualcat}}_{\Fun_\Cala(\calm,\caln)}(H)
  \cong H\lla\cir \mathbb{N}_\calm^l\cir \lSe_\calm^\cala
  \cong \mathbb{N}_\caln^l\cir H\cir \lSe_\calm^\cala \,,
  \ee
where the second isomorphism uses the second isomorphism in \eqref{eq:SA*Fun=Hrra/Hlla}
and the last isomorphism is the twisted module structure of the Nakayama functor. 
Moreover, a factor of $\DD_\Cala$ can be juggled between the Nakayama functor of 
$\caln$ and the relative Serre functor of $\calm$ using the module structure of $H$,
leading to the second isomorphism in \eqref{Nakayama_Fun}.
    \item
{\Color According to \Cite{Prop.\ 4.15}{schaum2}, the inner-homs of $\Fun_\Cala(\calm,\caln)$ with respect to $\calb$ and $\calc$ are described in terms of those of ${}_\cala\calm_\calb$ and ${}_\cala\caln_\calc$ and their respective module actions; this ensures exactness of $\,{}_{\calb}\Fun_\Cala(\calm,\caln)_{\calc}$.}
All four isomorphisms {\Color in \eqref{eq:Serre_bimodFun} follow then} from 
the relation \eqref{Nakayama_Serre} between relative Serre and Naka\-ya\-ma functors. 
For instance, the last isomorphism in \eqref{eq:Serre_bimodFun} is given by the composite
  \be
  \Se^{\calb}_{\Fun_\Cala(\calm,\caln)}(H)
  \cong \DD_\calb^{} \actF\mathbb{N}_{\Fun_\Cala(\calm,\caln)}^r(H)
  \cong \Se^\cala_\calm \cir H \cir \mathbb{N}^r_\calm(-)\actr\DD_\calb
  \cong \Se^\cala_\calm \cir H \cir \Se^{\overline{\calb}}_\calm \,,
  \ee
where the second isomorphism comes from \eqref{Nakayama_Funr} and the last one is 
an instance of \eqref{Nakayama_Serre}.
\end{enumerate}    
\end{proof}

\begin{prop}\label{pivotality_category_functors}
Let $\cala$, $\calb$ and $\calc$ be pivotal finite tensor categories, and consider
exact bimodules ${}_\cala\calm_\calb$ and ${}_\cala\caln_\calc$.
\Enumeratei
    \item 
Suppose that ${}_\cala\calm$ and $\caln_\calc$ are pivotal modules. 
Then $\Fun_\Cala(\calm,\caln)$ has the structure of a pivotal $\calc$-module category.    
    \item 
Suppose that ${}_\cala\caln$ and $\calm_\calb$ are pivotal modules. 
Then $\Fun_\Cala(\calm,\caln)$ has the structure of a pivotal $\calb$-module category.
    \item 
If ${}_\cala\calm_\calb$ and ${}_\cala\caln_\calc$ are pivotal bimodules, then 
$\Fun_\cala(\calm,\caln)$ has the structure of a pivotal $(\calb,\calc)$-bimodule category.
\end{enumerate}
\end{prop}

\begin{proof}
To show (i) denote by
$\widehat{q}\colon \id_\caln \,{\xRightarrow{\;\simeq\;}}\, \Se_\caln^{\overline{\calc}}$
and by $\widetilde{p}\colon \id_\calm \,{\xRightarrow{\;\sim\;}}\,\Se_\calm^{\cala}$ 
the pivotal structures of the pivotal modules ${}_\cala\calm$ and $\caln_\calc$. Define
for $H\iN\Fun_\cala(\calm,\caln)$ a natural isomorphism
  \be
  \widetilde{Q}_H\Colon H \xRightarrow{\,\widehat{q}\,\circ\,\id_H\,\circ\,\widetilde{p}~\,}
  \Se^{\overline{\calc}}_\caln \cir H \cir \Se^{\cala}_\calm
  \cong \Se_{\Fun_\cala(\calm,\caln)}^{\overline{\calc}}(H) \,;
  \ee
this serves as a $\calc$-pivotal structure for $\Fun_\cala(\calm,\caln)$. It remains 
to check that the diagram
  \be
  \begin{tikzcd}[row sep=2.3em, column sep=-0.6em]
  H\actrF c \ar[rr,"\widetilde{Q}_{H\ActrF c}",Rightarrow]
  \ar[dr,"\widetilde{Q}_H \ActrF \;\overline{q}_{\overline{c}}\!\!\!\;",swap,Rightarrow]
  &~& \Se^{\overline{\calc}}_{\Fun_\cala(\calm,\caln)}(H\ActrF c)
  \ar[dl,"\eqref{Serre_twisted_op}",Rightarrow]
  \\
  ~& ~\Se^{\overline{\calc}}_{\Fun_\Cala(\calm,\caln)}(H)\,\ActrF \ldd c &~
  \end{tikzcd}
  \ee
commutes for every $c\iN\calc$ and $H\iN \Fun_\cala(\calm,\caln)$, where $\overline{q}$
is the pivotal structure of $\overline{\calc}$. Now indeed, for every $m\iN\calm$ the 
diagram
  \be
  \begin{tikzcd}[row sep=3.1em, column sep=6.2em]
  H(m)\actr c
  \ar[r,"\widehat{q}_{H(m)\actr c}"] \ar[d,"\id\actr\, \overline{q}_{\overline{c}}",swap]
  & \Se_\caln^{\overline{\calc}}\left(H(m)\actr c\right)
  \ar[r,"\,\Se_\caln^{\overline{\calc}}\circ (-\actr c)\circ\,H(\widetilde{p}_{m})~"]
  \ar[d,"\eqref{Serre_twisted_op}"]
  & \Se_\caln^{\overline{\calc}} (H(\Se_\calm^{\cala}(m))\actr c)
  \ar[d,"(\ref{Serre_twisted_op})"]
  \\
  H(m)\actr \ldd c \ar[r,swap,"\widehat{q}_{\,H(m)}\actr \id"]
  & \Se_\caln^{\overline{\calc}} (H(m)) \actr \ldd c
  \ar[r,swap,"\Se_\caln^{\overline{\calc}}\circ\,H(\widetilde{p}_{m})\actr \id"]
  & \Se_\caln^{\overline{\calc}} (H(\Se_\calm^{\cala}(m))) \actr \ldd c
  \end{tikzcd}
  \ee
commutes: the square on the left corresponds to the condition fulfilled by 
$\widehat{q}$ of being a $\overline{\calc}$-pivotal structure for $\caln$, while 
the square on the right commutes owing to naturality of $\widetilde{p}$.
 \\[2pt]
The claim (ii) follows analogously by considering as $\calb$-pivotal structure 
for $\Fun_\cala(\Calm,\caln)$ the natural isomorphism
  \be\label{piv_fun}
  \widetilde{P}_H\Colon H\xRightarrow{~\widehat{p}\,\circ\,\id_H\,\circ\,\widetilde{q}~\,}
  \Se^\cala_\caln\cir H\cir \Se^{\overline{\calb}}_\calm
  \cong \Se_{\Fun_\Cala(\calm,\caln)}^\calb(H) \,,
  \ee
where $\widehat{p}\colon \id_\caln \,{\xRightarrow{\,\simeq~\;}}\,\Se_\caln^{\cala}$ and 
$\widetilde{q}\colon \id_\calm \,{\xRightarrow{\,\simeq~\;}}\, \Se_\calm^{\overline{\calb}}$
are the corresponding pivotal structures of ${}_\cala\caln$ and $\calm_\calb$.
 \\[2pt]
Assertion (iii) is verified by making the routine check of the diagrams 
\eqref{pivotal_bimod1} and \eqref{pivotal_bimod2}.
\end{proof}

\begin{rem}
According to Corollary \ref{relative_Deligne_module_functors} the category of 
module functors is a model for the relative Deligne product. Therefore Proposition
\ref{pivotality_category_functors} implies that the product of two pivotal bimodules 
inherits a pivotal structure.
\end{rem}


\subsection{Pivotal Morita theory}

Given a bicategory $\mathscr{F}$, the existence of dualities for $1$-morphisms extends to a pseudo-func\-tor
  \be
  \begin{aligned}[c]
  (-)^\vee\Colon \mathscr{F} & \,\longrightarrow\, \mathscr{F}^{\,\text{op},\text{op}},
  \\
  \hspace*{2.5em} x & \,\longmapsto\, x \,,
  \\
  \hspace*{2.5em} (a\colon x\,{\to}\, y) & \,\longmapsto\, (a^\vee\colon y\,{\to}\, x) \,.
  \end{aligned}
  \ee
A \emph{pivotal structure} on a bicategory $\mathscr{F}$ with dualities is a 
pseudo-natural equivalence 
  \be\label{pivotal_st_bicategory}
  \textbf{P} \Colon \id_\mathscr{F} \xRightarrow{\;\simeq~\,} (-)^{\!\vee\vee}
  \ee
obeying $\textbf{P}_{\!x} \eq \id_x$ for every object $x\iN\mathscr{F}$.
A bicategory together with a pivotal structure is called a \emph{pivotal bicategory}.

\begin{rem}\label{quasi_pivotal}
A tensor category $\cala$ can be seen as a bicategory with a single object 
$\mathbf{\mathbbm{A}}$. The requirement that $\textbf{P}_{\!x} \eq \id_x$ imposed
on the pseudo-natural equivalence \eqref{pivotal_st_bicategory} ensures that a 
pivotal structure on the bicategory $\mathbf{\mathbbm{A}}$ recovers a pivotal structure 
on the tensor category $\cala$. A pseudo-natural equivalence \eqref{pivotal_st_bicategory} without this requirement corresponds to the notion of 
a \emph{quasi-pivotal} structure on $\cala$ \Cite{Sec.\ 4}{Sh3}, that is, a pair 
$(d,\gamma)$ where $d\iN\cala$ is an invertible object and 
$\gamma \eq \{\gamma_a \colon d\,{\otimes}\, a \,{\xrightarrow{\,\simeq\,}}\,a\dd \,{\otimes}\, d \}$ 
is a twisted half-braiding.
\end{rem}

\begin{defi}\label{def:pivMorita}
A Morita context $(\cala,\calb,\calm,\caln,\mixto{}{},\mixtdo{}{})$ is said to be 
\emph{pivotal} iff its associated bicategory $\Mor$ is pivotal.
\end{defi}

Next we establish that the Morita context of a pivotal module is indeed pivotal, thus
justifying the terminology.

\begin{lem}\label{pivotal_FunM}
Let $\calm$ be a pivotal module category over a pivotal tensor category $\cala$. Then
we have:
\Enumeratei
    \item 
$\calm$ has the structure of a pivotal $(\cala,\overline{\dualcat})$-bimodule category.
    \item 
For every $\cala$-module category $\caln$, the $\overline{\dualcat}$-module category 
$\Fun_\Cala(\calm,\caln)$ inherits a pivotal structure.
    \item 
The functor category $\FunM$ has the structure of a pivotal 
$(\overline{\dualcat},\cala)$-bimodule category.
\end{enumerate}
\end{lem}

\begin{proof}
Denote by $\widetilde{p}\colon \id_\calm\,{\xRightarrow{\;\sim\;}}\,\Se_\calm^{\cala}$ 
the pivotal structure of ${}_\cala\calm$. According to Corollary 
\ref{Serre_functors_bimodules} we have $\Se_\calm^\dualcat \,{\cong}\, \lSe_\calm^\cala$,
so that for any $m\iN\calm$ we can define a natural isomorphism
  \be\label{Mpivotal_dual}
  \widetilde{q}_m := \widetilde{p}^{\,-1}_{\,\lSe_\calm(m)} \Colon
  m\xrightarrow{\;\cong\;} \lSe_\calm^\cala(m) \cong \Se_\calm^\dualcat(m) \,.
  \ee
Recall from Proposition \ref{pivotal_dual_category} that $\dualcat$ is endowed with the 
pivotal structure \eqref{pivotal_dualcat}. In view of Remark \ref{Serre_quasiinverse} this pivotal structure coincides with the composition
  \be
  q_F\Colon F\xRightarrow{\,\widetilde{q}\,\circ\,\id\,}
  \lSe_\calm^\cala\circ F\xRightarrow{\;\eqref{lSerre_module_functor}\;}
  F\lla\circ\lSe_\calm^\cala\xRightarrow{\,\id\,\circ\,\widetilde{q}^{-1}}F\lla.
  \ee
Now we verify that $\widetilde{q}$ is compatible with this pivotal 
structure, i.e.\ that the diagram
  \be
  \begin{tikzcd}[row sep=2.3em, column sep=1.6em]
  F(m)  \ar[rr,"\widetilde{q}_{F(m)}"]\ar[dr,"q_F\act \widetilde{q}_m\!\!\!",swap]
  &~& \lSe_\calm^\cala\cir F(m) \ar[dl,"\!\eqref{lSerre_module_functor}"]\\
  ~& F\lla\cir \lSe_\calm^\cala(m) &~
  \end{tikzcd}
  \ee
commutes for every $F\iN\dualcat$ and $m\iN\calm$. In fact, by invoking the relevant definitions, the diagram translates to
  \be
  \begin{tikzcd}[row sep=3.1em, column sep=3.9em]
  F(m) \ar[drr,"q_F",dotted]\ar[rrr,"\widetilde{q}_{\,F(m)}"]\ar[d,swap,"\widetilde{q}_{\,F(m)}"]
  &~&~& \lSe_\calm^\cala\cir F(m) 
  \ar[d,"\eqref{lSerre_module_functor}\,"]
  \\
  \lSe_\calm^\cala\cir F(m) 
  \ar[r,"\eqref{lSerre_module_functor}\,",swap]
  & F\lla \cir \lSe_\calm^\cala(m)\ar[r,"~\id\circ\,\widetilde{q}_{\,m}^{-1}",swap] 
  & F\lla(m)\ar[r,swap,"~\id\circ\,\widetilde{q}_{\,m}"]
  & F\lla \cir \lSe_\calm^\cala(m) 
    \end{tikzcd}
  \ee
which commutes trivially. Thus it is established that ${}_\cala\calm_{\,\overline{\dualcat}}$
is a pivotal bimodule category. 
  \\[2pt]
Statement (ii) follows from (i) and Proposition \ref{pivotality_category_functors}(ii).
Explicitly, the pivotal structure is the composite
  \be\label{FunMN_piv}
  \widehat{p}_H\Colon H \xRightarrow{~\widetilde{p}^{\,\caln}\circ\,\id~\,}
  \Se_\caln^\cala\cir H \xRightarrow{\,\eqref{Serre_module_functor}~\,}
  H\rra\cir \Se_\calm^\cala \xRightarrow{\,\id\,\circ\,\left(\widetilde{p}^{\,\calm}\right)^{-1}\,}
  H\rra \cong \Se^{\overline{\dualcat}}_{\Fun_\Cala(\calm,\caln)}(H)
  \ee
for a module functor $H\iN\Fun_\Cala(\calm,\caln)$.
  \\[2pt]
Similarly, claim (iii) follows from (i) and Proposition \ref{pivotality_category_functors}(iii) 
by considering ${}_\cala\caln_\cala \eq \cala$ as the regular bimodule category. In 
this situation the pivotal structures are explicitly given by
  \be
  \label{FunM_piv}
  \widehat{p}_H\Colon H \xRightarrow{~p\,\circ\,\id~\,}
  (-)\dd \cir H \xRightarrow{\,\eqref{Serre_module_functor}~\,}
  H\rra \cir \Se_\calm^\cala \xRightarrow{\,\id\,\circ\,\widetilde{p}^{-1}~\,}
  H\rra \cong \Se^{\overline{\dualcat}}_{\Fun_\cala(\calm,\cala)}(H)
  \ee
and
  \be
  \label{FunM_piv1}
  \widehat{q}_H\Colon H \xRightarrow{~\id\,\circ\,\widetilde{p}~\,}
  H \circ \Se_\calm^\cala  \xRightarrow{\,\eqref{Serre_module_functor}~\,}
   (-)\dd \cir H\lla  \xRightarrow{\,p^{-1}\,\circ\,\id~\,}
  H\lla \cong \Se_\FunM^{\overline{\cala}}(H)
  \ee
for a module functor $H\iN\FunM$.
\end{proof}

\begin{thm}\label{thm:pivmodule-pivMorita}
For $\calm$ a pivotal module category over a pivotal tensor category $\cala$, 
its Morita context
$(\cala,\,\overline{\dualcat},\,\calm,\,\FunM,\,\mixto{}{},\,\mixtdo{}{})$
is a pivotal Morita context.
\end{thm}

\begin{proof}
Let $\Mor$ be the bicategory associated to the Morita context of $\calm$. We need 
to construct a pseudo-natural equivalence 
$\textbf{P}\colon \id_\Mor \,{\xRightarrow{\;\simeq\;}}\, (-)\dd$, subject to the
condition that the components of $\textbf{P}$ on any object is the identity. For 
any $1$-morphism $a\iN\cala$ and any $F\iN\overline{\dualcat}$, $m\iN\calm$ and 
$H\iN\FunM$ we define
 \begin{enumerate}[(i)]
    \item 
$\textbf{P}_{\!a} \colon a \,{\xcong}\, a\dd$ as the pivotal structure $p_a$ of $\cala$;
    \item 
$\textbf{P}_{\!F}^{} \colon F \,{\xRightarrow{~\cong~\,}}\, F\rra$ as the isomorphism 
$\overline{q}_F$ in \eqref{pivotal_dualcat1} which serves as a pivotal structure 
of $\dualcat$;
    \item 
$\textbf{P}_{\!m} \colon m \,{\xcong}\, \Se_\calm^\cala(m)$ as the pivotal structure 
$\widetilde{p}_m$ of the module ${}_\cala\calm$;
    \item
$\textbf{P}_{\!H}^{} \colon H \,{\xRightarrow{~\cong~\,}}\, H\rra$ as the isomorphism 
$\widehat{p}_H$ in \eqref{FunM_piv} which serves as a pivotal structure of the 
module ${}_{\overline{\dualcat}\,}\FunM$.
\end{enumerate}
These assignments are natural for $2$-morphisms in $\Mor$. The compatibility 
with composition of $1$-morphisms reduces to the commutativity of the diagram
  \be
  \begin{tikzcd}[row sep=2.2em, column sep=1.8em]
  s\cir t \ar[rr,"\textbf{P}_{\!s\circ t}^{}",Rightarrow]
  \ar[dr,"\textbf{P}_{\!s}^{}\circ\,\textbf{P}_{\!t}^{}\!",Rightarrow,swap]
  &~& (s\cir t) \dd\ar[dl,Rightarrow,"\cong"]
  \\
  ~& s\dd {\circ}\, t\dd&
  \end{tikzcd}
  \ee
for $s$ and $t$ composable $1$-morphisms in $\Mor$. These translate into the 
following eight conditions,
	  which are indeed all satisfied:
 \\{}
The tensor categories $\cala$ and $\overline{\dualcat}$ are pivotal, i.e.:
\Enumeratei
    \item
the monoidality condition of $p\colon \id_\cala \,{\xRightarrow{\,\cong~\,}}\,(-)\dd$ 
in the case $s \eq a$ and $t \eq b$ in $\cala$;
    \item 
the monoidality condition of $\overline{q} \colon \id_{\overline{\dualcat}}
\,{\xRightarrow{\,\cong\,~}}\,(-)\rra$ in the case $s \eq F_1$ and $t \eq F_2$ in 
$\overline{\dualcat}$.
\end{enumerate}
Further, the bimodules $\calm$ and $\FunM$ are pivotal ($\widetilde{p}$ and
$\widehat{p}$ are bimodule natural transformations), i.e.:
\Enumeratei \setcounter{enumi}{2}
    \item
in case $s \eq a\iN\cala$ and $t \eq m\iN\calm$, the condition 
\eqref{condition_pivotal_module} fulfilled by $\widetilde{p}$;
    \item 
for $s \eq m\iN\calm$ and $t \eq F\iN\overline{\dualcat}$, the condition 
\eqref{pivotal_bimod1} fulfilled by $\widetilde{p}$;
    \item 
in case $s \eq F\iN\overline{\dualcat}$ and $t \eq H\iN\FunM$, the condition 
\eqref{condition_pivotal_module} fulfilled by $\widehat{p}$;
    \item 
for $s \eq H\iN\FunM$ and $t \eq a\iN\cala$, the condition \eqref{pivotal_bimod1}
fulfilled by $\widehat{p}$.
\end{enumerate}
Finally, two additional conditions involving the mixed products:
\Enumeratei \setcounter{enumi}{6}
    \item 
for $m\iN\calm$ and $H\iN\FunM$ the commutativity of the diagram
  \be
  \begin{tikzcd}[row sep=2.3em, column sep=1.3em]
  \mixt{m}{H} \ar[rr,"p_{\mixto{m}{H}}^{}"]
  \ar[dr,"\mixt{\widetilde{p}_m}{\widehat{p}_H\!\!\!}",swap]
  &~& (\mixto{m}{H})\dd \ar[dl,"\!\eqref{Serre_module_functor}"]
  \\
  ~& \mixt{\Se_\calm^\cala(m)}{H\rra} &~
  \end{tikzcd}
  \ee 
which can be rewritten in the more explicit form
  \be
  \begin{tikzcd}[row sep=2.2em, column sep=3.4em]
  H(m) \ar[rr,"p_{H(m)}^{}"] \ar[d,swap,"p\,\circ\,\id"]
  \ar[ddr,dotted,"\!\widehat{p}_H^{}"]
  &~& H(m)\dd \ar[dd,"\eqref{Serre_module_functor}"]
  \\
  (-)\dd \cir H(m) \ar[d,"\eqref{Serre_module_functor}",swap] &~&~ 
  \\
  H\rra \cir \Se_\calm^\cala(m) \ar[r,swap,"\id\,\circ\,\widetilde{p}_m^{-1}\,"]
  & H\rra(m) \ar[r,swap,"\id\,\circ\,\widetilde{p}_m"]
  & H\rra \cir \Se_\calm^\cala(m)
  \end{tikzcd}
  \ee
which trivially commutes;
\item similarly, for $m\iN\calm$ and $H\iN\FunM$ the commutativity of the diagram
  \be
  \begin{tikzcd}[row sep=2em, column sep=1.3em]
  \mixtd{H}{m} \ar[rr,"\overline{q}_{\mixtdo{H}{m}}",Rightarrow]
  \ar[dr,"\mixtd{\,\widehat{p}_H}{\widetilde{p}_m\,}",swap,Rightarrow]
  &~& (\mixtdo{H}{m})\rra \ar[dl,"\cong",Rightarrow]
  \\
  ~& \mixtd{H\rra}{\Se_\calm^\cala(m)}&
  \end{tikzcd}
  \ee
which is the same as the diagram
  \be\hspace*{-5.7em}
  \begin{tikzcd}[row sep=2.9em, column sep=2.7em]
  (-\Act m)\cir H\ar[r,,end anchor={[xshift=-6ex]},"\widetilde{p}\cir\id",Rightarrow]
  \ar[d,start anchor={[xshift=2.8ex]},end anchor={[xshift=2.8ex]},swap,
       "(-\act\widetilde{p}_m)\circ p\cir\id\,",Rightarrow]
  & \hspace*{-6ex} \Se_\calm^\cala\cir(-\Act m)\cir H
  \ar[r,Rightarrow,"\eqref{Serre_module_functor}~"]
  \ar[dl,"\eqref{Serre_module_functor}\,",Rightarrow]
  & \left[(-\Act m)\cir H\right]\rra {\circ}\, \Se_\calm^\cala \ar[d,"\,\cong",Rightarrow]
  \ar[r,Rightarrow,"\id\cir\widetilde{p}^{-1}\,"]
  & \left[(-\Act m)\cir H\right]\rra \ar[d,"\,\cong",Rightarrow]
  \\
  \hspace*{6ex} (-\Act \Se_\calm^\cala(m))\,{\circ}\,(-)\dd{\circ} H
  \ar[rr,end anchor={[xshift=-12ex]},"\eqref{Serre_module_functor}",swap,Rightarrow]
  &~& \hspace*{-12ex} (-\Act\Se_\calm^\cala(m))\cir H\rra {\circ}\, \Se_\calm^\cala
  \ar[r,swap,"\id\cir\widetilde{p}^{-1}\,",Rightarrow]
  & (-\Act\Se_\calm^\cala(m))\cir H\rra
  \end{tikzcd}
  \ee
which is commutative: the left triangle corresponds to the condition of $\widetilde{p}$
being a pivotal structure for $\calm$, the rightmost square commutes due to naturality,
and the square in the middle is the compatibility \eqref{Serre_compatibility_composition}.
\end{enumerate}
This shows that $\textbf{P}\colon \id_\Mor \,{\xRightarrow{\;\cong\;\,}}\,(-)\dd$ is 
a pivotal structure on the bicategory $\Mor$, and thus the claim is proven.
\end{proof}

\begin{defi}\label{PME}
Two pivotal tensor categories $\cala$ and $\calb$ are said to be \emph{pivotal Morita 
equivalent} iff there exists a pivotal $\cala$-module category $\calm$ together with 
a pivotal equivalence $\calb\,{\simeq}\,\overline{\dualcat}$.
\end{defi}

\begin{prop}\label{eq_relation}
Let $\cala$ be a pivotal category and $\calm$ a pivotal $\cala$-module.
\begin{enumerate}[\rm (i)]
    \item 
The tensor equivalence
  \be
  \label{reflexivity_PME}
  \cala\xrightarrow{\;\simeq\;}\overline{\cala_\cala^*} \,,
  \qquad a\xmapsto{\,~\,} -\ot a
  \ee
from {\rm\Cite{Ex.\,7.12.3}{EGno}} is pivotal.
    \item 
The canonical tensor equivalence
  \be
  \text{\normalfont can}\Colon \cala\xrightarrow{\;\simeq\;}(\dualcat)_\calm^* \,,
  \qquad a\xmapsto{\,~\,} a\act-
  \ee
from {\rm\Cite{Thm.\,7.12.11}{EGno}} is pivotal.
\end{enumerate}
\end{prop}

\begin{proof}~{}
\begin{enumerate}[(i)]
  \item 
The pivotal structure for the functor $-\ot a$ in $\overline{\cala_\cala^*}$ is given by
  \be
  q_{-\otimes a}\Colon (-\ot a)\xRightarrow{\;p_{-\otimes a}^{}\;\,}(-\ot a)\dd
  \xRightarrow{\;\eqref{Serre_module_functor}\;\;}(-)\dd \ot a\dd
  \xRightarrow{\;p_{-}^{-1}\otimes\id\;\;}(-\ot a\dd) \,.
  \ee    
But since $p\colon \id_\cala \,{\xRightarrow{\;\simeq\,\;}}\,(-)\dd$ is monoidal, we have
$q_{-\otimes a} \eq \id_{-}\ot p_a$.
  \item
According to Corollary \ref{Serre_functors_bimodules} there is an isomorphism 
$\lSe_\calm^{\dualcat} \,{\cong}\, \Se_\calm^\cala$. Taking this into consideration 
the pivotal structure on $(\dualcat)_\calm^*$ is given by
  \be
  \begin{aligned}
  \overline{\overline{q}}_{a\act-}\Colon (a\act-) \xRightarrow{\;\widetilde{p}\,\circ\,\id\;\,}
  \Se_\calm^\cala \left(a\act-\right) &\xRightarrow{\,\eqref{Serre_twisted}\;\,}
  a\dd\act\Se_\calm^\cala \\
  &\xRightarrow{\;\id\,\circ\, \widetilde{p}^{-1}\;} (a\dd\act-)
  = (a\act-)\lla .
  \end{aligned}
  \ee
But in a similar manner the defining condition \eqref{condition_pivotal_module} of 
$\widetilde{p}$ being a pivotal structure for $\calm$ implies that
$\text{\normalfont can}(p_a) \eq \overline{\overline{q}}_{a\act-}\,$.
\end{enumerate}
  ~\\[-2.3em]~
\end{proof}


\subsection{The bicategory of pivotal modules}

Pivotal modules over a pivotal tensor category form a pivotal bicategory, as 
presented in \Cite{Def.\,5.2}{schaum2} in terms of inner-product module 
categories. Here we express this fact in the language of relative Serre functors. Let
$\cala$ be a pivotal tensor category, and denote by $\textbf{Mod}^{\text{piv}\!}(\cala)$ 
the $2$-category that has pivotal $\cala$-module categories as objects, $\cala$-module 
functors as $1$-morphisms and module natural transformations as $2$-morphisms. Since 
pivotal modules are exact \Cite{Prop.\,4.24}{fuScSc}, every module functor 
$H\colon {}_{\cala\,}\caln_1 \To {}_{\cala\,}\caln_2$ comes with adjoints
\be
  H^* := H\la :~ {}_{\cala\,}\caln_2 \rarr{} {}_{\cala\,}\caln_1 \qquad\text{and}\qquad
  {}^{*\!}H := H\ra :~ {}_{\cala\,}\caln_2 \rarr{} {}_{\cala\,}\caln_1 \,.
  \ee
These turn $\textbf{Mod}^{\text{piv}\!}(\cala)$ into a bicategory with dualities 
for $1$-morphisms. Moreover, $\textbf{Mod}^{\text{piv}\!}(\cala)$ is endowed with 
a pivotal structure \eqref{pivotal_st_bicategory}. Indeed, given any $1$-morphism
$H\colon {}_{\cala\,}\caln_1 \Rarr{} {}_{\cala\,}\caln_2$ in 
$\textbf{Mod}^{\text{piv}\!}(\cala)$, define 
  \be
  \label{pivotal_structure_bicategory_of_pivotal_modules}
  \textbf{P}_{\!H}^{}\Colon H
  \xRightarrow{\;\id\cir\widetilde{p}_1~} H \cir \Se_{\caln_1}^\cala
  \xRightarrow{\;\eqref{Serre_module_functor}~} \Se_{\caln_2}^\cala \cir H\lla
  \xRightarrow{\;(\widetilde{p}_2)^{-1}\cir\id~\,} H\lla ,
  \ee
where $\widetilde{p}_i$ are the pivotal structures of the module categories $\caln_i$. 
The $2$-morphisms $\textbf{P}_{\!H}^{}$ are invertible and natural in $H$. Moreover,
\eqref{Serre_compatibility_composition} implies that they are compatible with the
composition of module functors. Therefore $\textbf{P}$ constitutes a pivotal structure 
on the $2$-category $\textbf{Mod}^{\text{piv}\!}(\cala)$.

\begin{thm}\label{bicategories_pivotal_modules}
Two pivotal tensor categories $\cala$ and $\calb$ are pivotal 
Morita equivalent if and only if $\mathbf{Mod}^{\rm{piv}\!}(\cala)$ and 
$\mathbf{Mod}^{\rm{piv}\!}(\calb)$ are equivalent as pivotal bicategories.
\end{thm}

\begin{proof}
Given a pivotal $\cala$-module $\calm$, according to Lemma \ref{pivotal_FunM}(ii)
for every $\caln\iN\textbf{Mod}^{\text{piv}\!}(\cala)$ the $\overline{\dualcat}$-mo\-dule 
$\Fun_\cala(\calm,\caln)$ is endowed with a pivotal structure. By Theorem
7.12.16 of \cite{EGno} this assignment extends to a $2$-equivalence
  \be
  \begin{aligned}
  \Psi\Colon \textbf{Mod}^{\text{piv}\!}(\cala) & 
  \longrightarrow\textbf{Mod}^{\text{piv}\!} (\overline{\dualcat}) \,,
  \nxl1
  \caln & \longmapsto \Fun_\cala(\calm,\caln) \,.
  \end{aligned}
  \ee
Moreover, $\Psi$ preserves the pivotal structure: To a $1$-morphism 
$H \colon {}_{\cala\,}\caln_1 \Rarr{} {}_{\cala\,}\caln_2$ in 
$\textbf{Mod}^{\text{piv}\!}(\cala)$ it assigns 
  \be
  \Psi(H)\Colon \Fun_\cala(\calm,\caln_1) \xrightarrow{\,H\cir-\;}
  \Fun_\cala(\calm,\caln_2) \,.
  \ee
The component at $\Psi(H)$ of the pivotal structure of
$\textbf{Mod}^{\text{piv}\!}(\overline{\dualcat})$ is the composite
  \be
  \textbf{P}_{\!\Psi(H)}^{}\Colon (H\cir-) \xRightarrow{\;\id\cir\widehat{p}_1~}
  H\cir (-)\rra \xRightarrow{~\cong~\,} (H\lla\cir-)\rra
  \xRightarrow{\;(\widehat{p}_2)^{-1}{\circ}\,\id~\,} (H\lla\cir-) \,,
  \ee
where, for $i \eq 1,2$, $\widehat{p}_i$ is the pivotal structure of the module 
category $\Fun_\cala(\calm,\caln_i)$ given by \eqref{FunMN_piv}.
The task at hand is to verify that the diagram
  \be
  \begin{tikzcd}
  \Psi(H)\ar[rr,"\textbf{P}_{\Psi(H)}^{}",Rightarrow]
  \ar[dr,"\Psi(\textbf{P}_{H})",Rightarrow,swap]&&\Psi(H)\lla\ar[dl,"\cong",Rightarrow]
  \\
  &\Psi(H\lla)&
  \end{tikzcd}
  \ee
commutes for every $1$-morphism $H$ in $\textbf{Mod}^{\text{piv}\!}(\cala)$. By inserting
the definitions, this diagram translates to
  \be
  \begin{tikzcd}[row sep=3.2em,column sep=3.6em]
  (H\cir-) \ar[r,"\id\cir\widetilde{p}_{\caln_1}{\circ}\,\id",Rightarrow]
  \ar[d,"\id\cir\widetilde{p}_{\caln_1}{\circ}\,\id",swap,Rightarrow]
  & (H\cir\Se^\cala_{\caln_1}\cir -) \ar[dl,"\id",Rightarrow]
  \ar[r,"\eqref{Serre_module_functor}",Rightarrow]
  & H\cir(-)\rra\cir\Se^\cala_{\calm}
  \ar[r,"\id\cir(\widetilde{p}_{\calm})^{-1}",Rightarrow]
  \ar[ddr,swap,"\,\cong",Rightarrow]
  & H\cir(-)\rra\ar[d,Rightarrow,"\cong"]
  \\
  (H\cir\Se^\cala_{\caln_1}\cir -)
  &~&~& (H\lla\circ-)\rra \ar[d,Rightarrow,"\id\cir\widetilde{p}_{\calm}"]
  \\
  (\Se^\cala_{\caln_2}\cir H\lla\cir -) \ar[u,Rightarrow,"\eqref{Serre_module_functor}"]
  & (H\lla\circ-) \ar[l,Rightarrow,"\widetilde{p}_{\caln_2}{\circ}\,\id"]
  \ar[r,Rightarrow,swap,"\widetilde{p}_{\caln_2}\cir\id"]
  & \Se^\cala_{\caln_2}\cir(H\lla\cir-) \ar[r,swap,Rightarrow,"\eqref{Serre_module_functor}"]
  \ar[ull,"\eqref{Serre_module_functor}",swap,Rightarrow]
  & (H\lla\cir-)\rra\cir\Se^\cala_\calm
  \end{tikzcd}
  \ee
This is indeed a commutative diagram: The pentagon in the middle commutes 
owing to the compatibility \eqref{Serre_compatibility_composition}, and the 
triangle and squares in the periphery commute trivially.
 \\[2pt]
To see the converse implication, consider a pivotal $2$-equivalence
  \be
  \Phi\Colon \textbf{Mod}^{\text{piv}\!}(\calb) \xrightarrow{~\simeq~}
  \textbf{Mod}^{\text{piv}}(\cala) \,.
  \ee
Define $\calm$ as the image of the regular pivotal module ${}_\calb\calb$ under 
$\Phi$. Since $\Phi$ is a $2$-equivalence, we obtain an equivalence 
  \be
  \Phi_{\calb,\calb}\Colon \calb_\calb^*=\Fun_\calb(\calb,\calb)
  \xrightarrow{\;\simeq\;} \Fun_\cala(\calm,\calm) = \dualcat
  \ee
of categories.
In addition, amongst the data of the $2$-functor $\Phi$ there is a natural isomorphism 
 \be
  \begin{tikzcd}[row sep=3.9em,column sep=3.2em]
  \Fun_\calb(\calb,\calb) \,{\times}\,\Fun_\calb(\calb,\calb) \ar[d,"\Phi\times\Phi",swap]
  \ar[r,"\circ"]
  & \Fun_\calb(\calb,\calb) \ar[dl,"\gamma",Rightarrow] \ar[d,"\Phi"]
  \\
  \Fun_\cala(\calm,\calm)\,{\times}\,\Fun_\cala(\calm,\calm) \ar[r,"\circ",swap]
  & \Fun_\cala(\calm,\calm)
  \end{tikzcd}
  \ee
whereby $\Phi_{\calb,\calb}$ is endowed with a tensor structure. Furthermore, 
$\Phi_{\calb,\calb}$ is pivotal. To see this, notice that an object $F\iN\dualcat$ 
is a $1$-morphism in $\textbf{Mod}^{\text{piv}\!}(\cala)$. Now on the one hand the 
pivotal structure \eqref{pivotal_dualcat} of the tensor category $\dualcat$ provides
an isomorphism $q_F \colon F \,{\xRightarrow{\;\cong~}}\, F\lla$, while on the other hand
the pivotal structure \eqref{pivotal_structure_bicategory_of_pivotal_modules} 
in the bicategory $\textbf{Mod}^{\text{piv}\!}(\cala)$ for $F$ is an isomorphism 
$\textbf{P}_{\!F}^{}\colon F\,{\xRightarrow{\;\cong~}}\, F\lla$, and in fact it
coincides with $q_F$. The same argument holds for $\calb_\calb^*$; since 
$\Phi$ preserves $\textbf{P}$, it follows that $\Phi_{\calb,\calb}$ is a pivotal 
equivalence. We have thus obtained an equivalence 
  \be
  \calb \xrightarrow{\;\eqref{reflexivity_PME}~}
  \overline{\calb_\calb^*} \xrightarrow{\;\Phi_{\calb,\calb}~} \overline{\dualcat}
  \ee
of pivotal tensor categories; hence $\cala$ and $\calb$ are pivotal Morita equivalent.
\end{proof}

\begin{rem}
Theorem \ref{bicategories_pivotal_modules} implies that pivotal Morita equivalence 
is indeed an equivalence relation on pivotal tensor categories. Proposition 
\ref{eq_relation} already shows reflexivity and symmetry.
\end{rem}


\subsection{Pivotality of the center and pivotal Morita equivalence}

Recall that the \emph{Drinfeld center} of a tensor category $\cala$ is  the
braided tensor category $\calz(\cala)$ whose objects are pairs $(a,\sigma)$ 
consisting of an object $a\iN\cala$ and a \emph{half-braiding} $\sigma$, i.e.\
a natural isomorphism $\sigma_{b,a} \colon b\,{\otimes}\, a\xcong a\,{\otimes}\, b$ 
for $b\iN\cala$ obeying the appropriate hexagon axiom. According to 
\Cite{Prop.\,8.10.10}{EGno} the Drinfeld center $\calz(\cala)$ is unimodular. A pivotal 
structure $p_a \colon a\xcong a\dd$ on $\cala$ induces a pivotal structure on 
$\calz(\cala)$ via $p_{(a,\sigma)} \,{:=}\, p_a$ \Cite{Ex.\,7.13.6}{EGno}. 

The following result relates the Drinfeld center and module categories: Given 
a module category $\calm$ over a tensor category $\cala$ there is a braided 
equivalence $\calz(\cala)\,{\simeq}\, \calz(\overline{\dualcat})$. Explicitly it can
be given, without making a choice of an algebra in $\cala$, as \Cite{Thm.\,3.13}{Sh2}
  \be
  \Sigma\Colon \calz(\cala)\xsimeq \calz(\,\overline{\dualcat}\,), \qquad
  (a,\sigma) \longmapsto (a\act-,\gamma) \,,
  \label{schauenburg}
  \ee
where $\sigma$ endows $a\act-$ with an $\cala$-module functor structure, and for 
$F\iN\overline{\dualcat}$ the half-braiding $\gamma_{\,F,a\Act-} \colon a\act F
\Rarr\cong F(a\act-)$ is given by the module functor structure of $F$.

\begin{lem}\label{lemma1}
Let $\cala$ be a finite tensor category and $\calm$ an $\cala$-module, and let
$(a,\sigma)\iN\calz(\cala)$.
\begin{enumerate}[\rm (i)]
    \item 
There is a natural isomorphism $\iota_a\colon a\dd\xcong\ldd a$ in $\calz(\cala)$.
    \item
The diagram 
  \be
  \begin{tikzcd}[row sep=3.3em,column sep=1.9em]
  a\dd\act\Se^\cala_\calm \ar[dr,swap,"\eqref{Serre_twisted}",Rightarrow]
  \ar[rr,"\iota_a\act\id",Rightarrow] 
  &~& \ldd a\act\Se^\cala_\calm \ar[dl,"\eqref{Serre_module_functor}",Rightarrow]
  \\
  ~& \Se^\cala_\calm(a\act-) &~
  \end{tikzcd}
  \ee
commutes, where the natural isomorphism \eqref{Serre_module_functor} is applied 
to the module functor $a\act-$.
\end{enumerate}
\end{lem}

\begin{proof}
The desired isomorphism in (i) comes from the composition
  \be
  a\dd \,{\otimes}\, \DD_\Cala\xrightarrow{~\eqref{Radford}~}
  \DD_\Cala \,{\otimes}\, \ldd a \xrightarrow{\;\ldd\sigma~}\ldd a \,{\otimes}\, \DD_\Cala\,.
  \ee
Now denote by $F_a \,{:=}\, a\act-$ the module functor induced by $(a,\sigma)$. Its 
double adjoints are $F_a\lla \eq a\dd\act-$ and $F_a\rra \eq \ldd a\act-$. Assertion 
(ii) is thus implied by the commutativity of the diagram
  \be
  \begin{tikzcd}[row sep=2.5em,column sep=2.4em]
  a\dd\Act\Se^\cala_\calm \ar[dd,swap,"\eqref{Serre_twisted}",Rightarrow]
  \ar[rrrr,Rightarrow,"\iota_a\act\id"] \ar[dr,Rightarrow,"\eqref{Nakayama_Serre}",swap]
  &~&~&~& \ldd a\Act\Se^\cala_\calm \ar[dd,Rightarrow,"\,\eqref{Serre_module_functor}"]
  \\
  ~& a\dd \,{\otimes}\, \DD_\Cala\Act\Nakr \ar[r,Rightarrow,"~\,\eqref{Radford}~~"]
  & \DD_\cala \,{\otimes}\, \ldd a\Act\Nakr \ar[r,Rightarrow,"~\ldd\sigma~~"]
  \ar[d,Rightarrow,"\,\eqref{Nakayama_twisted_functor}"]
  & \ldd a \,{\otimes}\, \DD_\cala\Act\Nakr\ar[ur,Rightarrow,swap,"\eqref{Nakayama_Serre}"] &~
  \\[.3em]
  \Se^\cala_\calm(a\Act-)\ar[rr,Rightarrow,"\eqref{Nakayama_Serre}",swap]
  &~& \DD_\cala\Act\Nakr(a\Act-) \ar[rr,Rightarrow,"\eqref{Nakayama_Serre}",swap]
  &~& \Se^\cala_\calm(a\Act-)
  \end{tikzcd}
  \ee
where \eqref{Nakayama_twisted_functor} and \eqref{Serre_module_functor} are applied
to the functor $F_a$. The pentagon at the top of this diagram commutes owing to 
naturality and the definition of $\iota_a$. Commutativity of the pentagon on the left 
is the condition that \eqref{Nakayama_Serre} is an isomorphism of twisted module 
functors. Similarly, the pentagon on the right is secretly the diagram
  \be
  \begin{tikzcd}[row sep=2.6em,column sep=3.2em]
  F_a\rra\,{\circ}\,\DD_\cala\act\Nakr\ar[r,Rightarrow,"\eqref{Nakayama_Serre}~"]
  \ar[d,Rightarrow,swap,"\eqref{Nakayama_twisted_functor}\cir \sigma\,"]
  & F_a\rra\,{\circ}\,\Se_\calm^\cala \ar[d,Rightarrow,"\,\eqref{Serre_module_functor}"]
  \\
  \DD_\cala\act\Nakr\,{\circ}\,F_a \ar[r,swap,Rightarrow,"\eqref{Nakayama_Serre}"]
  & \Se_\calm^\cala\cir F_a
  \end{tikzcd}
  \ee
which again commutes because \eqref{Nakayama_Serre} is an isomorphism of twisted 
bimodule functors.
\end{proof}

\begin{prop}\label{schauenburg_pivotal}
For any pivotal $\cala$-module category $\calm$ the braided equivalence $\Sigma$
between $\calz(\cala)$ and $\calz(\,\overline{\dualcat}\,)$ defined in
\eqref{schauenburg} is pivotal.
\end{prop}

\begin{proof}
{}From Lemma \ref{lemma1} it follows that the pivotal structure in 
$\calz(\overline{\dualcat})$ for the object $\Sigma(a,\sigma)$ is given by the composite
  \be
  \label{pivotal_Sigma}
  a\act- \xRightarrow{~\widetilde{p}\cir\id~\,} \Se_\calm^\cala(a\act-)
  \xRightarrow{~\eqref{Serre_twisted}~\,} a\dd{\Act}\,\Se_\calm^\cala
  \xRightarrow{~\id\cir(\,\widetilde{p}\;)^{-1}\;}(a\dd\act-) \,.
  \ee
The condition \eqref{condition_pivotal_module} on $\widetilde{p}$ means that
$\Sigma(p_a)$ coincides with the morphism \eqref{pivotal_Sigma}, and thus the 
assertion holds.
\end{proof}

Proposition \ref{schauenburg_pivotal} immediately implies

\begin{thm}\label{thm:A=B-ZA=ZB}
If two pivotal categories $\cala$ and $\calb$ are pivotal Morita equivalent, 
then their Drinfeld centers $\calz(\cala)$ and $\calz(\calb)$ are equivalent 
as pivotal braided tensor categories.
\end{thm}


\subsection{Sphericality of module categories}\label{sec:spherical}

{\Color
A notion of sphericality for pivotal tensor categories that is defined through 
the Radford isomorphism \eqref{Radford} has been studied, under the assumption of 
unimodularity, in \cite{DSPS}. In the semisimple case this notion is equivalent to 
trace-sphericality \Cite{Prop.\,3.5.4}{DSPS}, i.e.\ to the property that right 
and left traces coincide. 
Let $\calc$ be a unimodular finite tensor category. The monoidal functor given 
by conjugation with $\DD_\Calc$ can be canonically identified with the identity 
functor of $\calc$: Consider any trivialization 
$\textbf{u}_\calc\colon\textbf{1}\xcong\DD_\Calc^{-1}$
of the distinguished invertible object. Then
  \be
  \id_\calc\xRightarrow{\textbf{u}_\calc^*\otimes\id\otimes\textbf{u}_\calc^{}}\DD_\Calc^{}\otimes-\otimes\DD_\Calc^{-1}
  \label{conj_trivial}
  \ee
does not depend of the choice of $\textbf{u}_\calc$ since $\textbf{1}$ is simple 
and thus $\Hom_\calc(\textbf{1},\DD_\Calc)$ is one-dimensional. By precomposing 
the Radford isomorphism \eqref{Radford} with \eqref{conj_trivial}, we obtain a 
canonical monoidal isomorphism
  \be
  \overline{r}_\Calc^{}\Colon \id_\calc\xNatiso (-)^{\vee\vee\vee\vee} 
  \label{Radford_can}
  \ee
that trivializes the fourth power of the right dual functor.
\begin{defi}\label{spherical_tc}
\Cite{Def.\,3.5.2}{DSPS} A unimodular pivotal tensor category $\calc$ is called 
\emph{spherical} iff the diagram 
  \be
  \begin{tikzcd}[row sep=2.5em,column sep=1.8em]
  \id_\calc \ar[rr,"\overline{r}_\calc^{}",Rightarrow] \ar[rd,"p",swap,Rightarrow] &~
  & (-)^{\vee\vee\vee\vee}
  \\
  ~& (-)\dd \ar[ru,"\!p\dd",swap,Rightarrow] &~
  \end{tikzcd}
  \ee
commutes, where $p$ is the pivotal structure of $\calc$.
\end{defi}    

It turns out that sphericality is a property on unimodular pivotal tensor 
categories that is invariant under pivotal Morita equivalence, as we show next.

\begin{thm}\label{unimod_dualcat}
Let $\calm$ be a pivotal module category over a $($unimodular$)$ spherical 
tensor category $\cala$. Then the following statements hold:
\Enumeratei
    \item The dual tensor category $\dualcat$ is unimodular.
    \item The pivotal structure \eqref{pivotal_dualcat} induced on the dual tensor category $\dualcat$ is spherical.
\end{enumerate}    
\end{thm}

\begin{proof}
By Proposition \ref{distinguished_dualcat}, the distinguished invertible object 
of $\dualcat$ is $\DD_\Dualcat^{-1} \,{\cong}\,\DD_\cala^{-1}\act({\Se_\Calm})^2$.
Any trivialization $\textbf{u}_\cala\colon\textbf{1}\xcong\DD_\Cala^{-1}$ of 
the distinguished invertible object of $\cala$ furnishes an $\cala$-module 
natural isomorphism $\DD_\Dualcat^{-1} \,{\cong}\,({\Se_\Calm})^2$, where 
$({\Se_\Calm})^2$ is endowed with the module structure
  \be
  \left(\Se_\calm\right)^2(a\act m)\xrightarrow{~\eqref{Serre_twisted}^2~}
a^{\!\vee\vee\vee\vee}\act\left(\Se_\calm\right)^2(m)\xleftarrow{~\;\overline{r}_\Cala\act\id\;~}
  a\act \left(\Se_\calm\right)^2(m)\;.
  \label{eq:module_str1}
  \ee
Now the square of the pivotal structure $\tilde{p}$ of $\calm$ trivializes 
	$({\Se_\Calm})^2$ with module structure
  \be
  \left(\Se_\calm\right)^2(a\act m)\xrightarrow{~\eqref{Serre_twisted}^2~}
  a^{\!\vee\vee\vee\vee}\act\left(\Se_\calm\right)^2(m)\xleftarrow{~\;p\dd{\circ}\,p\act\id\;~}
  a\act \left(\Se_\calm\right)^2(m)\;,
  \label{eq:module_str2}
  \ee
by the defining condition \eqref{condition_pivotal_module} of $\widetilde{p}$. 
Sphericality of $\cala$ ensures that the two module structures \eqref{eq:module_str1} 
and \eqref{eq:module_str2} on $\left(\Se_\calm\right)^2$ coincide. And thus, the 
composition
  \be
  {}{\textbf{\normalsize u}_\dualcat}\Colon
  \id_\calm\xRightarrow{~\;\tilde{p}\cir \tilde{p}\;~}
  \left(\Se_\calm\right)^2\xRightarrow{\;{}^{\textbf{u}_\cala}\act\id\;}
  \DD_\Cala^{-1}\act\left(\Se_\calm\right)^2\xRightarrow{\;\eqref{eq:distinguished_dualcat}\;}\DD_\Dualcat^{-1}
  \label{piv_unimod_dualcat}
  \ee
provides an isomorphism of $\cala$-module functors; this proves (i). In order 
to prove (ii) we need to check the commutativity of the triangle
  \be
  \begin{tikzcd}[row sep=2.5em,column sep=1.7em]
  \id_{\dualcat} \ar[rr,"\overline{r}_{\!\dualcat}",Rightarrow] \ar[rd,"q",swap,Rightarrow]
  &~& (-)^{\rm lllla}
  \\
  ~& (-)\lla \ar[ru,"\!\!q\lla",swap,Rightarrow] &~
  \end{tikzcd}
  \ee
where $q$ is the pivotal structure of the dual tensor category given by the isomorphism
\eqref{pivotal_dualcat}, and $\overline{r}_{\!\dualcat}$ is the canonical isomorphism 
\eqref{Radford_can} of the dual tensor category $\dualcat$. Since 
$\overline{r}_{\!\dualcat}$ does not depend on the chosen trivialization of 
$\DD_\dualcat$, we can equally well use \eqref{piv_unimod_dualcat}. By inserting 
the definitions, it is not hard to see that this diagram reads, for $F\iN\dualcat$,
  \begin{eqnarray}
  \hspace*{-1.1em} \begin{tikzcd}[row sep=2.6em,column sep=2.5em]
  F \ar[r,Rightarrow,"\id\cir \tilde{p}\cir \tilde{p}\,"] \ar[d,swap,Rightarrow,"\id\cir\tilde{p}\,"]
  & F\cir\Se_\calm^\cala\cir\Se_\calm^\cala \ar[r,Rightarrow,"\eqref{Serre_module_functor}~"]
  & \Se_\calm^\cala\cir F\lla\cir\Se_\calm^\cala \ar[r,Rightarrow,"\eqref{Serre_module_functor}~"]
  & \Se_\calm^\cala\cir\Se_\calm^\cala\cir F^{\rm lllla} \ar[r,Rightarrow,"~(\tilde{p}\cir \tilde{p})^{-1}\circ\id~"]
  & F^{\rm lllla}
  \\
  F\cir\Se_\calm^\cala \ar[dr,Rightarrow,"\eqref{Serre_module_functor}",swap]
  \ar[ur,Rightarrow,"\id\cir\tilde{p}",swap] &~&~&~
  & \Se_\calm^\cala\cir F^{\rm lllla} \ar[ul,Rightarrow,"\tilde{p}\cir\id"]
  \ar[u,swap,Rightarrow,"\,\tilde{p}^{-1}\cir\id"]
  \\
  ~& \Se_\calm^\cala\cir F\lla \ar[r,swap,Rightarrow,"\tilde{p}^{-1}\cir\id"]
  \ar[uur,Rightarrow,"\id\cir\tilde{p}"]
  & F\lla \ar[r,swap,Rightarrow,"\id\cir\tilde{p}"]
  & F\lla\cir\Se_\calm^\cala \ar[uul,swap,Rightarrow,"\tilde{p}\cir\id"]
  \ar[ur,Rightarrow,"\eqref{Serre_module_functor}",swap]
  \end{tikzcd} \hspace*{-2.5em}
  \nonumber \\[-2.0em] ~
  \end{eqnarray}
This diagram indeed commutes: the upper-left and upper-right triangles commute trivially; the remaining squares commute due to functoriality of functor composition.
\end{proof}

\begin{cor}\label{spherical_pme_invariant}
Let $\cala$ and $\calb$ be two pivotal Morita equivalent pivotal tensor categories. 
$\cala$ is $($unimodular$)$ spherical if and only if $\calb$ is $($unimodular$)$ spherical.
\end{cor}

We will now explore sphericality for module categories in a similar vein. Let 
${}_\Calc\call_\cald$ be an exact bimodule category over unimodular pivotal 
tensor categories. Associated to $\call$ there is a Radford isomorphism
  \be
  \mathcal{R}_{\!\call}\Colon \DD_{\!\calc}^{-1} \act \Se^\calc_\call(-)
  \xNatiso \Se^{\overline{\cald}}_\call(-)\actr\DD_ {\!\cald}^{-1}\Colon
  \call\longrightarrow{}_{\ldd(-)}\call_{(-)\dd}
  \ee
of twisted bimodule functors (see Equation \eqref{eq:bimod_Radford}) which 
involves the distinguished invertible objects of both $\calc$ and $\cald$. The 
pivotal structures of $\calc$ and $\cald$ turn the functors 
$\DD_{\Calc}^{-1} \act \Se^\calc_\call$ and 
$\Se^{\overline{\cald}}_\call(-)\actr\DD_ {\!\cald}^{-1}$ into bimodule functors 
and $\mathcal{R}_{\!\call}$ becomes an isomorphism of bimodule functors.
Now we pick trivializations
  \be
  \textbf{u}_\calc\Colon\textbf{1}\xcong\DD_\Calc^{-1}
  \qquad\text{ and }\qquad
  \textbf{u}_\cald\Colon\textbf{1}\xcong\DD_\Cald^{-1}
  \ee
which together with $\mathcal{R}_{\!\call}$ furnish an isomorphism
  \be
  \overline{\mathcal{R}}_{\!\call}\Colon  \Se^\calc_\call \;\xNatiso \;\Se^{\overline{\cald}}_\call
  \label{Radford_bimod_triv}
  \ee
of bimodule functors. There are actually some subtleties in the definition of 
\eqref{Radford_bimod_triv}. First of all, unlike $\overline{r}_\Calc^{}$, the 
isomorphism $\overline{\mathcal{R}}_{\!\call}$ is not canonical; it depends on 
the choice of $\textbf{u}_\calc$ and $\textbf{u}_\cald$. Moreover, this is an 
isomorphism of bimodule functors by considering on $\Se^\calc_\call$ the left 
module structure
  \be
  \label{rad_untwist}
  \Se^\calc_\call(c\act x)\xrightarrow{~\eqref{Serre_twisted}~}
  c\dd\act\Se^\calc_\call(x)\xleftarrow{~\;\overline{r}_\Calc\;~}
  \ldd c\act \Se^\calc_\call(x)\xrightarrow{~\;p_{\tiny\ldd c}\;~}
  c\act \Se^\calc_\call(x)\;,
  \ee
and analogous module structure morphisms on $\Se^{\overline{\cald}}_\call$. Now 
a pivotal structure on ${}_\calc\call$ is a trivialization of $\Se^\calc_\call$ 
with module functor structure given by
  \be
  \label{piv_untwist}
  \Se^\calc_\call(c\act x)\xrightarrow{~\eqref{Serre_twisted}~}
  c\dd\act\Se^\calc_\call(x)\xleftarrow{~\;p_{c}\;~}
  c\act \Se^\calc_\call(x)\;.
  \ee
The module structures \eqref{rad_untwist} and \eqref{piv_untwist} on $\Se^\calc_\call$ 
might be different, but they coincide in the particular case that $\calc$ is 
spherical. A similar consideration holds for $\Se^{\overline{\cald}}_\call$ 
under sphericality of $\cald$.
We are now ready to define sphericality for bimodule categories by means of the 
Radford isomorphism $\overline{\mathcal{R}}_{\!\call}$:

\begin{defi}[Spherical bimodule category]\label{spherical_bimod}~\\
Let $\calc$ and $\cald$ be $($unimodular$)$ spherical tensor categories and 
$\textbf{u}_\calc\colon\textbf{1}\xcong\DD_\Calc^{-1}$ and 
$\textbf{u}_\cald\colon\textbf{1}\xcong\DD_\Cald^{-1}$ trivializations of the 
respective distinguished invertible objects.\\
A pivotal bimodule category ${}_\Calc\call_\cald$  is called 
\emph{spherical} iff the diagram
  \be
  \begin{tikzcd}[row sep=2.5em,column sep=1.8em]
  \Se^\calc_\call\ar[rr,"\overline{\mathcal{R}}_{\!\call}",Rightarrow]&~
  &  \Se^{\overline{\cald}}_\call\\
  ~&\id_\call\ar[ul,"\tilde{p}",Rightarrow]\ar[ur,"\tilde{q}",Rightarrow,swap]&~
  \end{tikzcd}
  \ee
commutes, where $\overline{\mathcal{R}}_{\!\call}$ is the composition
  \be
  \Se^\calc_\call \xRightarrow{~\textbf{u}_\calc\act\id~}  \DD_{\!\calc}^{-1}
  \act \Se^\calc_\call \xRightarrow{~\;\mathcal{R}_{\!\call}\;~}
  \Se^{\overline{\cald}}_\call(-)\actr\DD_ {\!\cald}^{-1} 
  \xRightarrow{~\id\actr\textbf{u}_\cald^{-1}~} \Se^{\overline{\cald}}_\call
  \ee
and $\tilde{p}$ and $\tilde{q}$ denote the pivotal structures of  
${}_\calc\call$ and $\call_\cald$, respectively.
\end{defi}

\begin{rem}\label{semi_dist_choices}
The notion of sphericality for bimodule categories given in Definition 
\ref{spherical_bimod} is relative to the choice of isomorphisms 
$\textbf{u}_\calc \colon \textbf{1}\xcong\DD_\Calc^{-1}$ and 
$\textbf{u}_\cald \colon \textbf{1}\xcong\DD_\Cald^{-1}$. However, there is 
a distinguished choice of such trivializations in the semisimple case:

\noindent 
Recall that the finite $\ko$-linear category underlying $\calc$ comes with a 
Nakayama functor \eqref{Nakayama}. For every projective object $p$ in $\calc$ 
there exists a linear map
  \be 
  \mathbf{t}^\calc_p\Colon \Hom_\calc\left(p,\Nak_\calc^r(p)\right) \longrightarrow \ko\label{twist_trace}\,,
  \ee 
called the \textit{$\Nak_\calx^r$-twisted trace} or simply the twisted trace 
of $\calx$ (see \Cite{Def.\ 2.4}{SchW} and \Cite{Def.\ 4.4}{shShi}). 
\noindent 
In case $\calc$ is semisimple, the monoidal unit $\textbf{1}$ is projective and 
consequently there is a twisted trace \eqref{twist_trace} associated to $\textbf{1}$.
Define the \textit{standard trivialization of} $\DD_\calc$ as the isomorphism 
$\textbf{s}_{\tiny\calc} \colon \mathbf{1}\xcong\Nak_\calc^r(\mathbf{1}) \eq
\DD_\calc^{-1}$ for which the composition
  \be
  \Hom_\calc\left(\mathbf{1},\mathbf{1}\right)\xrightarrow{~\,\textbf{s}_{\tiny\calc}
  \cir-\,~}\Hom_\calc\left(\mathbf{1},\Nak_\calc^r(\mathbf{1})\right) 
  \xrightarrow{~\;\eqref{twist_trace}\;~} \ko
  \ee
is equal to the $\ko$-linear map given by the assignment $\id_{\mathbf{1}}\,{\mapsto}\,1$.
\end{rem}

\medskip

We will now define sphericality of a one-sided (left) module category in terms 
of its associated invertible bimodule. Under the hypotheses of Theorem 
\ref{unimod_dualcat}, a choice of isomorphisms 
  \be
  \textbf{u}_\cala\Colon\textbf{1}\xcong\DD_\Cala^{-1}\qquad \text{ and }\qquad 
  \textbf{u}_\Dualcat\Colon\id_\Calm\xnatiso\DD_\Dualcat^{-1}
  \ee
turns the isomorphism \eqref{eq:distinguished_dualcat} into an $\cala$-module isomorphism
  \be
  \overline{r}_\Calm^{}\Colon \id_\calm \xRightarrow{~\textbf{u}_\Dualcat~\,}
  \DD_{\!\dualcat}^{-1} \xRightarrow{~\eqref{eq:distinguished_dualcat}~\,}  
  \DD_\Cala^{-1}\act
  \Se_\calm^\cala \cir \Se_\calm^\cala\xRightarrow{~\textbf{u}_\Cala^{-1}\Act\id~\,}
  \Se_\calm^\cala \cir \Se_\calm^\cala \,,
  \ee
where $\Se_\calm^\cala\cir\Se_\calm^\cala$ is the functor with module structure 
coming from the isomorphism $\overline{r}_\Cala^{}\eq p\dd{\circ}\,p$.

\medskip

\begin{defi}[Spherical module category]\label{spherical_mod}~\\
Let $\calm$ be a pivotal module category over a $($unimodular$)$ spherical 
tensor category $\cala$ and $\textbf{u}_\cala\colon\textbf{1}\xcong\DD_\Cala^{-1}$ 
and $\textbf{u}_\Dualcat\colon\id_\Calm\xnatiso\DD_\Dualcat^{-1}$ be trivializations 
of the distinguished invertible objects.\\
The pivotal module ${}_\cala\calm$ is called \emph{spherical}
iff the pivotal $(\cala,\overline{\dualcat})$-bimodule category $\calm$ is 
spherical, i.e. the diagram 
  \be
  \begin{tikzcd}[row sep=2.5em,column sep=1.8em]
  \Se^\cala_\calm  \ar[rr,"\overline{\mathcal{R}}_{\calm}",Rightarrow]  &~
  & \overline{\Se}^{\cala}_\Calm\ar[dl,"\tilde{p}",Rightarrow] \\
  ~& \id_\calm  \ar[ul,"\tilde{p}",Rightarrow]&~
  \end{tikzcd}
  \label{spherical_module}
  \ee
commutes, where $\widetilde{p}$ is the pivotal structure of $\calm$.
\end{defi}

\begin{rem}
\Enumeratei
\item 
 The commutativity of diagram \eqref{spherical_module} is  equivalent to the commutativity of
\be
  \begin{tikzcd}[row sep=2.5em,column sep=1.8em]
  \id_\Calm^{}  \ar[rr,"\overline{r}_\calm",Rightarrow] \ar[rd,"\tilde{p}",swap,Rightarrow] &~
  & \Se_\calm^\cala \cir\Se_\calm^\cala
  \\
  ~& \Se_\calm^\cala \ar[ru,"\id\cir\tilde{p}",swap,Rightarrow] &~
  \end{tikzcd}
\ee
which is in practice useful for checking sphericality on a pivotal module category.
\item In \cite{Ya} the choice of trivialization $\textbf{u}_\Dualcat$ is termed a \textit{unimodular structure} on $\calm$.
\end{enumerate}
\end{rem}

\begin{prop}\label{Fun_spherical}
Let $\calm$ and $\caln$ be spherical $\cala$-module categories. Then the 
$(\cala_\caln^*,\dualcat)$-bimodule category $\Fun_\cala(\calm,\caln)$ is spherical.
\end{prop}

\begin{proof}
According to Proposition \ref{pivotality_category_functors}(iii) 
$\Fun_\cala(\calm,\caln)$ inherits a pivotal structure as an 
$(\cala_\caln^*,\dualcat)$-bimodule category. Sphericality of 
$\Fun_\cala(\calm,\caln)$ is proved by checking the commutativity of the following diagram
  \be
  \begin{tikzcd}[row sep=2.5em,column sep=1.7em]
  H \ar[rr,"\overline{r}_{\,\Fun_\Cala(\!\calm,\caln)}",Rightarrow] \ar[rd,"\eqref{piv_fun}",swap,Rightarrow]
  &~& \big({\Se^{\overline{\dualcat}}_{\Fun_\Cala(\calm,\caln)}}\big)^2(H)
  \\
  ~& \Se^{\overline{\dualcat}}_{\Fun_\Cala(\calm,\caln)}(H)
  \ar[ru,"\!\!\eqref{piv_fun}\lla",swap,Rightarrow] &~
  \end{tikzcd}
  \ee
for every module functor $H\colon\calm\to\caln$. Considering 
\eqref{eq:Serre_bimodFun}, this diagram turns explicitly into
  \be\label{diagram_Fun}
  \begin{tikzcd}[row sep=2.5em,column sep=1.7em]
  H \ar[rr,"\overline{r}_{\!\caln}\cir\id",Rightarrow] \ar[dd,"\hat{p}\cir\id\,",swap,Rightarrow]
  &&\Se_\caln^\cala\cir\Se_\caln^\cala\cir H \ar[rr,"\id\cir
  \;\overline{r}^{-1}_{\!\calm}",Rightarrow]\ar[ddrr,Rightarrow,,"\id
  \cir\tilde{q}"]&& \Se_\caln^\cala\cir\Se_\caln^\cala\cir H
  \cir\lSe{}_\calm^\cala\cir\lSe{}_\calm^\cala
  \\
  \\
  \Se_\caln^\cala\cir H\ar[urur,Rightarrow,"\hat{p}\cir \id"]\ar[rr,Rightarrow,"\id
  \cir\tilde{q}",swap]&& \Se_\caln^\cala\cir H\cir\lSe{}_\calm^\cala
  \ar[rr,swap,Rightarrow,"\hat{p}\cir\id"] &&\Se_\caln^\cala\cir\Se_\caln^\cala
  \cir H\cir\lSe{}_\calm^\cala\ar[uu,Rightarrow,swap,"\;\id\cir\tilde{q}"]
  \end{tikzcd}
  \ee
where $\widehat{p}\colon \id_\caln \,{\xNatiso}\,\Se_\caln^{\cala}$ is the pivotal 
structure of ${}_\cala\caln$ and 
$\widetilde{q}\colon \id_\calm \,{\xNatiso}\, \lSe{}_\calm^\cala\cong\Se_\calm^{\Dualcat}$ 
is the pivotal structure coming from \eqref{Mpivotal_dual}. The commutativity of 
the left triangle in diagram \eqref{diagram_Fun} is precisely the sphericality 
condition of the pivotal module ${}_\cala\caln$, while the sphericality of 
${}_{\cala}\calm$ implies the commutativity of the right triangle. The square 
in the middle commutes due to functoriality of composition.
\end{proof}

\begin{rem}
Given a (unimodular) spherical tensor category $\cala$ and a spherical module 
category ${}_\cala\calm$, Proposition \ref{Fun_spherical} implies that the 
$(\overline{\dualcat},\cala)$-bimodule category $\FunM$ is spherical. By Theorem 
\ref{unimod_dualcat}, the dual tensor category $\dualcat$ is spherical as well.  
And thus, every Hom-category of the pivotal bicategory $\Mor$ associated to 
${}_\cala\calm$ is spherical.

\noindent The sphericality condition on a pivotal module category can then be interpreted in the 
bicategorical setting as the statement that the pivotal structure of $\Mor$ squares
to the Radford pseudo-natural equivalence \eqref{Radford_pseudo}. 
Beyond unimodularity, 
the non-triviality of the distinguished invertible objects in \eqref{Radford_pseudo}
suggests that such a bicategorical interpretation could be studied in the framework
of quasi-pivotal structures, as discussed in Remark \ref{quasi_pivotal}.
\end{rem}

\begin{defi}\label{def:sphMorita}
A pivotal Morita context $(\cala,\calb,\calm,\caln,\mixto{}{},\mixtdo{}{})$ is said to be 
\emph{spherical} iff the Radford pseudo-equivalence \eqref{Radford_pseudo} of its 
associated bicategory $\Mor$ is equivalent to the square of the pivotal structure of $\Mor$.
\end{defi}
}


\section{Dualities and pivotality for equivariant Morita Theory}\label{sec:equivariant}

In \cite{GJS} the notion of categorical Morita equivalence has been extended to 
finite tensor categories graded by a finite group $G$. The intention of the present 
section is to study the interaction of pivotal structures and Morita theory in the 
equivariant setting. {\Color We now assume that the base field \ko\ is of characteristic zero and, as before, that it is algebraically closed.}

\subsection{Graded module categories and graded Morita theory}

Let $G$ be a finite group. A \emph{$G$-grading} on a tensor category $\cala$ 
consists of a decomposition 
  \be
  \cala = \Gsum \,\cala_g
  \ee
into a direct sum of full abelian subcategories such that for $g,h\iN G$ the 
tensor product restricts to a bifunctor 
$\otimes\colon \cala_g\,{\times}\,\cala_h\Rarr{}\cala_{gh}$. If $\cala_g\,{\ne}\,0$ 
for every $g\iN G$ the $G$-grading is called \emph{faithful}. We will only consider
faithful gradings. A \emph{$G$-graded module category over $\cala$} is an 
$\cala$-module category with a decomposition
  \be
  \calm = \Gsum \calm_g
  \label{eq:Gsum-calm}
  \ee
into a direct sum of full abelian subcategories, with $\calm_g\,{\ne}\,0$ for 
every $g\iN G$, and such that for $g,h\iN G$ the $\cala$-action restricts to 
$\act\colon \cala_g\,{\times}\,\calm_h\Rarr{}\calm_{gh}$. 
Given $G$-graded module categories $\calm$ and $\caln$, the category of module 
functors decomposes as \Cite{Prop.\,4.8}{GJS} 
  \be
  \Fun_\cala(\calm,\caln)=\Gsum \Fun_\cala(\calm,\caln)_g \,,
  \ee
where $\Fun_\cala(\calm,\caln)_g$ is the category of \emph{homogeneous module functors}
of degree $g$, i.e.\ module functors  $H\colon \calm\Rarr{}\caln$ satisfying
$H(\calm_x) \,{\subseteq}\, \caln_{xg}$ for every $x\iN G$. A \emph{grading preserving}
module functor is a homogeneous module functor of trivial degree. The decomposition
\eqref{eq:Gsum-calm} turns $\overline{\dualcat}$ into a $G$-graded tensor category and
$\Fun_\cala(\calm,\caln)$ into a $G$-graded $\overline{\dualcat}$-module category.

\begin{defi}\Cite{Def.\,4.10}{GJS}\label{GME}
Two $G$-graded tensor categories $\cala$ and $\calb$ are said to be \emph{graded Morita 
equivalent} iff there exists a $G$-graded $\cala$-module category $\calm$ together with 
a $G$-gra\-ded tensor equivalence $\calb\,{\simeq}\,\overline{\dualcat}$.
\end{defi}

One expects that the equivalence data for the notion of graded Morita equivalence are
endowed with a graded structure in a compatible manner. That is, the categories in 
the Morita context of a $G$-graded module category should be graded and be related 
via grading preserving actions. Indeed we have

\begin{prop}\label{graded_strong_Morita_context}
Let $\cala$ be a $G$-graded tensor category and $\calm$ a $G$-graded $\cala$-module 
category and consider $(\cala,\,\overline{\dualcat},\,\calm,\,\FunM,\,\mixto{}{},\,\mixtdo{}{})$
the Morita context associated to it.
\begin{enumerate}[\rm (i)]
    \item 
$\overline{\dualcat}$ is a $G$-graded tensor category.
    \item 
$\calm$ and $\FunM$ are $G$-graded bimodule categories.
    \item 
The mixed products $\mixt{}{}$ and $\mixtd{}{}$ are compatible
with the group law of $G$.
\end{enumerate}    
\end{prop}

\begin{proof}
That $\overline{\dualcat}$ is a $G$-graded 
tensor category and $\FunM$ is a $G$-graded $\overline{\dualcat}$-module is part
of the statement in Proposition 4.8 of \cite{GJS}. The remaining assertions follow
from the fact that, by the definition of the gradings, we have
  \be
  \begin{aligned}
  m\actr F = F(m) \,\in\calm_{gy} \,, \qquad 
  & H\actr a = H(-)\otimes a \,\in\FunM_{xh} \,,
  \Nxl3
  \mixt{m}{H} = H(m) \,\in\calm_{gx} \,, \qquad
  & \mixtd{H}{m} = H(-)\act m \, \in (\overline{\dualcat})_{xg}
  \end{aligned}
  \ee
for all $m\iN\calm_g$, $a\iN\cala_h$, $H\iN\FunM_x$ and $F\iN(\overline{\dualcat})_y$.
\end{proof}

\begin{rem}
Proposition \ref{graded_strong_Morita_context} can be seen as a statement about
the bicategory $\Mor$ associated to the Morita context of an exact $G$-graded module 
category: $\Mor$ is a bicategory enriched in the (non-symmetric) monoidal $2$-category of $G$-graded 
linear abelian categories and grading preserving functors.
\end{rem}

A feature of a $G$-graded tensor category is that the duals of a homogeneous object
are again homogeneous, with inverse degree. This holds for the Morita context of
an exact $G$-graded module category as well:

\begin{prop}\label{graded_duals}
Let $\cala$ be a $G$-graded finite tensor category and $\calm$ an exact $G$-graded 
$\cala$-mo\-du\-le category.
\begin{enumerate}[\rm (i)]
    \item  
For a homogeneous object $m\iN\calm_g$, the duals $m^\vee$ and $\Vee m$ are
homogeneous of degree $g^{-1}$ in $\FunM$.
    \item 
For a homogeneous object $H\iN\FunM_g$, the duals $H^\vee$ and $\Vee H$ are
homogeneous of degree $g^{-1}$ in $\calm$.
    \item 
The relative Serre functors of $\calm$ and of $\FunM$ are grading preserving.
\end{enumerate}
\end{prop}

\begin{proof}
According to \Cite{Prop.\,4.2}{GJS} we have $\iHomM^\cala(m,n)\iN\cala_{hg^{-1}}$
for $m\iN\calm_g$ and $n\iN\calm_h$. This also implies that 
$\icoHom_\calm^\cala(m,n) \,{\cong}\, \Vee\iHomM^\cala(n,m)\iN\cala_{hg^{-1}}$,
which proves (i). 
 \\[2pt]
To show (ii), notice that the adjoints of a module functor
$H\iN\FunM_g$ are homogeneous module functors in $\Fun_\cala(\cala,\calm)_{g^{-1}}$. 
Since $\textbf{1}\iN\cala_e$, it follows that 
$H\ra(\textbf{1}),H\la(\textbf{1})\iN\calm_{g^{-1}}$.
 \\[2pt]
Assertion (iii) follows from (i) and (ii) together with Proposition \ref{double_duals}.
\end{proof}


\subsection{De-equivariantization and the equivariant center}

For a $G$-graded finite tensor category $\cala$ there is a construction called 
the \emph{equivariant center} \Cite{Sect.\,3}{GNN} which assigns a braided 
$G$-crossed tensor category $\calz_G(\cala)$ to $\cala$. This is an instance of a 
procedure known as de-equivariantization, which is well studied \Cite{Sect.\,4.4}{DGNO}:
Let $\cald$ be a braided tensor category together with a braided fully faithful functor
$\text{Rep}(G)\Rarr{}\cald$. The group $G$ acts by left translations on the set 
$\text{Fun}(G,\ko)$ of functions, thereby turning it into an object in $\text{Rep}(G)$.
Moreover, this object has a canonical structure of a commutative special Frobenius 
algebra in $\text{Rep}(G)$. Denote by $L$ the image in $\cald$ of $\text{Fun}(G,\ko)$ 
under the functor $\text{Rep}(G)\Rarr{}\cald$.
The \emph{de-equivariantization of $\cald$} is a braided $G$-crossed tensor category
$\cald_G$ whose underlying tensor category is the category of modules 
${}_L\Mod(\cald)$ with tensor product $\otimes_L$.

Now the Drinfeld center of a $G$-graded finite tensor category $\cala$ is a 
braided tensor category endowed with a fully faithful braided functor
  \be
  \text{Rep}(G) \rarr~ \calz(\cala)\,, \qquad
  (\ko^n,\rho) \xmapsto{~} (\mathbf{1}^n,\gamma_{-,\mathbf{1}^n}) \,,
\ee
where for $a\iN\cala_g$ the half-braiding is defined via
$\gamma_{a,\mathbf{1}^n_{}} \colon a\ot \mathbf{1}^n 
\,{\xrightarrow{\;\id_a\otimes \rho(g)\;}}\, a\ot \mathbf{1}^n \eq\mathbf{1}^n\ot a$.
The \emph{equivariant center} $\calz_G(\cala)$ can be characterized as the
de-equivariantization of $\calz(\cala)$, i.e.\ $\calz_G(\cala)={}_L\Mod(\calz(\cala))$.

\begin{prop}{\rm \Cite{Prop.\,4.20}{GJS}}
Let $\calm$ be an exact $G$-graded $\cala$-module category. The braided equivalence 
$\calz(\cala)\,{\simeq}\, \calz(\overline{\dualcat})$ given by \eqref{schauenburg}
induces, via de-equivariantization, an equivalence
  \be
  \label{G_schauenburg}
  \calz_G(\cala) \xsimeq\, \calz_G(\,\overline{\dualcat}\,)
  \ee
of braided $G$-crossed tensor categories.
\end{prop}


\subsection{Pivotality in the equivariant picture}

We consider now a $G$-graded tensor category $\cala$ endowed with a pivotal structure.

\begin{defi}
An exact $G$-graded module $\calm$ over a $G$-graded pivotal category $\cala$ is 
said to be \emph{pivotal} iff the underlying module category has the structure of
a pivotal module.
\end{defi}

\begin{prop}
Let $\calm$ be an exact $G$-graded module category over a $G$-graded pivotal 
tensor category $\cala$. A pivotal structure on ${}_\cala\calm$ is the same as 
a pivotal structure on ${}_{\cala_e}\calm_e$.
\end{prop}

\begin{proof}
First notice that for $m,n\iN\calm_e$ we have $\iHomM^\cala(m,n) 
\eq \iHom_{\calm_e}^{\cala_e}(m,n)$, and thus the restriction of the relative Serre 
functor obeys $\Se_\calm^\cala|_{\calm_e}^{} \,{=}\; \Se_{\calm_e}^{\cala_e}$. The 
pivotal structure of $\cala$ turns both $\Se_\calm^\cala$ and $\Se_{\calm_e}^{\cala_e}$ 
into module functors. According to Proposition \ref{graded_duals}, relative Serre 
functors are grading preserving, and thus a pivotal structure on ${}_\cala\calm$ is 
an isomorphism $\id_\calm \Rarr\cong \Se_\calm^\cala$ in $(\dualcat)_e^{}$. Now 
Theorem 3.3 of \cite{G} implies that restriction induces an equivalence
  \be
  \left(\dualcat\right)_e = \Fun_\cala(\calm,\calm)_e
  \simeq \Fun_{\cala_e}(\calm_e,\calm_e) = (\cala_e)_{\calm_e}^{\,*}
  \ee
under which a module natural isomorphism $\id_\calm\Rarr{\cong}\Se_\calm^\cala$ 
corresponds to a module natural isomorphism
$\id_{\calm_e}\Rarr{\cong}\Se_{\calm_e}^{\cala_e}$.
\end{proof}

\begin{defi}
Two $G$-graded pivotal categories $\cala$ and $\calb$ are said to be 
\emph{graded pivotal Morita equivalent} iff there exists a $G$-graded pivotal 
$\cala$-module category $\calm$ together with 
a $G$-graded pivotal equivalence $\calb\,{\simeq}\,\overline{\dualcat}$.
\end{defi}

The Drinfeld center of a pivotal tensor category inherits a pivotal structure 
\Cite{Ex.\,7.13.6}{EGno}. In particular for a $G$-graded pivotal tensor category 
$\cala$, the Drinfeld center $\calz(\cala)$ has a canonical pivotal structure 
$p\colon \id_{\calz(\cala)}\Rarr{\cong}(-)\dd$. According to \Cite{Thm.\ 1.17}{KO},
$p$ serves as pivotal structure for ${}_L\Mod(\calz(\cala)) \eq \calz_G(\cala)$.

\begin{prop}
Let $\calm$ be a $G$-graded pivotal $\cala$-module category. The braided $G$-crossed
equivalence \eqref{G_schauenburg} is pivotal.
\end{prop}

\begin{proof}
By Proposition \ref{schauenburg_pivotal} the equivalence \eqref{schauenburg} is pivotal.
Further, \eqref{G_schauenburg} is induced by \eqref{schauenburg} and thus preserves 
$p$ as well.
\end{proof}

\begin{cor}
If two $G$-graded pivotal categories $\cala$ and $\calb$ are graded pivotal Morita 
equivalent, then their equivariant centers $\calz_G(\cala)$ and $\calz_G(\calb)$ 
are equivalent as pivotal braided $G$-crossed tensor categories.
\end{cor}


\vskip 3.5em

\noindent
{\sc Acknowledgments:}\\[.3em]
We thank David Reutter and Harshit Yadav for helpful comments.
JF is supported by VR under project no.\ 2017-03836. DJ and CS are partially
supported by DFG under Germany's Excellence Strategy - EXC 2121 ``Quantum Universe''
- 390833306. CG is partially supported by the Research Fund of the School
of Sciences at the Universidad de los Andes under project INV-2023-162-2830.

\newpage
\newcommand\wb{\,\linebreak[0]} \def\wB {$\,$\wb}
\newcommand\Arep[2]  {{\em #2}, available at {\tt #1}}
\newcommand\Bi[2]    {\bibitem[#2]{#1}}
\newcommand\inBO[9]  {{\em #9}, in:\ {\em #1}, {#2}\ ({#3}, {#4} {#5}), p.\ {#6--#7} {\tt [#8]}}
\newcommand\J[7]     {{\em #7}, {#1} {#2} ({#3}) {#4--#5} {{\tt [#6]}}}
\newcommand\JJ[7]     {{\em #7}, {#1}{#2} ({#3}) {#4#5}{{\tt [#6]}}}
\newcommand\JO[6]    {{\em #6}, {#1} {#2} ({#3}) {#4--#5} }
\newcommand\JP[7]    {{\em #7}, {#1} ({#3}) {{\tt [#6]}}}
\newcommand\Jpress[7]{{\em #7}, {#1} {} (in press) {} {{\tt [#6]}}}
\newcommand\Jtoa[7]  {{\em #7}, {#1} {} (to appear) {} {{\tt [#6]}}}
\newcommand\BOOK[4]  {{\em #1\/} ({#2}, {#3} {#4})}
\newcommand\PhD[2]   {{\em #2}, Ph.D.\ thesis #1}
\newcommand\Prep[2]  {{\em #2}, preprint {\tt #1}}
\newcommand\uPrep[2] {{\em #2}, unpublished preprint {\tt #1}}
\def\adma  {Adv.\wb Math.}
\def\alrt  {Algebr.\wb Represent.\wB Theory}   
\def\anma  {Ann.\wb Math.}
\def\anop  {Ann.\wb Phys.}
\def\apcs  {Applied\wB Cate\-go\-rical\wB Struc\-tures}
\def\atmp  {Adv.\wb Theor.\wb Math.\wb Phys.}   
\def\bbms   {Bull.\wb Belg.\wb Math.\wb Soc.\wb Simon Stevin}
\def\cocm  {Com\-mun.\wb Con\-temp.\wb Math.}
\def\comp  {Com\-mun.\wb Math.\wb Phys.}
\def\coma  {Con\-temp.\wb Math.}
\def\duke  {Duke\wB Math.\wb J.}
\def\geat  {Geom.\wB and\wB Topol.}
\def\imrn  {Int.\wb Math.\wb Res.\wb Notices}
\def\jims  {J.\wb Indian\wb Math.\wb Soc.}
\def\jams  {J.\wb Amer.\wb Math.\wb Soc.}
\def\joal  {J.\wB Al\-ge\-bra}
\def\joms  {J.\wb Math.\wb Sci.}
\def\jktr  {J.\wB Knot\wB Theory\wB and\wB its\wB Ramif.}
\def\jpaa  {J.\wB Pure\wB Appl.\wb Alg.}
\def\kyjm  {Ky\-o\-to J.\ Math.}
\def\mama  {ma\-nu\-scripta\wB mathematica\wb}
\def\mams  {Memoirs\wB Amer.\wb Math.\wb Soc.}
\def\momj  {Mos\-cow\wB Math.\wb J.}
\def\nupb  {Nucl.\wb Phys.\ B}
\def\pams  {Proc.\wb Amer.\wb Math.\wb Soc.}
\def\pspm  {Proc.\wb Symp.\wB Pure\wB Math.}
\def\quto  {Quantum Topology}
\def\sema  {Selecta\wB Mathematica}
\def\sigm  {SIGMA}
\def\slnm  {Sprin\-ger\wB Lecture\wB Notes\wB in\wB Mathematics}
\def\taac  {Theo\-ry\wB and\wB Appl.\wb Cat.}
\def\tams  {Trans.\wb Amer.\wb Math.\wb Soc.}
\def\toap  {Topology\wB Applic.}
\def\topo  {Topology}
\def\trgr  {Trans\-form.\wB Groups}
\def\alnt  {Algebra\wB\&\wB Number\wB Theory}
\def\jlms  {J.\wB London\wB Math.\wb Soc.}
\def\jhep  {J.\wb High\wB Energy\wB Phys.}

 \end{document}